\documentclass[11pt]{amsart}

\title[Intrinsic characterization of Riemannian submanifolds]
{Intrinsic Characterization of $3$-dimensional Riemannian submanifolds of $\mathbb{R}^4$} 

\author[Y.\ Agaoka]{Yoshio Agaoka}
\author[T.\ Hashinaga]{Takahiro Hashinaga}
\address[Y.\ Agaoka]{Department of Mathematics, Hiroshima University, 
Higashi-Hiroshima 739-8521, Japan}
\address[T.\ Hashinaga]{Faculty of Education, Saga University, Saga 840-8502, Japan}

\email[Y.\ Agaoka]{agaoka@hiroshima-u.ac.jp}
\email[T.\ Hashinaga]{hashinag@cc.saga-u.ac.jp}

\thanks{The first author was supported by JSPS KAKENHI Grant Number 16K05132, 20K03589 
and the second author was supported by JSPS KAKENHI Grant Number 16K17603.}

\usepackage{a4}
\usepackage{amsmath,amssymb,amscd,amsthm,cases}
\usepackage{amsfonts}
\usepackage[dvips]{graphicx, color}

\newtheorem{thm}{Theorem}[section]
\newtheorem{prop}[thm]{Proposition}
\newtheorem{lem}[thm]{Lemma}
\newtheorem{cor}[thm]{Corollary}

\theoremstyle{definition}

\newtheorem{rem}[thm]{Remark}
\newtheorem{exam}[thm]{Example}

\numberwithin{equation}{section}

\begin{document}
\begin{abstract}
It is well known that an $m$-dimensional Riemannian manifold can be locally isometrically embedded into the $m+1$-dimensional Euclidean space if and only if there exists a symmetric 2-tensor field satisfying the Gauss and Codazzi equations. 
In this paper, we prove that two known intrinsic conditions, which were obtained previously by Weiss, Thomas and Rivertz, are sufficient to ensure the existence of such symmetric 2-tensor field under certain generic condition when $m=3$. 
Note that, in the case $m=3$, a symmetric 2-tensor field satisfying the Gauss equation does not satisfy the Codazzi equation automatically, which is different from the cases $m \geq 4$.  
In our proof, the symbolic method, which is a famous 
tool known in classical invariant theory, plays an important role.
As applications of our result, we consider $3$-dimensional warped product 
Riemannian manifolds whether they can be locally isometrically embedded into 
$\mathbb{R}^4$.
In some case, the Monge-Amp\`{e}re equation naturally appears.
\end{abstract}

\maketitle

\section{Introduction}

Let $(M^m,g)$ be an $m$-dimensional Riemannian manifold. 
A smooth mapping $f$ of $(M^m,g)$ into the Euclidean space $\mathbb{R}^{m+k}$ is 
called an \textit{isometric immersion} if the Riemannian metric $g$ is equal to the pullback of 
the standard Euclidean metric by $f$.  
Namely the differential equation
\begin{align*}
\langle df, df \rangle_{\mathbb{R}^{m+k}} =g
\end{align*} 
holds, where $\langle \; , \; \rangle_{\mathbb{R}^{m+k}}$ is the standard inner product 
of $\mathbb{R}^{m+k}$.
In case $f$ is an embedding, we call it an \textit{isometric embedding}.

As is well known, any $m$-dimensional Riemannian manifold $(M^m,g)$ can be locally or  
globally isometrically embedded into a Euclidean space with sufficiently large dimension 
\cite{J2}, \cite{Car}, \cite{N}.
But generic Riemannian manifolds $(M^m,g)$ cannot be isometrically embedded 
into a low dimensional Euclidean space even locally.
If $(M^m,g)$ can be locally isometrically embedded into a low dimensional space, we may 
say that the Riemannian manifold is specific in some sense, and it is a natural problem to 
characterize such a Riemannian manifold in an intrinsic manner.

In the case of hypersurfaces, i.e., the case $k=1$, the fundamental theorem of hypersurfaces 
is well known.
It states as follows:
an $m$-dimensional Riemannian manifold $M^m$ can be locally isometrically 
embedded into $\mathbb{R}^{m+1}$ if and only if there exists a symmetric $2$-tensor field 
$\alpha$ that satisfies the Gauss and Codazzi equations.
(For details, see \S 2.)
But these equations are a system of quadratic equations and partial differential equations on 
the extrinsic quantity $\alpha$ respectively, and it is in general hard to 
determine whether $(M^m,g)$ admits such a solution or not.
 
Concerning this problem, for the case of $m \geq 4$,  $\lq\lq$intrinsic characterization'' 
of hypersurfaces of $\mathbb{R}^{m+1}$ is settled by the classical work of Thomas 
\cite{T} from a different viewpoint.
In the case $m \geq 4$, under some generic conditions on the curvature, Thomas showed that 
the Codazzi equation automatically follows from the Gauss equation, and hence 
the solvability of the Gauss equation ensures the existence of local isometric embeddings.
In addition, we can theoretically determine whether the Gauss equation admits 
a solution or not in codimension $1$ case, and we may say that the problem is settled 
in the case $m \geq 4$.
(For details on this subject, see Remark~\ref{Rem_const} (3).)

However, the situation is completely different in the case $m = 3$.
In this case, even if the Gauss equation admits a solution, it does not satisfy the 
Codazzi equation in general, and hence we must further investigate the Codazzi equation.
But the Codazzi equation is hard to treat in an intrinsic manner.
Maybe by this reason, there is no essential progress concerning this problem after Thomas, 
though $\lq\lq$local isometric embedding'' is one of the fundamental theme in classical 
differential geometry.
 
In this paper we consider the intrinsic characterization of hypersurfaces $M^3$ in 
$\mathbb{R}^4$, and show that two known obstructions 
explained below give a sufficient condition so that $M^3$ can be locally 
isometrically embedded into $\mathbb{R}^4$ under some generic condition.

Before explaining these two obstructions and our main theorem, we here state one 
important remark. 
Throughout this paper, we consider the problem of isometric embeddings 
from the $\lq\lq$local'' standpoint.
We carry out several calculations always restricted to a coordinate neighborhood around a point 
of $M^m$, and express this neighborhood by $M^m$ again by the abuse use of notations.
In addition, we always assume that this neighborhood is simply connected.
However, if necessary, we can easily re-state our results to a global fashion by adding 
some topological conditions.
By the same reason, the mapping $f$ which we consider in this paper is always assumed 
to be an embedding unless otherwise stated.

Now we state two obstructions to the existence of a local isometric embedding 
$f: (M^3,g) \longrightarrow \mathbb{R}^4$.
Weiss \cite{W} and Thomas \cite{T} proved that if the Gauss equation of $M^3$ admits a 
solution in codimension $1$, then the curvature tensor $R_{ijkl}$ must satisfy the 
following inequality 
\begin{align*}
|\widetilde{R}|:=\begin{vmatrix}
R_{1212} & R_{1213} & R_{1223}\\
R_{1213} & R_{1313} & R_{1323}\\
R_{1223} & R_{1323} & R_{2323}
\end{vmatrix} \geq 0
\end{align*}
at each point of $M^3$.
This inequality is the first obstruction.
(See also \cite{Kawa}.)
The second obstruction was found by Rivertz \cite{R1} in 1999.
If $M^3$ is isometrically embedded into $\mathbb{R}^4$, then the curvature $R_{ijkl}$ and
its covariant derivative $S_{ijklm} = \nabla_m R_{ijkl}$ must satisfy the following six equalities 
(see \S 2 for the precise notations):
 
\medskip

\begin{align*}
& r_1=(R_{1213}R_{2323} - R_{1223}R_{1323})S_{12121}
- (R_{1213}R_{1323} - R_{1223}R_{1313})S_{12122} \\
& - (R_{1212}R_{2323}- R_{1223}{}^2)S_{12131}
+ (R_{1212}R_{1323} - R_{1213}R_{1223})S_{12132} \\
& + (R_{1212}R_{1323} - R_{1213}R_{1223})S_{12231}
- (R_{1212}R_{1313} - R_{1213}{}^2)S_{12232} =0, \\
& \rule{0cm}{0.5cm} r_2=(R_{1313}R_{2323} - R_{1323}{}^2)S_{12121}
- 2(R_{1213}R_{1323}- R_{1223}R_{1313})S_{12123} \\
& - (R_{1212}R_{1313} - R_{1213}{}^2)S_{12233}
- (R_{1212}R_{2323}- R_{1223}{}^2)S_{13131} \\
& + 2(R_{1212}R_{1323} - R_{1213}R_{1223})S_{13132} 
- (R_{1212}R_{1313} - R_{1213}{}^2)S_{13232}=0, \\
& \rule{0cm}{0.5cm} r_3= (R_{1313}R_{2323} - R_{1323}{}^2)S_{12122}
- 2(R_{1213}R_{2323}- R_{1223}R_{1323})S_{12123} \\
& + (R_{1212}R_{2323} - R_{1223}{}^2)S_{12133}
- (R_{1212}R_{2323}- R_{1223}{}^2)S_{13231} \\
& + 2(R_{1212}R_{1323} - R_{1213}R_{1223})S_{23231}
- (R_{1212}R_{1313}- R_{1213}{}^2)S_{23232} =0, \\
& \rule{0cm}{0.5cm} r_4= (R_{1313}R_{2323} - R_{1323}{}^2)S_{12131}
- (R_{1213}R_{1323}- R_{1223}R_{1313})S_{12133} \\
& - (R_{1213}R_{2323}- R_{1223}R_{1323})S_{13131}
+ (R_{1212}R_{1323} - R_{1213}R_{1223})S_{13133} \\
& + (R_{1213}R_{1323} - R_{1223}R_{1313})S_{13231}
- (R_{1212}R_{1313} - R_{1213}{}^2)S_{13233} =0, \\
& \rule{0cm}{0.5cm} r_5= (R_{1313}R_{2323} - R_{1323}{}^2)S_{12132}
+ (R_{1313}R_{2323} - R_{1323}{}^2)S_{12231} \\
& - 2(R_{1213}R_{2323} - R_{1223}R_{1323})S_{13132}
+ (R_{1212}R_{2323} - R_{1223}{}^2)S_{13133} \\
&+ 2(R_{1213}R_{1323} - R_{1223}R_{1313})S_{23231}
- (R_{1212}R_{1313}  - R_{1213}{}^2)S_{23233} =0, \\
& \rule{0cm}{0.5cm} r_6= (R_{1313}R_{2323} - R_{1323}{}^2)S_{12232}
- (R_{1213}R_{2323}- R_{1223}R_{1323})S_{12233} \\
& - (R_{1213}R_{2323} - R_{1223}R_{1323})S_{13232}
+ (R_{1212}R_{2323} - R_{1223}{}^2)S_{13233} \\
& + (R_{1213}R_{1323} - R_{1223}R_{1313})S_{23232}
- (R_{1212}R_{1323} - R_{1213}R_{1223})S_{23233} =0.
\end{align*}

\medskip

\noindent
(Actually, among the above six equalities, only the first one $r_1=0$ is presented in p.99 
of  \cite{R1}.
But $r_1$ generates a $6$-dimensional $GL(3,\mathbb{R})$-irreducible representation space 
of the polynomials of curvature and its covariant derivative, and the set of six polynomials 
$\{r_1, \cdots, r_6\}$ gives a basis of this irreducible space.
See Remark~\ref{Rem_rep} (2).)

\bigskip

Then our main theorem may be stated as follows.

\begin{thm}\label{MainTh}
Let $(M^3,g)$ be a $3$-dimensional simply connected Riemannian manifold.
Assume that $|\widetilde{R}| \neq 0$ everywhere.
Then, there exists an isometric embedding $f: (M^3,g) \longrightarrow \mathbb{R}^4$ if and 
only if $|\widetilde{R}| > 0$ and Rivertz's six equalities $r_1= \cdots =r_6=0$ hold.
\end{thm}

\noindent
Clearly, $|\widetilde{R}| \neq 0$ is a generic condition on $M^3$.
This theorem gives a local intrinsic characterization of a generic 
Riemannian submanifold $M^3$ in $\mathbb{R}^4$.
Roughly speaking, this theorem is proved in the following way.
From the condition $|\widetilde{R}|>0$, it follows that the Gauss equation admits a 
unique solution up to sign \cite{W}, \cite{T}, and by using Rivertz's six equalities 
$r_1=\cdots=r_6=0$, we can prove that the solution of the Gauss equation satisfies the 
Codazzi equation, which is the essential part of the present paper.
In this way, combined with the fundamental theorem of hypersurfaces,  the above theorem 
is proved.

However, Rivertz's six equalities are too lengthy, and it is hard to treat these expressions 
directly.
To overcome this difficulty, we introduce an old tool known as the $\lq\lq$symbolic method" in 
the field of classical invariant theory, and we use it in a renewal form fitted for our purpose.
(For classical invariant theory, see for example \cite{Dol}, \cite{Gu}, \cite{O}.)
By using this renewal tool,  surprisingly, we can express Rivertz's long expressions in a quite 
compact factorized form, and we can complete the proof of Theorem~\ref{MainTh} by 
applying only elementary linear algebra.
The symbolic method must have other good applications in geometry, where 
tremendously long expressions necessarily appear.

By the above theorem, the characterization problem of local isometric 
embeddings for hypersurfaces 
by intrinsic invariants is completely settled for the case $m =3$ under the generic condition 
$|\widetilde{R}| \neq 0$.
We note that to prove the sufficiency, we must exclude the singular points 
where $|\widetilde{R}|=0$, 
and these cases seem to be hard to treat in our general setting.
For such explicit examples, see \S 5.

As an application of Theorem~\ref{MainTh}, we consider local isometric embeddings 
of warped product metrics.
In $3$-dimensional case, there are two types of warped product metric 
\begin{align*}
& ds^2 = dx_1{}^2 + f(x_1)^2 (E dx_2{}^2+2Fdx_2dx_3+Gdx_3{}^2), \\
& ds^2 = E dx_1{}^2+2Fdx_1dx_2+Gdx_2{}^2+ f(x_1,x_2)^2 dx_3{}^2.
\end{align*}
In the former case $E,F,G$ are functions of $x_2$, $x_3$, and in the latter case, they are 
functions of $x_1$, $x_2$.
We assume that the warping function $f$'s are positive for both cases.

Roughly speaking, under the generic condition $|\widetilde{R}| \neq 0$, the former metric 
can be locally isometrically embedded into $\mathbb{R}^4$ if and only if the 
$2$-dimensional fiber is a space of positive constant curvature, satisfying the inequality 
$K > f^{\prime}{}^2$, where $K$ is the Gaussian curvature of the fiber 
(Theorem~\ref{Appl1}).
For the latter metric, it can be locally isometrically embedded into $\mathbb{R}^4$ if 
and only if the warping function $f(x_1,x_2)$ satisfies a certain Monge-Amp\`{e}re 
differential equation and some inequality on the Hessian of $f$ 
(Theorem~\ref{Appl3}).

As another application, we consider the classification of left-invariant 
Riemannian metrics on $3$-dimensional Lie groups that 
can be locally isometrically embedded into $\mathbb{R}^4$.
This problem is completely settled in our 
previous paper \cite{AH}.
We re-examine this result in the last part of this paper from a different standpoint, base 
on our main theorem.

\medskip

Now we explain the contents of this paper.
In \S 2, we review some fundamental facts on hypersurfaces of the space $\mathbb{R}^{m+1}$ 
such as the fundamental theorem of hypersurfaces and the classical results of Weise, Thomas 
and Rivertz.
We prove their results in \S 3 for the sake of completeness.
Especially, in proving Rivertz's result, we use the symbolic method, and we review this 
classical tool in some detail during the proof.
We prove our main theorem (Theorem~\ref{MainTh}) in \S 4.
We show that under the conditions $|\widetilde{R}|>0$ and $r_1=\cdots=r_6=0$, there exists 
a local isometric embedding from $M^3$ into $\mathbb{R}^4$.
We repeatedly use the symbolic method in the proof.
In the final section \S 5, we state some application of the main theorem.
We exhibit the results concerning the warped product metrics and the left-invariant metrics 
on $3$-dimensional Lie groups here.

\medskip

For the survey of isometric embedding problems, see the summary 
papers \cite{AK10}, \cite{Bor1}, \cite{Chen1}, \cite{GR}, \cite{Mat}, \cite{PS}, \cite{Sp}.

\bigskip

%{\bf Acknowledgement.}
%The first author was supported by JSPS KAKENHI Grant Number 16K05132, 20K03589 
%and the second author was supported by JSPS KAKENHI Grant Number 16K17603.

\bigskip

\section{Basic facts and known results on hypersurfaces in the Euclidean spaces}

We first recall some fundamental facts on hypersurfaces in the Euclidean space 
$\mathbb{R}^{m+1}$, following \cite{D}, \cite{KN}, etc. 
Let $(M^m,g)$ be an $m$-dimensional Riemannian manifold,  
and $f$ be an isometric embedding of $(M^m,g)$ into $\mathbb{R}^{m+1}$. 

Let $\nabla$ and $\overline{\nabla}$ be the Levi-Civita connections of $M^m$ and 
$\mathbb{R}^{m+1}$ respectively. 
Then, for any vector fields $X,Y$ on $M^m$, we decompose 
$\overline{\nabla}_X Y$ into its 
tangential part  $(\overline{\nabla}_X Y)^{\top }$ and its normal part
 $(\overline{\nabla}_X Y)^{\perp}$.
Note that $\nabla_X Y=(\overline{\nabla}_X Y)^{\top }$.
We define the second fundamental form $\overline{\alpha}$ of $M^m$ by 
\begin{align*}
\overline{\alpha} (X,Y) :=(\overline{\nabla}_X Y)^{\perp},
\end{align*}
and the covariant derivative of the second fundamental form by
\begin{align*}
(\nabla_Z^{\perp} \overline{\alpha})(X,Y):= \nabla^{\perp}_Z \overline{\alpha}(X,Y) 
- \overline{\alpha}(\nabla_Z X, Y)- \overline{\alpha}(X, \nabla_Z Y)
\end{align*}
for vector fields $X,Y,Z$ on $M^m$.
As is well known, the second fundamental form $\overline{\alpha}$ is symmetric, and 
$\overline{\alpha}$, $\nabla^{\perp} \overline{\alpha}$ are tensor fields on $M^m$.
In addition, they satisfy the following equations 
\begin{align*}
\tag{$\overline{1}$} & -g( R(X,Y)Z , W ) = \langle \overline{\alpha}(X,Z) , 
\overline{\alpha}(Y,W) \rangle 
- \langle \overline{\alpha}(X,W) , \overline{\alpha}(Y,Z) \rangle, \\
\tag{$\overline{2}$} & (\nabla_Z ^{\perp} \overline{\alpha}) (X,Y) 
= (\nabla_Y ^{\perp} \overline{\alpha}) (X,Z)
\end{align*}
\noindent
for vector fields $X,Y,Z,W$ on $M^m$. 
Here $R$ is the Riemannian curvature tensor of $M^m$. 
The first equation is called the {\it Gauss equation}, and the second one the {\it Codazzi equation}.

We fix a unit normal vector field $\nu$ on $M^m$ throughout this paper.
(Note that such a vector field exists locally, or in case $M^m$ is simply connected.)
Then the second fundamental form $\overline{\alpha}$ may be expressed as
$$
\overline{\alpha}(X,Y) = \alpha(X,Y)\nu
$$
for some symmetric $2$-tensor $\alpha$.
Then the above two equations $(\overline{1})$ and $(\overline{2})$ can be expressed as 
\begin{align*}
\tag{1} & -g( R(X,Y)Z , W ) = \alpha(X,Z) \alpha(Y,W) - \alpha(X,W) \alpha(Y,Z), \\
\tag{2} & (\nabla_Z\alpha) (X,Y) = (\nabla_Y\alpha)(X,Z).
\end{align*}

Now we recall the fundamental theorem of hypersufaces in the Euclidean space, which is a
key theorem in this paper (cf. \cite{Chen2}, \cite{D}, \cite{DT}, \cite{KN}).

\begin{thm}\label{FT}
Let $(M^m,g)$ be an $m$-dimensional simply connected Riemannian manifold. 
Then, there exists an isometric embedding $f:(M^m,g) \longrightarrow 
\mathbb{R}^{m+1}$ if and only if there exists a symmetric $2$-tensor field $\alpha$ 
on $M^m$ satisfying the Gauss and Codazzi equations $(1)$, $(2)$.
\end{thm} 

\noindent
This theorem is very useful to show the existence of a local isometric embedding $f$, because 
we have only to find a symmetric tensor field $\alpha$ satisfying two equations $(1)$ and $(2)$. 
However, it is very hard to show the non-existence of a local isometric 
embedding for a given Riemannian manifold by using this theorem.
In fact, the equation (2) is a differential equation on $\alpha$ and it is in general difficult to 
show the non-existence of a solution of (2) among the set of solutions of the algebraic 
equation (1).

Concerning the existence and non-existence of a local isometric embedding $f$, we 
review classical results of Thomas \cite{T} on the intrinsic characterization of Riemannian 
manifolds that can be locally isometrically embedded into $\mathbb{R}^{m+1}$.

\begin{thm}[\cite{T}]\label{Thomas}
Let $(M^m,g)$ be a Riemannian manifold.
Suppose that the type number $\geq 4$ at any point of $M^m$. 
If there exists a symmetric $2$-tensor field $\alpha$ satisfying the Gauss equation on 
$M^m$, then $\alpha$ automatically satisfies the Codazzi equation on 
$M^m$.
\end{thm}

\noindent
For the definition of the type number, see \cite{T}, \cite{KN}.
We note that  the type number is an intrinsic quantity of a Riemannian manifold $(M^m,g)$.
For details, see Theorem 6.1 in p.42 of \cite{KN} and also the comments after Corollary 6.5 
in p.46.
We emphasize that Theorem~\ref{Thomas} is useful only in the case $m \geq 4$.
(See also Remark~\ref{G-C}.)

Concerning the solvability of the Gauss equation, we already know a theoretical  
method to determine whether it admits a solution or not, as we stated in Introduction.
For details, see Remark~\ref{Rem_const} (3).
See also Vilms \cite{V} for the case $m \geq 5$.

In this paper we give an intrinsic characterization of $3$-dimensional Riemannian manifolds
so that it can be locally isometrically embedded into $\mathbb{R}^4$.
For this purpose we must treat the Codazzi equation in an intrinsic manner.
This problem is already treated in Thomas \cite{T} and Rosenson \cite{Ro}.
But the statement in p.206 of \cite{T}, or the condition stated in p.345 of \cite{Ro} is 
merely a direct rehash of the Codazzi equation, and it seems to be far from satisfactory as 
a final intrinsic characterization.

Now we review the results of Weise \cite{W}, Thomas \cite{T} and Rivertz \cite{R1} 
for the case $\dim M = 3$.
We prove this theorem in \S 3 for the sake of completeness, by using the symbolic method.

\begin{thm}\label{WTR}
Assume that a $3$-dimensional Riemannian manifold $(M^3,g)$ is isometrically 
embedded into $\mathbb{R}^4$.
Then:

$(1)$ $(${\rm Weiss, Thomas}$)$ The following inequality holds:
\begin{align*}
|\widetilde{R}|=\begin{vmatrix}
R_{1212} & R_{1213} & R_{1223}\\
R_{1213} & R_{1313} & R_{1323}\\
R_{1223} & R_{1323} & R_{2323}
\end{vmatrix} \geq 0.
\end{align*}

$(2)$ $(${\rm Rivertz}$)$ The curvature tensor $R$ and its covariant derivative 
$S=\nabla R$ satisfy the six equalities  $r_1 = \cdots = r_6=0$ listed up in Introduction.
\end{thm}

\noindent
Here, the component of the curvature tensor $R$ is given by 
$$
R_{ijkl} = -g(R(X_i,X_j)X_k,X_l) = R(X_i,X_j,X_k,X_l), 
$$
where $X_1$, $X_2$, $X_3$ are local vector fields on $M^3$, giving a basis of 
the tangent space $T_pM^3$ at each point $p$.
Then the coefficient of the covariant derivative $S = \nabla R$ is given by 
$$
S_{ijklm} = (\nabla_{X_m} R)(X_i,X_j,X_k,X_l).
$$
We denote the component of the second fundamental form $\alpha$ by 
$\alpha_{ij}=\alpha(X_i,X_j)$ $(\alpha_{ij} = \alpha_{ji})$.
Then the Gauss equation is expressed as 
\begin{align*}
& R_{ijkl} = \alpha_{ik}\alpha_{jl} - \alpha_{il}\alpha_{jk}.
\end{align*}
The covariant derivative of the Gauss equation is given by 
$$
S_{ijklm} = \beta_{ikm}\alpha_{jl}+\alpha_{ik}\beta_{jlm}-\beta_{ilm}\alpha_{jk}-
\alpha_{il}\beta_{jkm},
$$
where $\beta= \nabla \alpha$ and $\beta_{ijk} = (\nabla_{X_k}\alpha)(X_i,X_j)$.
We call the above equality the \textit{derived Gauss equation} in this paper.
This equation plays a crucial role in the following arguments.
Note that the tensor $\beta$ satisfies $\beta_{ijk} = \beta_{jik}$ and the Codazzi equation 
$\beta_{ijk} = \beta_{ikj}$.
Hence, $\beta$ is a symmetric $3$-tensor field.
The tensor $S$ satisfies the second Bianchi identity 
$$
S_{ijklm} + S_{ijlmk} + S_{ijmkl} = 0,
$$
in addition to the standard symmetric and anti-symmetric properties 
\begin{align*}
& S_{ijklm} = -S_{jiklm} = -S_{ijlkm} = S_{klijm}.
\end{align*}

We add one remark.
It is known that if $M^3$ can be isometrically embedded into $\mathbb{R}^4$, and the 
type number $=3$ at every point of $M^3$, then the isometric embedding $f : M^3 
\rightarrow \mathbb{R}^4$ is rigid.
Note that the assumption on the type number is a generic condition on $(M^3.g)$ in this 
situation.
For details, see \cite{KN}.

\bigskip

\section{Proof of Theorem~\ref{WTR}}

In this section, we give a proof of Theorem~\ref{WTR} for the sake of completeness.
Especially, our proof of Rivertz's result (Theorem~\ref{WTR} (2)) is based on the symbolic 
method in classical invariant theory, which is different from his original proof of $r_1=0$ given 
in p.119$\sim$122 of \cite{R1}.
We explain this classical tool in detail during the proof of Theorem~\ref{WTR} (2).
In the rest of this paper we always assume ${\rm dim}\: M = 3$, unless otherwise stated.

\subsection{Inverse formula of the Gauss equation}
We use the same notations as in the previous section.
Assume that a $3$-dimensional Riemannian manifold $(M,g)$ is isometrically 
embedded into $\mathbb{R}^4$ with its second fundamental form $\alpha$.
We put 
$$
A :=\begin{pmatrix}
\alpha_{11} & \alpha_{12} & \alpha_{13} \\
\alpha_{12} & \alpha_{22} & \alpha_{23} \\
\alpha_{13} & \alpha_{23} & \alpha_{33} \\
\end{pmatrix}.
$$
Then we have the following proposition.
Theorem~\ref{WTR} (1) follows immediately from (1) of this proposition.
The latter part (2) plays an essential role in the proof of Theorem~\ref{WTR} (2) 
(cf. Lemma~\ref{Lem_det_id}) and also the arguments in \S 4.

\begin{prop}\label{RA2}
$(1)$ The equality $|\widetilde{R}| = |A|^2$ holds.

$(2)$ At the point $p \in M$ where $|\widetilde{R}|>0$, the second fundamental form 
$\alpha_{ij}$ can be expressed uniquely as functions of $R_{ijkl}$ as follows:

\begin{align*}
\alpha_{11} = \frac{\varepsilon}{\sqrt{|\widetilde{R}|}}
\begin{vmatrix}
R_{1212} & R_{1213} \\
R_{1213} & R_{1313} 
\end{vmatrix}, \quad 
\alpha_{12} = \frac{\varepsilon}{\sqrt{|\widetilde{R}|}}
\begin{vmatrix}
R_{1212} & R_{1223} \\
R_{1213} & R_{1323} 
\end{vmatrix}, \\
\rule{0cm}{0.9cm} \alpha_{13} = \frac{\varepsilon}{\sqrt{|\widetilde{R}|}}
\begin{vmatrix}
R_{1213} & R_{1223} \\
R_{1313} & R_{1323} 
\end{vmatrix}, \quad 
\alpha_{22} = \frac{\varepsilon}{\sqrt{|\widetilde{R}|}}
\begin{vmatrix}
R_{1212} & R_{1223} \\
R_{1223} & R_{2323} 
\end{vmatrix}, \\
\rule{0cm}{0.9cm} \alpha_{23} = \frac{\varepsilon}{\sqrt{|\widetilde{R}|}}
\begin{vmatrix}
R_{1213} & R_{1223} \\
R_{1323} & R_{2323} 
\end{vmatrix}, \quad 
\alpha_{33} = \frac{\varepsilon}{\sqrt{|\widetilde{R}|}}
\begin{vmatrix}
R_{1313} & R_{1323} \\
R_{1323} & R_{2323} 
\end{vmatrix}, 
\end{align*}
where $\varepsilon=1$ or $-1$.
\end{prop}

\begin{proof}
To prove this proposition, we only use the pointwise value of $R_{ijkl}$ and 
$\alpha_{ij}$.
Hence, in this proof, we consider $\alpha_{ij}$ as a symmetric tensor on 
a $3$-dimensional abstract vector space, and $R_{ijkl}$ as a curvature like tensor on 
this vector space, satisfying the Gauss equation.

(1) 
First of all, we treat the case where $A$ is non-singular.
From the Gauss equation, we have
\begin{align}
A^{-1} & = \frac{1}{|A|} 
\begin{pmatrix}
\alpha_{22}\alpha_{33}-\alpha_{23}{}^2 & 
\alpha_{13}\alpha_{23}-\alpha_{12}\alpha_{33} & 
\alpha_{12}\alpha_{23}-\alpha_{13}\alpha_{22} \\
\alpha_{13}\alpha_{23}-\alpha_{12}\alpha_{33} & 
\alpha_{11}\alpha_{33}-\alpha_{13}{}^2 & 
\alpha_{12}\alpha_{13}-\alpha_{11}\alpha_{23} \\
\alpha_{12}\alpha_{23}-\alpha_{13}\alpha_{22} &
\alpha_{12}\alpha_{13}-\alpha_{11}\alpha_{23} &
\alpha_{11}\alpha_{22}-\alpha_{12}{}^2
\end{pmatrix} \label{eq6} \\
& = \frac{1}{|A|}
\begin{pmatrix}
R_{2323} & -R_{1323} & R_{1223} \\
-R_{1323} & R_{1313} & -R_{1213} \\
R_{1223} & -R_{1213} & R_{1212}
\end{pmatrix}. \notag
\end{align}
By taking the determinants of both sides, we have 
$$
\frac{1}{|A|} = \frac{1}{|A|^3} |\widetilde{R}|,
$$
which implies $|\widetilde{R}| = |A|^2$.

Next we consider the singular case $|A|=0$.
The components $R_{ijkl}$ of the determinant $|\widetilde{R}|$ are quadratic polynomials of 
$\alpha_{ij}$, and hence $|\widetilde{R}|-|A|^2$ is a homogeneous polynomial of 
$\alpha_{ij}$ with degree six.
This polynomial vanishes identically on the open dense subset consisting of 
non-singular symmetric $3$-matrices $A=(\alpha_{ij})$, as we showed above.
Since the polynomial $|\widetilde{R}|-|A|^2$ is continuous as a function of 
$\alpha_{ij}$, it vanishes identically on the whole space, which implies that the equality 
$|\widetilde{R}|-|A|^2=0$ also holds in case $|A|=0$.

(2) 
From the assumption $|\widetilde{R}|>0$, it follows that the matrix $A$ is non-singular.
Then, taking the inverse of the equality (\ref{eq6}), we have 
$$
A = |A|
\begin{pmatrix}
R_{2323} & -R_{1323} & R_{1223} \\
-R_{1323} & R_{1313} & -R_{1213} \\
R_{1223} & -R_{1213} & R_{1212}
\end{pmatrix}^{-1}.
$$
From this equality and $|\widetilde{R}|=|A|^2$, we see that the $(1,1)$-component of 
$A$ is 
$$
\alpha_{11}=
\frac{|A|}{|\widetilde{R}|} \begin{vmatrix}
R_{1313} & -R_{1213} \\ -R_{1213} & R_{1212} 
\end{vmatrix}
= \frac{1}{|A|} \begin{vmatrix}
R_{1212} & R_{1213} \\ R_{1213} & R_{1313} 
\end{vmatrix}.
$$
We put $\varepsilon = \sqrt{|\widetilde{R}|}/|A|$.
Then it is $1$ or $-1$.
Thus we obtain the first equality in Proposition~\ref{RA2} (2).
Remaining five equalities can be obtained in the same way.
\end{proof}

\medskip

\begin{rem}\label{Rem_const}
(1) This proposition is essentially due to Weiss \cite{W} and 
Thomas \cite{T}.
In the following, we call six equalities in (2) of Proposition~\ref{RA2} the 
\textit{inverse formula} of the Gauss equation.
It is easy to see that this inverse formula can be expressed in the following unified manner:
$$
\alpha_{ip} = \frac{1}{|A|} \begin{vmatrix}
R_{ijpq} & R_{ijpr} \\
R_{ikpq} & R_{ikpr} 
\end{vmatrix}
= \frac{\varepsilon}{\sqrt{|\widetilde{R}|}}
\begin{vmatrix}
R_{ijpq} & R_{ijpr} \\
R_{ikpq} & R_{ikpr} 
\end{vmatrix}, 
$$
where $(i,j,k)$ and $(p,q,r)$ are cyclic permutations of $(1,2,3)$.
We once emphasize that to prove Proposition~\ref{RA2} we only use the pointwise 
Gauss equation.
The existence of isometric embedding itself or the Codazzi equation is unnecessary as 
our assumption.

(2) In case the ambient space $\mathbb{R}^4$ is replaced by the $4$-dimensional space of 
constant curvature $c$, the similar inequality as Theorem~\ref{WTR} (1) holds.
In this case the Gauss equation is expressed as 
$$
R_{ijkl} - c(\delta_{ik}\delta_{jl}-\delta_{il}\delta_{jk}) = 
\alpha_{ik}\alpha_{jl}-\alpha_{il}\alpha_{jk},
$$
where $\delta_{ij}$ is the Kronecker symbol.
Thus, as in the above,  we can show the inequality 
\begin{align*}
\begin{vmatrix}
R_{1212}-c & R_{1213} & R_{1223}\\
R_{1213} & R_{1313}-c & R_{1323}\\
R_{1223} & R_{1323} & R_{2323}-c
\end{vmatrix} \geq 0, 
\end{align*}
which gives an obstruction to the existence of local isometric embedding of $M$ into 
the $4$-dimensional space of constant curvature $c$.
See also Remark~\ref{subst_exchange} (4).

(3) 
Weiss \cite{W} and Thomas \cite{T} considered the inverse formula of the Gauss equation for 
general dimensional case, not only restricted to the case ${\rm dim} \: M = 3$.
Assume that the determinants 
$$
|\widetilde{R}_{ijk,pqr}|:=
\begin{vmatrix} 
R_{ijpq} & R_{ijpr} & R_{ijqr} \\
R_{ikpq} & R_{ikpr} & R_{ikqr} \\
R_{jkpq} & R_{jkpr} & R_{jkqr} 
\end{vmatrix}
$$
do not vanish for all distinct indices $(i,j,k)$ and $(p,q,r)$, where $i,j,k,p,q,r$ move from $1$ to 
${\rm dim} \: M (\geq 3)$.
Then, if the Gauss equation admits a solution, we have $|\widetilde{R}_{ijk,pqr}|>0$, 
and the equality 
$$
\alpha_{ip}{}^2 = \frac{
\begin{vmatrix}
R_{ijpq} & R_{ijpr} \\ R_{ikpq} & R_{ikpr} \end{vmatrix}^2
}{|\widetilde{R}_{ijk,pqr}|}
$$
holds as in the case of ${\rm dim}\: M = 3$.
Thus, under the generic assumption $|\widetilde{R}_{ijk,pqr}| \neq 0$,  we can determine 
the value of $\alpha_{ip}$ up to  sign from the curvature.
But in case ${\rm dim} \: M \geq 4$, we can select different indices to get the value of 
$\alpha_{ip}$ such as 
$$
\alpha_{12}{}^2 =\frac{
\begin{vmatrix}
R_{1223} & R_{1221} \\ R_{1323} & R_{1321} \end{vmatrix}^2
}{|\widetilde{R}_{123,231}|}
=\frac{
\begin{vmatrix}
R_{1223} & R_{1224} \\ R_{1323} & R_{1324} \end{vmatrix}^2
}{|\widetilde{R}_{123,234}|}.
$$
Sometimes there may occur the case where the curvature gives different values 
$\alpha_{ip}^2$.
Clearly it is a contradiction, and hence the Gauss equation does not admit a solution in 
such a case.

In case ${\rm dim}\: M \geq 4$, we can theoretically determine whether the Gauss equation 
admits a solution or not in the following way, assuming $|\widetilde{R}_{ijk,pqr}|>0$.
At first, for each pair of indices $(i,p)$ with $i \leq p$, we fix the value of $\alpha_{ip}$ by 
using the above formula on $\alpha_{ip}{}^2$.
(Here, we choose different indices $j, k (\neq i)$ and $q, r (\neq p)$ and select the sign of 
$\alpha_{ip}$ arbitrarily.) 
Then we can easily verify whether such $\alpha_{ip}$'s satisfy the Gauss equation or not, 
by substituting $\alpha_{ip}$ into the Gauss equation.
If they satisfy the Gauss equation, then it gives a solution.
If not, then we change the sign of some $\alpha_{ip}$'s and examine again whether it is a solution 
or not.
We continue this procedure untill we arrive at a solution, or exhaust all combination of signs 
of $\alpha_{ip}$'s.
This procedure will end after some examinations since the combination of signs are 
finite in number.
In this way we can finally determine whether the Gauss equation admits a solution or not.
But this procedure certainly contains many redundant examinations, and it is desirable to 
obtain more efficient method to determine the solvability of the Gauss equation in case 
${\rm dim}\: M \geq 4$.

\end{rem}

\medskip

\subsection{Symbolic method}

To prove Theorem~\ref{WTR} (2), we revive an old tool called the $\lq\lq$symbolic 
method'', which was introduced in the field of classical invariant theory in the 
nineteenth century.
(For details on the symbolic method, see Dolgachev \cite{Dol}, Gurevich \cite{Gu}, 
Olver \cite{O}, and references therein.
For its history, see Friedman \cite{F}.)
We first express Rivertz's six equalities $r_1 = \cdots = r_6=0$ in terms 
of symbols (Proposition~\ref{Prop_Riv}).
We define a quadratic polynomial $RS(x_1,x_2,x_3)$ on $x_1,x_2,x_3$ by 
$$
RS(x_1,x_2,x_3) = 
\begin{vmatrix} 
R_{12} & \overline{R}_{12} & S_{12} \\
R_{13} & \overline{R}_{13} & S_{13} \\
R_{23} & \overline{R}_{23} & S_{23} \\
\end{vmatrix}
\begin{vmatrix} 
R_{12} & \overline{R}_{12} & (x \wedge S)_{12} \\
R_{13} & \overline{R}_{13} & (x \wedge S)_{13} \\
R_{23} & \overline{R}_{23} & (x \wedge S)_{23} \\
\end{vmatrix}
(x_1S_{23}  -x_2S_{13}+x_3S_{12}), 
$$
where 
$$
(x \wedge S)_{ij} = x_i S_j - x_j S_i.
$$
The coefficients of $x_ix_j$ in $RS(x_1,x_2,x_3)$ are expressed as a product 
of $\lq\lq$symbols'' $R_{ij}$, 
$\overline{R}_{ij}$, $S_{ij}$ and $S_i$.
In calculating the products, we introduce the rules 
\begin{align*}
& R_{ij}R_{kl}=\overline{R}_{ij} \overline{R}_{kl} =R_{ijkl}, \\
& S_{ij}S_{kl}S_m = S_{ijklm},
\end{align*}
where $R_{ijkl}$ and $S_{ijklm}$ are the curvature tensor and its covariant derivative, 
respectively.
On account of the properties 
$$
R_{ijkl} = R_{klij}, \quad S_{ijklm} = S_{klijm},
$$
these rules are applicable without any confusion.
For example, in the expression of $RS(x_1,x_2,x_3)$, the product of the diagonal components 
of two determinants is equal to 
$$
R_{12} \overline{R}_{13}S_{23} \cdot R_{12} \overline{R}_{13} (x \wedge S)_{23}
= R_{1212}R_{1313}S_{23}(x_2S_3-x_3S_2), 
$$
and by taking the product with $x_1S_{23}$, it becomes 
$$
R_{1212}R_{1313}S_{23233}\,x_1x_2 - R_{1212}R_{1313}S_{23232}\,x_1x_3.
$$
In this way, we know that the coefficients of $x_ix_j$ in $RS(x_1,x_2,x_3)$ are restored to 
homogeneous polynomials of $R_{ijkl}$ and $S_{ijklm}$ with degrees two and one, respectively.

In the following we show that these coefficients correspond to $r_1 \sim r_6$ 
(Proposition~\ref{Prop_Riv}).
But before describing the precise statement on $RS(x_1,x_2,x_3)$, we further explain 
several facts and remarks on symbols for later use.
We note that the symbols which we use in this paper may look different from 
the classical ones, though they are essentially the same.
First of all, we mention that the single symbols such as $R_{12}$, 
$\overline{R}_{13}$, $S_{23}$, 
$S_1$ are merely symbols, which do not possess any geometric meaning by themselves.
Symbols must be always paired with their symbolic partners which is determined uniquely 
in each monomial.
On account of this property, we can safely and uniquely restore the product of symbols to the 
polynomial as in the following way:
\begin{align*}
& (R_{12}S_{23}-R_{13}S_{12})^2S_3 \\
= & R_{12}{}^2S_{23}{}^2S_3-2R_{12}R_{13}S_{23}S_{12}S_3
+R_{13}{}^2S_{12}{}^2S_3 \\
= & R_{1212}S_{23233}-2R_{1213}S_{12233}+R_{1313}S_{12123}.
\end{align*}
However, we cannot restore the expressions such as 
$$
R_{12}S_{12}S_{13}S_3, \quad R_{12}R_{23}S_{12}S_{13}, \quad 
R_{12}R_{13}\overline{R}_{23}
$$
to polynomials, since some partners are lacking for each case.
Similarly, we cannot restore the expression $R_{12}R_{23}S_{12}S_{13}S_{23}S_1$ to a 
polynomial, since one $S_{ij}$ is abundant in this case.

During symbolic calculations, we can treat symbols as if they are polynomials.
But actually, symbols are not polynomials, because we cannot substitute scalar values into them.
For example, if the symbol $R_{12}$ takes a real value, we have $R_{1212} = 
R_{12}{}^2 \geq 0$, which is not the case in general.
Similarly, if $R_{12}$ and $R_{13}$ take scalar values $a$, $b$, respectively, we have 
$R_{1212} = R_{12}{}^2=a^2$, $R_{1213}=R_{12}R_{13}=ab$ and $R_{1313}=R_{13}{}^2=b^2$, 
which imply that the value of the polynomial $R_{1212}R_{1313} - R_{1213}{}^2$ is 
always zero.
This is also not the case.
However, unless we substitute scalar values into symbols, we may treat them as if they 
are usual polynomials.

In our case, we further impose the following relations 
\begin{align*}
& R_{ij} = -R_{ji}, \quad \overline{R}_{ij} = -\overline{R}_{ji}, \\
& S_{ij} = -S_{ji}, \quad S_{ij}S_k+S_{jk}S_i+S_{ki}S_j=0
\end{align*}
on symbols in order to inherit the properties 
$$
R_{ijkl} = -R_{jikl} = -R_{ijlk}, \quad 
S_{ijklm} = -S_{jiklm} = -S_{ijlkm}, \\
$$
and the second Bianchi identity 
\begin{align*}
& S_{ijklm}+S_{ijlmk}+S_{ijmkl}=0.
\end{align*}
For example, on account of the relation $S_{12}S_3+S_{23}S_1+S_{31}S_2=0$, the second 
Bianchi identity 
$$
S_{ij}(S_{12}S_3+S_{23}S_1+S_{31}S_2)= S_{ij123}+S_{ij231}+S_{ij312}=0 
$$
follows.
(Note that the first Bianchi identity $R_{ijkl}+R_{iklj}+R_{iljk} = 0$ 
becomes a trivial relation in case ${\rm dim}\: M = 3$.)

We add one more remark.
In the above explanation, we say that symbols must be always paired with their symbolic 
partners.
But sometimes we are allowed to calculate such as 
\begin{align*}
& 2S_{12}-3S_{12}+S_{12}=0, \\
& (2S_{12}-3S_{12}+S_{12})S_{13}=0
\end{align*}
symbolically, under the premise that the remaining partners are multiplied at the 
final stage of calculations.
Actually, these equalities themselves possess no geometric meaning, but after taking products 
with the remaining partners $S_{ij}S_k$ and $S_{l}$, respectively, they revive the meaning
\begin{align*}
& 2S_{12ijk}-3S_{12ijk}+S_{12ijk}=0, \\
& 2S_{1213l}-3S_{1213l}+S_{1213l}=0
\end{align*}
as polynomial equalities.
For more remarks on symbolic calculations, we will explain them where they are necessary.

Now we return to the property of the polynomial $RS(x_1,x_2,x_3)$.
We prove the following proposition.

\begin{prop}\label{Prop_Riv}
The equality 
$$
\frac{1}{\,2\,} RS(x_1,x_2,x_3) = 
r_1x_3{}^2-r_2x_2x_3+r_3x_1x_3+r_4x_2{}^2-r_5x_1x_2+r_6x_1{}^2
$$
holds, where $r_1 \sim r_6$ are Rivertz's six polynomials listed up in Introduction.
In particular, Rivertz's six equalities $r_1 = \cdots = r_6=0$ are equivalent to the 
vanishing of the single polynomial $RS(x_1,x_2,x_3)=0$.
\end{prop}

\begin{proof}
We assume that $(i,j,k)$ is a cyclic permutation of $(1,2,3)$ in this proof.
Then we have 
\begin{align*}
RS(x_1,x_2,x_3) = 
& \begin{vmatrix} 
R_{12} & \overline{R}_{12} & S_{12} \\
R_{23} & \overline{R}_{23} & S_{23} \\
R_{31} & \overline{R}_{31} & S_{31} \\
\end{vmatrix}
\begin{vmatrix} 
R_{12} & \overline{R}_{12} & (x \wedge S)_{12} \\
R_{23} & \overline{R}_{23} & (x \wedge S)_{23} \\
R_{31} & \overline{R}_{31} & (x \wedge S)_{31} \\
\end{vmatrix}
(x_1S_{23} +x_2S_{31}+x_3S_{12}) \\
= & \begin{vmatrix} 
R_{ij} & \overline{R}_{ij} & S_{ij} \\
R_{jk} & \overline{R}_{jk} & S_{jk} \\
R_{ki} & \overline{R}_{ki} & S_{ki} \\
\end{vmatrix}
\begin{vmatrix} 
R_{ij} & \overline{R}_{ij} & (x \wedge S)_{ij} \\
R_{jk} & \overline{R}_{jk} & (x \wedge S)_{jk} \\
R_{ki} & \overline{R}_{ki} & (x \wedge S)_{ki} \\
\end{vmatrix}
(x_iS_{jk} +x_jS_{ki}+x_kS_{ij}). 
\end{align*}
In this expression, the coefficient of $x_i{}^2$ is given by 
\begin{align*}
& \begin{vmatrix} 
R_{ij} & \overline{R}_{ij} & S_{ij} \\
R_{jk} & \overline{R}_{jk} & S_{jk} \\
R_{ki} & \overline{R}_{ki} & S_{ki} \\
\end{vmatrix}
\left\{ 
\begin{vmatrix}
R_{jk} & \overline{R}_{jk} \\
R_{ki} & \overline{R}_{ki}
\end{vmatrix} S_j - 
\begin{vmatrix}
R_{ij} & \overline{R}_{ij} \\
R_{jk} & \overline{R}_{jk}
\end{vmatrix} S_k \right\} S_{jk} \\
= & \left\{ \begin{vmatrix} 
R_{jk} & \overline{R}_{jk} \\
R_{ki} & \overline{R}_{ki} 
\end{vmatrix} S_{ij}
- \begin{vmatrix} 
R_{ij} & \overline{R}_{ij} \\
R_{ki} & \overline{R}_{ki} 
\end{vmatrix} S_{jk}
+ \begin{vmatrix} 
R_{ij} & \overline{R}_{ij} \\
R_{jk} & \overline{R}_{jk} 
\end{vmatrix} S_{ki} \right\} \\
& \times 
\left\{ 
\begin{vmatrix}
R_{jk} & \overline{R}_{jk} \\
R_{ki} & \overline{R}_{ki}
\end{vmatrix} S_j - 
\begin{vmatrix}
R_{ij} & \overline{R}_{ij} \\
R_{jk} & \overline{R}_{jk}
\end{vmatrix} S_k \right\} S_{jk}.
\end{align*}
Here, we have in general 
\begin{align*}
& \begin{vmatrix}
R_{ab} & \overline{R}_{ab} \\
R_{cd} & \overline{R}_{cd}
\end{vmatrix}
\begin{vmatrix}
R_{pq} & \overline{R}_{pq} \\
R_{rs} & \overline{R}_{rs}
\end{vmatrix} \\
= & (R_{ab} \overline{R}_{cd} - R_{cd} \overline{R}_{ab})
(R_{pq} \overline{R}_{rs} - R_{rs} \overline{R}_{pq}) \\
= & 2(R_{abpq}R_{cdrs} - R_{abrs}R_{cdpq}).
\end{align*}
By using this equality, we know that the coefficient of $x_i{}^2$ is equal to 
\begin{align*}
& 2 \{ (R_{jkjk}R_{kiki}-R_{jkki}R_{kijk})S_{ijjkj}
-(R_{jkij}R_{kijk}-R_{jkjk}R_{kiij})S_{ijjkk} \\
& -(R_{ijjk}R_{kiki}-R_{ijki}R_{kijk})S_{jkjkj}
+(R_{ijij}R_{kijk}-R_{ijjk}R_{kiij})S_{jkjkk} \\
& +(R_{ijjk}R_{jkki}-R_{ijki}R_{jkjk})S_{kijkj}
-(R_{ijij}R_{jkjk}-R_{ijjk}R_{jkij})S_{kijkk} \}.
\end{align*}
Then, by putting $i=1$, $j=2$, $k=3$, it becomes 
\begin{align*}
& 2 \{ (R_{2323}R_{3131}-R_{2331}R_{3123})S_{12232}
-(R_{2312}R_{3123}-R_{2323}R_{3112})S_{12233} \\
& -(R_{1223}R_{3131}-R_{1231}R_{3123})S_{23232}
+(R_{1212}R_{3123}-R_{1223}R_{3112})S_{23233} \\
& +(R_{1223}R_{2331}-R_{1231}R_{2323})S_{31232}
-(R_{1212}R_{2323}-R_{1223}R_{2312})S_{31233} \} \\
= & 2 \{ (R_{2323}R_{1313}-R_{1323}{}^2)S_{12232}
+(R_{1223}R_{1323}-R_{2323}R_{1213})S_{12233} \\
& -(R_{1223}R_{1313}-R_{1213}R_{1323})S_{23232}
-(R_{1212}R_{1323}-R_{1223}R_{1213})S_{23233} \\
& +(R_{1223}R_{1323}-R_{1213}R_{2323})S_{13232}
+(R_{1212}R_{2323}-R_{1223}{}^2)S_{13233} \} \\
= & 2r_6.
\end{align*}
Similarly, we can see that the coefficients of $x_2{}^2$, $x_3{}^2$ are $2r_4$, 
$2r_1$, by putting $(i,j,k) = (2,3,1)$, $(3,1,2)$, respectively.

Next, we calculate the coefficient of $x_ix_j$ in $RS(x_1,x_2,x_3)$, where $(i,j,k)$ is a 
cyclic permutation of $(1,2,3)$ again.
Then, it is equal to 
\begin{align*}
& \begin{vmatrix} 
R_{ij} & \overline{R}_{ij} & S_{ij} \\
R_{jk} & \overline{R}_{jk} & S_{jk} \\
R_{ki} & \overline{R}_{ki} & S_{ki} \\
\end{vmatrix}
\left\{ 
\begin{vmatrix}
R_{jk} & \overline{R}_{jk} \\
R_{ki} & \overline{R}_{ki}
\end{vmatrix} S_j - 
\begin{vmatrix}
R_{ij} & \overline{R}_{ij} \\
R_{jk} & \overline{R}_{jk}
\end{vmatrix} S_k \right\} S_{ki} \\
+ & \begin{vmatrix} 
R_{ij} & \overline{R}_{ij} & S_{ij} \\
R_{jk} & \overline{R}_{jk} & S_{jk} \\
R_{ki} & \overline{R}_{ki} & S_{ki} \\
\end{vmatrix}
\left\{ 
- \begin{vmatrix}
R_{jk} & \overline{R}_{jk} \\
R_{ki} & \overline{R}_{ki}
\end{vmatrix} S_i - 
\begin{vmatrix}
R_{ij} & \overline{R}_{ij} \\
R_{ki} & \overline{R}_{ki}
\end{vmatrix} S_k \right\} S_{jk}.
\end{align*}
As above, we restore this symbolic expression to a polynomial form.
Then we have 
\begin{align*}
& \left\{ \begin{vmatrix} 
R_{jk} & \overline{R}_{jk} \\
R_{ki} & \overline{R}_{ki} 
\end{vmatrix} S_{ij}
- \begin{vmatrix} 
R_{ij} & \overline{R}_{ij} \\
R_{ki} & \overline{R}_{ki} 
\end{vmatrix} S_{jk}
+ \begin{vmatrix} 
R_{ij} & \overline{R}_{ij} \\
R_{jk} & \overline{R}_{jk} 
\end{vmatrix} S_{ki} \right\} \\
& \times 
\left\{ 
\begin{vmatrix}
R_{jk} & \overline{R}_{jk} \\
R_{ki} & \overline{R}_{ki}
\end{vmatrix} S_{ki}S_j - 
\begin{vmatrix}
R_{ij} & \overline{R}_{ij} \\
R_{jk} & \overline{R}_{jk}
\end{vmatrix} S_{ki}S_k 
- \begin{vmatrix}
R_{jk} & \overline{R}_{jk} \\
R_{ki} & \overline{R}_{ki}
\end{vmatrix} S_{jk}S_i - 
\begin{vmatrix}
R_{ij} & \overline{R}_{ij} \\
R_{ki} & \overline{R}_{ki}
\end{vmatrix} S_{jk}S_k \right\} \\
\rule{0cm}{0.4cm} =& 2 \{ (R_{jkjk}R_{kiki}-R_{jkki}R_{kijk})S_{ijkij}
 - (R_{jkij}R_{kijk}-R_{jkjk}R_{kiij})S_{ijkik} \\
& - (R_{jkjk}R_{kiki}-R_{jkki}R_{kijk})S_{ijjki} 
 - (R_{jkij}R_{kiki}-R_{jkki}R_{kiij})S_{ijjkk} \\
& - (R_{ijjk}R_{kiki}-R_{ijki}R_{kijk})S_{jkkij}
+ (R_{ijij}R_{kijk}-R_{ijjk}R_{kiij})S_{jkkik} \\
& + (R_{ijjk}R_{kiki}-R_{ijki}R_{kijk})S_{jkjki}
+ (R_{ijij}R_{kiki}-R_{ijki}R_{kiij})S_{jkjkk} \\
& + (R_{ijjk}R_{jkki}-R_{ijki}R_{jkjk})S_{kikij}
 - (R_{ijij}R_{jkjk}-R_{ijjk}R_{jkij})S_{kikik} \\
& - (R_{ijjk}R_{jkki}-R_{ijki}R_{jkjk})S_{kijki} 
 - (R_{ijij}R_{jkki}-R_{ijki}R_{jkij})S_{kijkk} \}.
 \end{align*}
The term $(R_{ijij}R_{jkki}-R_{ijki}R_{ijjk})S_{jkkik}$ appears twice in this expression, 
and they are canceled.
Thus the coefficient becomes 
\begin{align*}
& 2 \{ (R_{jkjk}R_{kiki}-R_{jkki}R_{kijk})S_{ijkij}
 - (R_{jkij}R_{kijk}-R_{jkjk}R_{kiij})S_{ijkik} \\
& - (R_{jkjk}R_{kiki}-R_{jkki}R_{kijk})S_{ijjki} 
 - (R_{jkij}R_{kiki}-R_{jkki}R_{kiij})S_{ijjkk} \\
& - (R_{ijjk}R_{kiki}-R_{ijki}R_{kijk})S_{jkkij} 
+ (R_{ijjk}R_{kiki}-R_{ijki}R_{kijk})S_{jkjki} \\
& + (R_{ijij}R_{kiki}-R_{ijki}R_{kiij})S_{jkjkk}
+ (R_{ijjk}R_{jkki}-R_{ijki}R_{jkjk})S_{kikij} \\
&  - (R_{ijij}R_{jkjk}-R_{ijjk}R_{jkij})S_{kikik}
- (R_{ijjk}R_{jkki}-R_{ijki}R_{jkjk})S_{kijki} \}.
\end{align*}
But this polynomial, consisting of twenty monomials, is a little lengthy.
In order to shorten its length, we modify it by adding the following 
polynomial, which is actually zero on account of the second Bianchi identity 
\begin{align*}
& 2 \{ (R_{ijjk}R_{jkki}-R_{ijki}R_{jkjk})
(S_{kiijk}+S_{kijki}+S_{kikij}) \\
& +(R_{jkij}R_{kiki}-R_{jkki}R_{kiij})
(S_{jkijk}+S_{jkjki}+S_{jkkij}) \}.
\end{align*}
Then, finally  we obtain the polynomial 
\begin{align*}
& 2 \{ (R_{jkjk}R_{kiki}-R_{jkki}R_{kijk})S_{ijkij} 
- (R_{jkjk}R_{kiki}-R_{jkki}R_{kijk})S_{ijjki} \\
& + 2(R_{ijjk}R_{kiki}-R_{ijki}R_{kijk})S_{jkjki} 
+ (R_{ijij}R_{kiki}-R_{ijki}R_{kiij})S_{jkjkk} \\
& + 2(R_{ijjk}R_{jkki}-R_{ijki}R_{jkjk})S_{kikij} 
 - (R_{ijij}R_{jkjk}-R_{ijjk}R_{jkij})S_{kikik} \}.
\end{align*}
By putting $i=1$, $j=2$, $k=3$ for example, it becomes 
\begin{align*}
& 2 \{ (R_{2323}R_{3131}-R_{2331}R_{3123})S_{12312} 
- (R_{2323}R_{3131}-R_{2331}R_{3123})S_{12231} \\
& + 2(R_{1223}R_{3131}-R_{1231}R_{3123})S_{23231} 
+ (R_{1212}R_{3131}-R_{1231}R_{3112})S_{23233} \\
& + 2(R_{1223}R_{2331}-R_{1231}R_{2323})S_{31312} 
 - (R_{1212}R_{2323}-R_{1223}R_{2312})S_{31313} \} \\
= & 2 \{ -(R_{2323}R_{1313}-R_{1323}{}^2)S_{12132} 
- (R_{2323}R_{1313}-R_{1323}{}^2)S_{12231} \\
& + 2(R_{1223}R_{1313}-R_{1213}R_{1323})S_{23231} 
+ (R_{1212}R_{1313}-R_{1213}{}^2)S_{23233} \\
& - 2(R_{1223}R_{1323}-R_{1213}R_{2323})S_{13132} 
 - (R_{1212}R_{2323}-R_{1223}{}^2)S_{13133} \} \\
 = & -2r_5.
\end{align*}
Similarly, we can see that the coefficients of $x_2x_3$, $x_3x_1$ are $-2r_2$, 
$2r_3$, respectively.
\end{proof}

The equality in Proposition~\ref{Prop_Riv} shows the power of the 
symbolic expression, since the set of six 
lengthy polynomials are condensed to a simple factorized form of $RS(x_1,x_2,x_3)$.

\begin{rem}\label{Rem_length}
We can rewrite the coefficient of $x_i{}^2$ in $\frac{1}{2} RS(x_1,x_2,x_3)$ in a simple 
determinantal form 
$$
\begin{vmatrix}
R_{ijjk} & R_{ijki} & S_{ijjkj} \\
R_{jkjk} & R_{jkki} & S_{jkjkj} \\
R_{kijk} & R_{kiki} & S_{kijkj}
\end{vmatrix}
- \begin{vmatrix}
R_{ijij} & R_{ijjk} & S_{ijjkk} \\
R_{jkij} & R_{jkjk} & S_{jkjkk} \\
R_{kiij} & R_{kijk} & S_{kijkk}
\end{vmatrix},
$$
where $(i,j,k)$ is a cyclic permutation of $(1,2,3)$.
Similarly, the coefficient of $x_ix_j$ in $\frac{1}{2} RS(x_1,x_2,x_3)$ can be expressed 
as 
\begin{align*}
& \begin{vmatrix}
R_{ijjk} & R_{ijki} & S_{ijkij} \\
R_{jkjk} & R_{jkki} & S_{jkkij} \\
R_{kijk} & R_{kiki} & S_{kikij}
\end{vmatrix}
- 
\begin{vmatrix}
R_{ijij} & R_{ijjk} & S_{ijkik} \\
R_{jkij} & R_{jkjk} & S_{jkkik} \\
R_{kiij} & R_{kijk} & S_{kikik}
\end{vmatrix} \\
& \rule{0cm}{1cm} + \begin{vmatrix}
R_{ijki} & R_{ijij} & S_{ijjkk} \\
R_{jkki} & R_{jkij} & S_{jkjkk} \\
R_{kiki} & R_{kiij} & S_{kijkk}
\end{vmatrix}
- 
\begin{vmatrix}
R_{ijjk} & R_{ijki} & S_{ijjki} \\
R_{jkjk} & R_{jkki} & S_{jkjki} \\
R_{kijk} & R_{kiki} & S_{kijki}
\end{vmatrix}.
\end{align*}
After the modification by the second Bianchi identity given in the above proof, it becomes
$$
\begin{vmatrix}
R_{ijjk} & R_{ijki} & S_{ijkij}-S_{ijjki} \\
R_{jkjk} & R_{jkki} & -2S_{jkjki} \\
R_{kijk} & R_{kiki} & 2S_{kikij}
\end{vmatrix}
+\begin{vmatrix}
R_{ijij} & R_{ijki} \\
R_{kiij} & R_{kiki} 
\end{vmatrix} S_{jkjkk}
-\begin{vmatrix}
R_{ijij} & R_{ijjk} \\
R_{jkij} & R_{jkjk} 
\end{vmatrix} S_{kikik}.
$$ 
The lengths of the expressions of $r_2$, $r_3$, $r_5$ are shortened by this modification, 
though they lost their symmetries, as can be seen from this expression.
\end{rem}

\medskip

\subsection{Proof of Theorem~\ref{WTR} (2)}
Under these preliminaries, we now prove Theorem~\ref{WTR} (2).
We have only to substitute the Gauss equation and the derived Gauss 
equation 
\begin{align*}
& R_{ijkl} = \alpha_{ik}\alpha_{jl}-\alpha_{il}\alpha_{jk}, \\
& S_{ijklm} = \beta_{ikm}\alpha_{jl}+\alpha_{ik}\beta_{jlm}-\beta_{ilm}\alpha_{jk}-
\alpha_{il}\beta_{jkm}
\end{align*}
into Rivertz's six polynomials $r_1 \sim r_6$, and show that they become zero.
But actually, it is hard to verify it by hand calculations, since these six polynomials are too 
lengthy.
To avoid such hard labor, we once use the symbolic method.
Namely, we substitute $R_{ijkl}$ and $S_{ijklm}$ into $RS(x_1,x_2,x_3)$ symbolically, 
and show that it becomes zero.
For this purpose, we prepare one more lemma in advance.

Let $\rule{0cm}{0.5cm} \alpha_i$, $\alpha_i^{(1)}$, $\alpha_i^{(2)}$, $\alpha_i^{(3)}$, 
$\alpha_i^{(4)}$ be symbols satisfying 
$$
\alpha_i \alpha_j = \alpha_i^{(1)} \alpha_j^{(1)} = \alpha_i^{(2)} \alpha_j^{(2)} =
\alpha_i^{(3)} \alpha_j^{(3)} = \alpha_i^{(4)} \alpha_j^{(4)} = \alpha_{ij}.
$$
Since $\alpha_{ij}$ is symmetric, we can safely restore the symbols to polynomials without 
considering the order of products.
For example, we have 
$$
\alpha_1^{(1)}\alpha_2^{(3)}\alpha_3^{(2)}\alpha_2^{(2)}\alpha_1^{(1)}\alpha_1^{(3)} 
= \alpha_{11}\alpha_{12}\alpha_{23}.
$$
But the products such as $\alpha_1^{(2)} \alpha_2^{(3)} \alpha_1^{(2)} \alpha_3^{(2)}$, 
$\alpha_1 \alpha_2^{(4)}$ have no meaning.
Under these notations, we prove the following lemma.

\begin{lem}\label{Lem_det_id}
Let $x=(x_1,x_2,x_3)$, $y=(y_1,y_2,y_3)$, $z=(z_1,z_2,z_3)$, $w=(w_1,w_2,w_3)$ be 
$3$-dimensional vectors.
Then the following identity holds 
\begin{align*}
& \begin{vmatrix} 
(\alpha^{(1)} \wedge \alpha^{(2)})_{12} & (\alpha^{(3)} \wedge \alpha^{(4)})_{12} 
& (x \wedge y)_{12} \\
\rule{0cm}{0.4cm} (\alpha^{(1)} \wedge \alpha^{(2)})_{13} & (\alpha^{(3)} 
\wedge \alpha^{(4)})_{13} & (x \wedge y)_{13} \\
\rule{0cm}{0.4cm} (\alpha^{(1)} \wedge \alpha^{(2)})_{23} & (\alpha^{(3)} 
\wedge \alpha^{(4)})_{23} & (x \wedge y)_{23} 
\end{vmatrix} 
\begin{vmatrix} 
(\alpha^{(1)} \wedge \alpha^{(2)})_{12} & (\alpha^{(3)} \wedge \alpha^{(4)})_{12} 
& (z \wedge w)_{12} \\
\rule{0cm}{0.4cm} (\alpha^{(1)} \wedge \alpha^{(2)})_{13} & (\alpha^{(3)} 
\wedge \alpha^{(4)})_{13} & (z \wedge w)_{13} \\
\rule{0cm}{0.4cm} (\alpha^{(1)} \wedge \alpha^{(2)})_{23} & (\alpha^{(3)} 
\wedge \alpha^{(4)})_{23} & (z \wedge w)_{23}
\end{vmatrix} \\
& = 8 |A| \begin{vmatrix} 
\alpha_1 & x_1 & y_1 \\
\alpha_2 & x_2 & y_2 \\
\alpha_3 & x_3 & y_3
\end{vmatrix}
\begin{vmatrix} 
\alpha_1 & z_1 & w_1 \\
\alpha_2 & z_2 & w_2 \\
\alpha_3 & z_3 & w_3
\end{vmatrix}.
\end{align*}
\end{lem}

Remind the definition of $\wedge$ defined in \S 3.2.
It is given by $(x \wedge y)_{ij} = x_iy_j-x_jy_i$ etc.

\begin{proof}
As in the proof of Proposition~\ref{RA2}, we consider $\alpha_{ij}$ as a symmetric tensor on 
a $3$-dimensional abstract vector space and $R_{ijkl}$ as a curvature like tensor, 
satisfying the Gauss equation.

We first consider the case $|A| \neq 0$.
We assume that $(i,j,k)$ and $(p,q,r)$ are cyclic permutations of $(1,2,3)$ in this proof.
Then the left hand side of the above equality is equal to 
\begin{align*}
&  \begin{vmatrix} 
(\alpha^{(1)} \wedge \alpha^{(2)})_{ij} & (\alpha^{(3)} \wedge \alpha^{(4)})_{ij} 
& (x \wedge y)_{ij} \\
\rule{0cm}{0.4cm} (\alpha^{(1)} \wedge \alpha^{(2)})_{jk} & (\alpha^{(3)} 
\wedge \alpha^{(4)})_{jk} & (x \wedge y)_{jk} \\
\rule{0cm}{0.4cm} (\alpha^{(1)} \wedge \alpha^{(2)})_{ki} & (\alpha^{(3)} 
\wedge \alpha^{(4)})_{ki} & (x \wedge y)_{ki} 
\end{vmatrix} 
\begin{vmatrix} 
(\alpha^{(1)} \wedge \alpha^{(2)})_{pq} & (\alpha^{(3)} \wedge \alpha^{(4)})_{pq} 
& (z \wedge w)_{pq} \\
\rule{0cm}{0.4cm} (\alpha^{(1)} \wedge \alpha^{(2)})_{qr} & (\alpha^{(3)} 
\wedge \alpha^{(4)})_{qr} & (z \wedge w)_{qr} \\
\rule{0cm}{0.4cm} (\alpha^{(1)} \wedge \alpha^{(2)})_{rp} & (\alpha^{(3)} 
\wedge \alpha^{(4)})_{rp} & (z \wedge w)_{rp}
\end{vmatrix}.
\end{align*}
We take out the coefficient of $x_jy_kz_qw_r$.
Then it is equal to 
\begin{align*}
& \begin{vmatrix}
(\alpha^{(1)} \wedge \alpha^{(2)})_{ij} & (\alpha^{(3)} \wedge \alpha^{(4)})_{ij} \\
\rule{0cm}{0.4cm} (\alpha^{(1)} \wedge \alpha^{(2)})_{ki} & (\alpha^{(3)} 
\wedge \alpha^{(4)})_{ki} 
\end{vmatrix} 
\begin{vmatrix} 
(\alpha^{(1)} \wedge \alpha^{(2)})_{pq} & (\alpha^{(3)} \wedge \alpha^{(4)})_{pq}  \\
\rule{0cm}{0.4cm} (\alpha^{(1)} \wedge \alpha^{(2)})_{rp} & (\alpha^{(3)} 
\wedge \alpha^{(4)})_{rp} 
\end{vmatrix} \\
& =\{  (\alpha_i^{(1)}\alpha_j^{(2)}-\alpha_j^{(1)}\alpha_i^{(2)})
(\alpha_k^{(3)}\alpha_i^{(4)}-\alpha_i^{(3)}\alpha_k^{(4)}) \\
& \hspace{2.5cm} -(\alpha_k^{(1)}\alpha_i^{(2)}-\alpha_i^{(1)}\alpha_k^{(2)})
(\alpha_i^{(3)}\alpha_j^{(4)}-\alpha_j^{(3)}\alpha_i^{(4)}) \}\\
& \hspace{0.5cm} \times \{ (\alpha_p^{(1)}\alpha_q^{(2)}-\alpha_q^{(1)}\alpha_p^{(2)})
(\alpha_r^{(3)}\alpha_p^{(4)}-\alpha_p^{(3)}\alpha_r^{(4)}) \\
& \hspace{2.5cm} -(\alpha_r^{(1)}\alpha_p^{(2)}-\alpha_p^{(1)}\alpha_r^{(2)})
(\alpha_p^{(3)}\alpha_q^{(4)}-\alpha_q^{(3)}\alpha_p^{(4)}) \}.
\end{align*}
We first calculate the product of top terms of both braces.
Then we have 
\begin{align*}
& (\alpha_i^{(1)}\alpha_j^{(2)}-\alpha_j^{(1)}\alpha_i^{(2)})
(\alpha_p^{(1)}\alpha_q^{(2)}-\alpha_q^{(1)}\alpha_p^{(2)}) \\
& \hspace{1cm} \times (\alpha_k^{(3)}\alpha_i^{(4)}-\alpha_i^{(3)}\alpha_k^{(4)})
(\alpha_r^{(3)}\alpha_p^{(4)}-\alpha_p^{(3)}\alpha_r^{(4)}) \\
& = 2(\alpha_{ip}\alpha_{jq}-\alpha_{iq}\alpha_{jp}) \cdot 
2(\alpha_{kr}\alpha_{ip}-\alpha_{kp}\alpha_{ir}) \\
& = 4R_{ijpq}R_{kirp}.
\end{align*}
We can similarly calculate the remaining three parts, and by summing up them, we know that  
the above coefficient is equal to 
\begin{align*}
& 4(R_{ijpq}R_{kirp} - R_{ijrp}R_{kipq} - R_{kipq}R_{ijrp} + R_{kirp}R_{ijpq}) \\
= & 8(R_{ijpq}R_{kirp} - R_{ijrp}R_{kipq}).
\end{align*}
On the other hand, from the inverse formula of the Gauss equation (Proposition~\ref{RA2} (2)  
and Remark~\ref{Rem_const} (1)), we have 
$$
\alpha_{ip} = \frac{1}{|A|} \begin{vmatrix}
R_{ijpq} & R_{ijpr} \\
R_{ikpq} & R_{ikpr} 
\end{vmatrix}.
$$
Hence, the coefficient of $x_jy_kz_qw_r$ is equal to $8|A|\alpha_{ip}=8|A|\alpha_i \alpha_p$.
In a similar way, we can calculate the remaining coefficients, and finally we know that the 
left hand side of the equality is equal to 
$$
8 |A| \begin{vmatrix} 
\alpha_i & x_i & y_i \\
\alpha_j & x_j & y_j \\
\alpha_k & x_k & y_k
\end{vmatrix}
\begin{vmatrix} 
\alpha_p & z_p & w_p \\
\alpha_q & z_q & w_q \\
\alpha_r & z_r & w_r
\end{vmatrix}
= 8 |A| \begin{vmatrix} 
\alpha_1 & x_1 & y_1 \\
\alpha_2 & x_2 & y_2 \\
\alpha_3 & x_3 & y_3
\end{vmatrix}
\begin{vmatrix} 
\alpha_1 & z_1 & w_1 \\
\alpha_2 & z_2 & w_2 \\
\alpha_3 & z_3 & w_3
\end{vmatrix}.
$$

The case $|A|=0$ can be treated in the same way as in the proof of Proposition~\ref{RA2} (1). 
For both sides of the equality, the coefficients of $x_jy_kz_qw_r$ are polynomials of 
$\alpha_{ij}$ with degree four.
Both coefficients coincide on the open dense subset consisting of non-singular symmetric 
$3$-matrices $A=(\alpha_{ij})$, and hence it follows that they coincide identically on the 
whole space, which completes the proof of Lemma~\ref{Lem_det_id}.
\end{proof}

\medskip

{\it Proof of Theorem~\ref{WTR} (2).}
We substitute the Gauss equation and the derived Gauss equation into $RS(x_1,x_2,x_3)$ 
symbolically.
As for the Gauss equation, we have 
\begin{align*}
R_{ijkl} & = R_{ij}R_{kl} \\
& = \alpha_{ik}\alpha_{jl}-\alpha_{il}\alpha_{jk}.
\end{align*}
Here, by using the symbols $\alpha_i^{(1)}$ and $\alpha_i^{(2)}$, we decompose the last 
expression in the form 
\begin{align*}
& \frac{1}{\,2\,}(\alpha_i^{(1)}\alpha_k^{(1)}\alpha_j^{(2)}\alpha_l^{(2)}
+\alpha_i^{(2)}\alpha_k^{(2)}\alpha_j^{(1)}\alpha_l^{(1)}
-\alpha_i^{(1)}\alpha_l^{(1)}\alpha_j^{(2)}\alpha_k^{(2)}
-\alpha_i^{(2)}\alpha_l^{(2)}\alpha_j^{(1)}\alpha_k^{(1)}) \\
= & \frac{1}{\,2\,} 
(\alpha_i^{(1)} \alpha_j^{(2)} - \alpha_j^{(1)}\alpha_i^{(2)})
(\alpha_k^{(1)} \alpha_l^{(2)} - \alpha_l^{(1)}\alpha_k^{(2)}) \\
= & \frac{1}{\sqrt{2}} (\alpha^{(1)} \wedge \alpha^{(2)})_{ij} \cdot 
\frac{1}{\sqrt{2}} (\alpha^{(1)} \wedge \alpha^{(2)})_{kl}.
\end{align*}
Thus, if we substitute 
$$
R_{ij} = \frac{1}{\sqrt{2}} (\alpha^{(1)} \wedge \alpha^{(2)})_{ij}
$$
symbolically into $RS(x_1,x_2,x_3)$, then by following the above calculation reversely, 
we know that the curvature $R_{ijkl} = R_{ij}R_{kl}$ is automatically replaced by 
$$
R_{ijkl} = \frac{1}{\sqrt{2}} (\alpha^{(1)} \wedge \alpha^{(2)})_{ij} \cdot 
\frac{1}{\sqrt{2}} (\alpha^{(1)} \wedge \alpha^{(2)})_{kl}
= \alpha_{ik}\alpha_{jl}-\alpha_{il}\alpha_{jk}
$$
in the final step of calculations.
Hence we have only to put 
$$
R_{ij} = \frac{1}{\sqrt{2}} (\alpha^{(1)} \wedge \alpha^{(2)})_{ij}
$$
into $RS(x_1,x_2,x_3)$ in order to substitute the Gauss equation.
(Remark that the equality 
$$
\frac{1}{\sqrt{2}} (\alpha^{(1)} \wedge \alpha^{(2)})_{ij} 
= - \frac{1}{\sqrt{2}} (\alpha^{(1)} \wedge \alpha^{(2)})_{ji}
$$
holds.
Hence, the property $R_{ij} = -R_{ji}$ is preserved, and we are allowed the above setting.)
As for $\overline{R}_{ij}$, we must use the symbols different from $\alpha_i^{(1)}$, 
$\alpha_i^{(2)}$, since both $R_{ij}$ and $\overline{R}_{ij}$ appear simultaneously 
in $RS(x_1,x_2,x_3)$.
Hence, we put 
$$
\overline{R}_{ij} = \frac{1}{\sqrt{2}} (\alpha^{(3)} \wedge \alpha^{(4)})_{ij}
$$
in the following calculations.
Substitution of the Gauss equation is carried out in this way.

Next, concerning the derived Gauss equation 
$$
S_{ijklm} = \beta_{ikm}\alpha_{jl}+\alpha_{ik}\beta_{jlm}-\beta_{ilm}\alpha_{jk}-
\alpha_{il}\beta_{jkm},
$$
we put 
$$
\beta_{ijk} = \beta_i \beta_j \beta_k, \qquad \alpha_{ij} = \alpha_i^{(5)} \alpha_j^{(5)} 
$$
symbolically.
We remark that $\beta_{ijk}$ is a symmetric $3$-tensor on account of the Codazzi 
equation, and hence we may freely change the order of products of symbols $\beta_i$, 
$\beta_j$, $\beta_k$.
Then the derived Gauss equation is decomposed as 
\begin{align*}
S_{ijklm} & = S_{ij}S_{kl}S_m \\
& = \beta_{ikm}\alpha_{jl}+\alpha_{ik}\beta_{jlm}-\beta_{ilm}\alpha_{jk}-
\alpha_{il}\beta_{jkm} \\
& = \beta_i \beta_k \beta_m \alpha_j^{(5)} \alpha_l^{(5)} 
+ \alpha_i^{(5)} \alpha_k^{(5)} \beta_j \beta_l \beta_m \\
& \hspace{1.5cm} - \beta_i \beta_l \beta_m \alpha_j^{(5)} \alpha_k^{(5)} 
- \alpha_i^{(5)} \alpha_l^{(5)} \beta_j \beta_k \beta_m \\ 
& = (\alpha_i^{(5)} \beta_j - \alpha_j^{(5)} \beta_i)
(\alpha_k^{(5)} \beta_l - \alpha_l^{(5)} \beta_k) \beta_m \\
& = (\alpha^{(5)} \wedge \beta)_{ij}(\alpha^{(5)} \wedge \beta)_{kl} \beta_m.
\end{align*}
Hence it suffices to substitute 
$$
S_{ij} = (\alpha^{(5)} \wedge \beta)_{ij} 
\quad{\rm and} \quad
S_m=\beta_m
$$
into $RS(x_1,x_2,x_3)$ by the same reason as $R_{ijkl}$.
Note that the properties $S_{ij} = -S_{ji}$ and $S_{ij}S_k+S_{jk}S_i+S_{ki}S_j=0$ 
hold in this setting.
Then we have 
\begin{align*}
& RS(x_1,x_2,x_3) = 
\begin{vmatrix} 
R_{12} & \overline{R}_{12} & S_{12} \\
R_{13} & \overline{R}_{13} & S_{13} \\
R_{23} & \overline{R}_{23} & S_{23} \\
\end{vmatrix}
\begin{vmatrix} 
R_{12} & \overline{R}_{12} & (x \wedge S)_{12} \\
R_{13} & \overline{R}_{13} & (x \wedge S)_{13} \\
R_{23} & \overline{R}_{23} & (x \wedge S)_{23} \\
\end{vmatrix}
(x_1S_{23}  -x_2S_{13}+x_3S_{12}) \\
& \rule{0cm}{1.1cm} = 
\frac{1}{\,4\,} \begin{vmatrix} 
(\alpha^{(1)} \wedge \alpha^{(2)})_{12} & (\alpha^{(3)} \wedge \alpha^{(4)})_{12} 
& (\alpha^{(5)} \wedge \beta)_{12} \\
\rule{0cm}{0.4cm} (\alpha^{(1)} \wedge \alpha^{(2)})_{13} & (\alpha^{(3)} 
\wedge \alpha^{(4)})_{13} & (\alpha^{(5)} \wedge \beta)_{13} \\
\rule{0cm}{0.4cm} (\alpha^{(1)} \wedge \alpha^{(2)})_{23} & (\alpha^{(3)} 
\wedge \alpha^{(4)})_{23} & (\alpha^{(5)} \wedge \beta)_{23} \\
\end{vmatrix} \\
& \hspace{1.2cm} \times \begin{vmatrix} 
(\alpha^{(1)} \wedge \alpha^{(2)})_{12} & (\alpha^{(3)} \wedge \alpha^{(4)})_{12} 
& (x \wedge \beta)_{12} \\
\rule{0cm}{0.4cm} (\alpha^{(1)} \wedge \alpha^{(2)})_{13} & (\alpha^{(3)} 
\wedge \alpha^{(4)})_{13} & (x \wedge \beta)_{13} \\
\rule{0cm}{0.4cm} (\alpha^{(1)} \wedge \alpha^{(2)})_{23} & (\alpha^{(3)} 
\wedge \alpha^{(4)})_{23} & (x \wedge \beta)_{23} \\
\end{vmatrix}
\begin{vmatrix}
\alpha_1^{(5)} & \beta_1 & x_1 \\
\rule{0cm}{0.49cm} \alpha_2^{(5)} & \beta_2 & x_2 \\
\rule{0cm}{0.49cm} \alpha_3^{(5)} & \beta_3 & x_3 
\end{vmatrix}.
\end{align*}
By Lemma~\ref{Lem_det_id}, it is equal to 
\begin{align*}
& 2|A| \begin{vmatrix}
\alpha_1 & \alpha_1^{(5)} & \beta_1 \\
\rule{0cm}{0.49cm} \alpha_2 & \alpha_2^{(5)} & \beta_2 \\
\rule{0cm}{0.49cm} \alpha_3 & \alpha_3^{(5)} & \beta_3
\end{vmatrix}
\begin{vmatrix}
\rule{0cm}{0.5cm} \alpha_1 & x_1 & \beta_1 \\
\rule{0cm}{0.5cm} \alpha_2 & x_2 & \beta_2 \\
\rule{0cm}{0.5cm} \alpha_3 & x_3 & \beta_3
\end{vmatrix}
\begin{vmatrix}
\alpha_1^{(5)} & \beta_1 & x_1 \\
\rule{0cm}{0.49cm} \alpha_2^{(5)} & \beta_2 & x_2 \\
\rule{0cm}{0.49cm} \alpha_3^{(5)} & \beta_3 & x_3
\end{vmatrix} \\
& =-2|A| \begin{vmatrix}
\alpha_1 & \alpha_1^{(5)} & \beta_1 \\
\rule{0cm}{0.49cm} \alpha_2 & \alpha_2^{(5)} & \beta_2 \\
\rule{0cm}{0.49cm} \alpha_3 & \alpha_3^{(5)} & \beta_3
\end{vmatrix}
\begin{vmatrix}
\rule{0cm}{0.5cm} \alpha_1 & \beta_1 & x_1 \\
\rule{0cm}{0.5cm} \alpha_2 & \beta_2 & x_2 \\
\rule{0cm}{0.5cm} \alpha_3 & \beta_3 & x_3
\end{vmatrix}
\begin{vmatrix}
\alpha_1^{(5)} & \beta_1 & x_1 \\
\rule{0cm}{0.49cm} \alpha_2^{(5)} & \beta_2 & x_2 \\
\rule{0cm}{0.49cm} \alpha_3^{(5)} & \beta_3 & x_3
\end{vmatrix}.
\end{align*}
In the last expression, we may change the role of the symbols $\alpha_i$ and 
$\alpha_i^{(5)}$, i.e., we may use the symbol $\alpha_i^{(5)}$ in Lemma~\ref{Lem_det_id} 
instead of $\alpha_i$, and also we may put $S_{ij} = (\alpha \wedge \beta)_{ij}$ instead of 
$(\alpha^{(5)} \wedge \beta)_{ij}$ in substituting to $RS(x_1,x_2,x_3)$, since we 
restore both $\alpha_i \alpha_j$ and $\alpha_i^{(5)} \alpha_j^{(5)}$ to the same polynomial 
$\alpha_{ij}$ at the final stage.
But by this exchange, we obtain two expressions with different signs on account of the term 
$$
\begin{vmatrix}
\alpha_1 & \alpha_1^{(5)} & \beta_1 \\
\rule{0cm}{0.49cm} \alpha_2 & \alpha_2^{(5)} & \beta_2 \\
\rule{0cm}{0.49cm} \alpha_3 & \alpha_3^{(5)} & \beta_3
\end{vmatrix}, 
$$
which implies that they are zero as polynomials.
Hence we have $RS(x_1,x_2,x_3)=0$ after substituting the Gauss equation and the derived 
Gauss equation.
This completes the proof of Theorem~\ref{WTR} (2).
\hfill{$\Box$}

\begin{rem}\label{subst_exchange}
(1) 
In this proof we used the property that $\beta_{ijk}$ is symmetric.
This condition is essential, i.e., without the Codazzi equation 
$\beta_{ijk} = \beta_{ikj}$, the equality $RS(x_1,x_2,x_3)=0$ does not hold in general.

(2) 
We remark that polynomials may be decomposed into symbols in several ways.
For example, during the above proof, we may decompose the polynomial 
$\alpha_{ik} \alpha_{jl}$ as 
$$
t \alpha_i^{(1)}\alpha_k^{(1)}\alpha_j^{(2)}\alpha_l^{(2)}
+(1-t) \alpha_i^{(2)}\alpha_k^{(2)}\alpha_j^{(1)}\alpha_l^{(1)}
$$
by using  some parameter $t$, instead of 
$$
\alpha_{ik} \alpha_{jl}= \frac{1}{\,2\,}(\alpha_i^{(1)}\alpha_k^{(1)}\alpha_j^{(2)}\alpha_l^{(2)}
+\alpha_i^{(2)}\alpha_k^{(2)}\alpha_j^{(1)}\alpha_l^{(1)}).
$$
In decomposing polynomials to symbols, we must find the best decomposition fitted to 
each situation.

(3) 
In the final step of the above proof, we prove the vanishing of $RS(x_1,x_2,x_3)$ by 
exchanging the role of $\alpha_i$ and $\alpha_i^{(5)}$.
This $\lq\lq$exchanging method'' is a powerful tool to show the vanishing of a 
polynomial, which is classically well known.
See for example, Gurevich \cite{Gu}, p.201.
Its simplest toy model is 
$$
\begin{vmatrix} \alpha_1^{(1)} & \alpha_1^{(2)} \\
\rule{0cm}{0.5cm} \alpha_2^{(1)} & \alpha_2^{(2)} \end{vmatrix} \alpha_1^{(1)} 
\alpha_1^{(2)} =0.
$$
Applying the exchanging method, we can see this result at a first glance without any 
calculation.
As another example, consider a symmetric $3$-tensor $a_{ijk}$.
We decompose it as $a_{ijk} = a_i^{(1)}a_j^{(1)}a_k^{(1)} 
= a_i^{(2)} a_j^{(2)} a_k^{(2)} = a_i^{(3)}  a_j^{(3)} a_k^{(3)}$ 
symbolically.
Then, by the exchanging method,  we can easily see the following equality. 
$$
\begin{vmatrix}
a_1^{(1)} & a_1^{(2)} & a_1^{(3)} \\
\rule{0cm}{0.5cm} a_2^{(1)} & a_2^{(2)} & a_2^{(3)} \\
\rule{0cm}{0.5cm} a_3^{(1)} & a_3^{(2)} & a_3^{(3)}
\end{vmatrix}^3 =0.
$$

We give more examples.
By restricting the range of indices $i,j,k=1,2$, we may consider $a_{ijk}$ as a symmetric 
$3$-tensor on $\mathbb{R}^2$, i.e., $a_{ijk} \in S^3(\mathbb{R}^2)$.
We decompose it as $a_{ijk} = a_i^{(p)} a_j^{(p)} a_k^{(p)}$ ($p=1,2,3,4$) in terms of symbols. 
Then we can easily see the equality 
$$
\begin{vmatrix} a_1^{(1)} & a_1^{(2)} \\ 
\rule{0cm}{0.5cm} a_2^{(1)} & a_2^{(2)} \end{vmatrix}
\begin{vmatrix} a_1^{(1)} & a_1^{(3)} \\ 
\rule{0cm}{0.5cm} a_2^{(1)} & a_2^{(3)} 
\end{vmatrix}
\begin{vmatrix} a_1^{(1)} & a_1^{(4)} \\ 
\rule{0cm}{0.5cm} a_2^{(1)} & a_2^{(4)} 
\end{vmatrix}
\begin{vmatrix} a_1^{(2)} & a_1^{(3)} \\ 
\rule{0cm}{0.5cm} a_2^{(2)} & a_2^{(3)} 
\end{vmatrix}
\begin{vmatrix} a_1^{(2)} & a_1^{(4)} \\ 
\rule{0cm}{0.5cm} a_2^{(2)} & a_2^{(4)} 
\end{vmatrix}
\begin{vmatrix} a_1^{(3)} & a_1^{(4)} \\ 
\rule{0cm}{0.5cm} a_2^{(3)} & a_2^{(4)} 
\end{vmatrix} =0
$$
by the exchanging method.
On the other hand, the product 
$$
-\frac{1}{\,2\,} 
\begin{vmatrix} a_1^{(1)} & a_1^{(2)} \\ 
\rule{0cm}{0.5cm} a_2^{(1)} & a_2^{(2)} 
\end{vmatrix}^2
\begin{vmatrix} a_1^{(3)} & a_1^{(4)} \\ 
\rule{0cm}{0.5cm} a_2^{(3)} & a_2^{(4)} 
\end{vmatrix}^2
\begin{vmatrix} a_1^{(1)} & a_1^{(3)} \\ 
\rule{0cm}{0.5cm} a_2^{(1)} & a_2^{(3)} 
\end{vmatrix}
\begin{vmatrix} a_1^{(2)} & a_1^{(4)} \\ 
\rule{0cm}{0.5cm} a_2^{(2)} & a_2^{(4)} 
\end{vmatrix}
$$
is non-zero, and it expresses the discriminant 
$$
a_{111}{}^2a_{222}{}^2-6a_{111}a_{112}a_{122}a_{222}+4a_{111}a_{122}{}^3
+4a_{112}{}^3a_{222}-3a_{112}{}^2a_{122}{}^2
$$
of the cubic equation $a_{111}x^3+3a_{112}x^2+3a_{122}x+a_{222}=0$ 
(cf. Olver \cite{O}, p.117).

As another example, we consider $a_{ijkl} \in S^4(\mathbb{R}^2)$, and put 
$a_{ijkl} = a_i^{(p)}a_j^{(p)}a_k^{(p)}a_l^{(p)}$ for $p=1,2,3$.
Then, 
\begin{align*}
& \frac{1}{\,6\,} 
\begin{vmatrix} a_1^{(1)} & a_1^{(2)} \\ 
\rule{0cm}{0.5cm} a_2^{(1)} & a_2^{(2)} 
\end{vmatrix}^2
\begin{vmatrix} a_1^{(1)} & a_1^{(3)} \\ 
\rule{0cm}{0.5cm} a_2^{(1)} & a_2^{(3)} 
\end{vmatrix}^2
\begin{vmatrix} a_1^{(2)} & a_1^{(3)} \\ 
\rule{0cm}{0.5cm} a_2^{(2)} & a_2^{(3)} 
\end{vmatrix}^2 \\
= & a_{1111}a_{1122}a_{2222}+2a_{1112}a_{1122}a_{1222}-a_{1111}a_{1222}{}^2
-a_{1112}{}^2a_{2222}-a_{1122}{}^3 \\
= & \begin{vmatrix}
a_{1111} & a_{1112} & a_{1122} \\
a_{1112} & a_{1122} & a_{1222} \\
a_{1122} & a_{1222} & a_{2222}
\end{vmatrix}
\end{align*}
gives an invariant, called the catalecticant of a binary quartic $a_{1111}x^4+4a_{1112}x^3
+6a_{1122}x^2+4a_{1222}x+a_{2222}$ (cf. Dolgachev \cite{Dol}, p.10).
Maybe, these are the most classical invariants expressed in terms of symbols.

Classically, symbolic method is a tool to express the invariants (and covariants) in compact 
simple forms.
In this paper, we further use it as a powerful tool to verify several identities on tensors, which 
shows a new role of symbolic method.

(4) As stated in Remark~\ref{Rem_const} (2), we consider the case where the ambient 
space is the $4$-dimensional space of constant curvature $c$.
By taking the covariant derivative of the Gauss equation 
$$
R_{ijkl} - c(\delta_{ik}\delta_{jl}-\delta_{il}\delta_{jk}) = 
\alpha_{ik}\alpha_{jl}-\alpha_{il}\alpha_{jk},
$$
we have the completely same equation 
$$
S_{ijklm} = \beta_{ikm}\alpha_{jl}+\alpha_{ik}\beta_{jlm}-\beta_{ilm}\alpha_{jk}-
\alpha_{il}\beta_{jkm}, 
$$
since the space of constant curvature is locally symmetric.
Thus, by replacing $R_{ijkl}$ in $r_1 \sim r_6$ by $R_{ijkl} - c(\delta_{ik}\delta_{jl}
-\delta_{il}\delta_{jk})$, we obtain Rivertz's six equalities for the space of 
constant curvature.
In addition, we can extend our result (Theorem~\ref{MainTh}) to the space of constant 
curvature (see Corollary~\ref{MainCor}).
\end{rem}

\bigskip

\section{Proof of the main theorem}

\medskip

In this section we give a proof of the main theorem (Theorem~\ref{MainTh}).
We already re-established Theorem~\ref{WTR}, and hence we have only to prove the 
$\lq\lq$if'' part of the main theorem, i.e., our remaining task is to prove the following 
proposition.

\begin{prop}\label{Prop_R}
Let $(M,g)$ be a $3$-dimensional simply connected Riemannian manifold.
Assume that $|\widetilde{R}|>0$ and Rivertz's six equalities $r_1=\cdots = r_6 = 0$ hold.
Then, there exists an isometric embedding $f: (M,g) \longrightarrow \mathbb{R}^4$.
\end{prop}

\subsection{Solution of the derived Gauss equation}
We consider the Gauss equation as an abstract algebraic equation defined 
on each tangent space of $M$.
Then, under the assumption $|\widetilde{R}|>0$, it admits a unique solution up to sign, 
as stated in Proposition~\ref{RA2} (2) and Remark~\ref{Rem_const} (1).
We fix the sign of the solution once for all.
For example, we set $\varepsilon =1$ in Proposition~\ref{RA2} (2).
Then such defined $\alpha$ gives a solution of the Gauss equation at each point of $M$.
We here denote its solution by $\alpha_{ij}^0$.
We emphasize that $\alpha_{ij}^0$ is not a second fundamental form, since we are unaware 
of the existence of local isometric embedding $f: (M,g) 
\longrightarrow \mathbb{R}^4$ at this 
stage.
It is merely a smooth symmetric $2$-tensor field on $M$, satisfying the pointwise Gauss 
equation.
We do not know whether $\alpha_{ij}^0$ satisfies the Codazzi equation or not at this time.

We put 
$$
A_0=\begin{pmatrix}
\alpha_{11}^0 & \alpha_{12}^0 & \alpha_{13}^0 \\
\rule{0cm}{0.4cm} \alpha_{12}^0 & \alpha_{22}^0 & \alpha_{23}^0 \\
\rule{0cm}{0.4cm} \alpha_{13}^0 & \alpha_{23}^0 & \alpha_{33}^0 \\
\end{pmatrix}.
$$
Note that from the assumption $|\widetilde{R}|>0$ we have $|A_0| \neq 0$ 
(cf. Proposition~\ref{RA2} (1)).
By changing the sign of $\alpha_{ij}^0$ (i.e., the sign of $\varepsilon$ in 
Proposition~\ref{RA2} (2)) if necessary, we may assume that $|A_0|>0$.
In the following argument, we fix such a solution and denote it 
$\alpha_{ij}^0$ again.

Since $\alpha_{ij}^0$ satisfies the Gauss equation 
$$
R_{ijkl}=\alpha_{ik}^0\alpha_{jl}^0-\alpha_{il}^0\alpha_{jk}^0
$$
at every point of $M$, we have the derived Gauss equation for $\alpha_{ij}^0$
$$
S_{ijklm} = \beta_{ikm}^0\alpha_{jl}^0+\alpha_{ik}^0\beta_{jlm}^0
-\beta_{ilm}^0\alpha_{jk}^0-\alpha_{il}^0\beta_{jkm}^0
$$
by taking the covariant derivative of the Gauss equation.
Here, we put $\beta_{ijk}^0 = (\nabla_{X_k}\alpha^0)(X_i,X_j)$.
Then, to prove Proposition~\ref{Prop_R}, we have only to show the following 
proposition on account of the fundamental theorem of hypersurfaces (Theorem~\ref{FT}).

\begin{prop}\label{Prop_Codazzi_1}
Under the above notations, $\alpha_{ij}^0$ satisfies the Codazzi equation 
$\beta_{ijk}^0=\beta_{ikj}^0$ if the conditions $|\widetilde{R}|>0$ and 
$r_1 = \cdots = r_6=0$ hold. 
\end{prop}

\noindent
We remark that the equality $\beta_{ijk}^0=\beta_{jik}^0$ is automatically satisfied 
since $\alpha_{ij}^0$ is symmetric.
It should be also remarked that we have 
\begin{align}
& S_{ijklm} + S_{ijlmk} + S_{ijmkl}  \notag \\
& \hspace{0.3cm} = \alpha_{ik}^0(\beta_{jlm}^0-\beta_{jml}^0) 
+\alpha_{il}^0(\beta_{jmk}^0-\beta_{jkm}^0) 
+\alpha_{im}^0(\beta_{jkl}^0-\beta_{jlk}^0) \label{secondB} \\
& \hspace{0.5cm} +\alpha_{jk}^0(\beta_{iml}^0-\beta_{ilm}^0) 
+\alpha_{jl}^0(\beta_{ikm}^0-\beta_{imk}^0)
+\alpha_{jm}^0(\beta_{ilk}^0-\beta_{ikl}^0)=0 \notag
\end{align}
from the derived Gauss equation and the second Bianchi identity.
But in case ${\rm dim}\:M = 3$, we cannot show the Codazzi equation only from this 
equality.
(See Remark~\ref{G-C} below.)

Now, we introduce the following abstract system of linear equations on 
a new tensor $\overline{\beta}_{ijm}$  
\begin{align}
\overline{\beta}_{ikm}\alpha_{jl}^0+\alpha_{ik}^0\overline{\beta}_{jlm}
-\overline{\beta}_{ilm}\alpha_{jk}^0-\alpha_{il}^0\overline{\beta}_{jkm}
= \overline{S}_{ijklm}, \label{moddG}
\end{align}
where the unknown tensor $\overline{\beta}_{ijm}$ satisfies the relation 
$\overline{\beta}_{ijm}= \overline{\beta}_{jim}$, and the target tensor 
$\overline{S}_{ijklm}$ satisfies the relation 
$$
\overline{S}_{ijklm} = -\overline{S}_{jiklm} = -\overline{S}_{ijlkm} 
= \overline{S}_{klijm}.
$$
(For a reason that will be explained below, we here introduce different letters 
$\overline{\beta}_{ijm}$ and $\overline{S}_{ijklm}$, instead of $\beta_{ijm}^0$ and $S_{ijklm}$.
Remark that we do not assume the second Bianchi identity on the tensor 
$\overline{S}_{ijklm}$.)

Under these preparations, we now divide the proof of Proposition~\ref{Prop_Codazzi_1} in 
the following three steps:

\bigskip

\noindent
(A) We first show that the system of equations (\ref{moddG}) admits a unique solution 
$\overline{\beta}_{ijm}$ under the condition $|A_0|>0$.
Next, we express the solution $\overline{\beta}_{ijm}$ explicitly in terms of $\alpha_{ij}^0$ 
and $\overline{S}_{ijklm}$.
By setting $\overline{\beta}_{ijm} = \beta_{ijm}^0$ and $\overline{S}_{ijklm} 
= S_{ijklm}$, we obtain a formula on $\beta_{ijm}^0$, expressed in terms of 
$\alpha_{ij}^0$ and $S_{ijklm}$ (Proposition~\ref{Prop_beta}).

\medskip

\noindent
(B)
We introduce a new quadratic polynomial $H_0(x_1,x_2,x_3)$ on $x_1,x_2,x_3$, 
whose coefficients are bilinear forms of $\alpha_{ij}^0$ and $S_{ijklm}$.
We prove that the condition 
$RS(x_1,x_2,x_3)=0$ is equivalent to $H_0(x_1,x_2,x_3)=0$ through the Gauss 
equation in case $|A_0| \neq 0$ (Proposition~\ref{Prop_Ralpha}).

\medskip

\noindent
(C) We prove the equality $\beta_{ipq}^0 - \beta_{iqp}^0=0$ 
under the assumptions $|A_0|\neq0$ and $H_0(x_1,x_2,x_3)=0$, by using 
the formula obtained in step (A) (Proposition~\ref{Prop_Codazzi_2}).

\bigskip

\begin{rem}\label{G-C}
Historically, the equation (\ref{secondB}) plays a quite important role in local isometric 
embedding theory.
In the case ${\rm dim}\:M \geq 4$, we can show that the Codazzi equation 
follows from the 
equation (\ref{secondB}) under some generic condition on $\alpha_{ij}^0$, by considering it 
as a system of linear equations on $\beta_{ipq}^0-\beta_{iqp}^0$.
In this case, the system (\ref{secondB}) contains many independent conditions on 
$\beta_{ipq}^0-\beta_{iqp}^0$, and we can show the above fact.
This is nothing but classical Thomas's theorem (Theorem~\ref{Thomas}).
But in the case ${\rm dim}\:M=3$, the system (\ref{secondB}) contains 
less information, and the Codazzi equation does not follow from it.
Thus, we must take another way to show the Codazzi equation.
\end{rem}

\medskip

Now we first prove the fact in step (A).
The system of equations (\ref{moddG}) can be divided into three independent parts 
according as the value of $m=1,2,3$, and three divided parts are essentially 
the same system of equations, 
i.e., we have to solve the equivalent system of equations three times.
To avoid such a threefold calculations, we prepare the following lemma.
To prove this lemma, we constantly use the symbolic method as before.

\begin{lem}\label{Lem_dG}
Let $T_{ijkl}$, $\gamma_{ij}$ be tensors $(i,j,k,l=1,2,3)$, which satisfy 
$$
T_{ijkl}=-T_{jikl}=-T_{ijlk}=T_{klij}, \quad \gamma_{ij}=\gamma_{ji}.
$$
Then the system of linear equations on $\gamma_{ij}$
$$
\gamma_{ik}\alpha_{jl}^0+\alpha_{ik}^0\gamma_{jl}
-\gamma_{il}\alpha_{jk}^0-\alpha_{il}^0\gamma_{jk}=T_{ijkl}
$$
admits a unique solution.
The solution is given by $\gamma_{ij} = \gamma^1_{ij}-\gamma^2_{ij}$, 
where $\gamma^1_{ij}$ and $\gamma^2_{ij}$ are symbolically expressed as 
$$
\gamma_{ij}^1= \gamma_i^1 \gamma_j^1, \qquad 
\gamma_{ij}^2= \gamma_i^2 \gamma_j^2
$$
with 
$$
\gamma_i^1= \frac{1}{\sqrt{2 |A_0|}}
\begin{vmatrix}
\alpha_1^0 & \overline{\alpha}_1^{\,0} & T_{1i} \\ 
\rule{0cm}{0.4cm} \alpha_2^0 & \overline{\alpha}_2^{\,0} & T_{2i} \\ 
\rule{0cm}{0.4cm} \alpha_3^0 & \overline{\alpha}_3^{\,0} & T_{3i} 
\end{vmatrix}, \quad 
\gamma_i^2= \frac{1}{\sqrt{2 |A_0|}}\,
\alpha_i^0 (\overline{\alpha}_1^{\,0} T_{23} 
-\overline{\alpha}_2^{\,0} T_{13}+\overline{\alpha}_3^{\,0} T_{12}).
$$ 
\end{lem}

\noindent
Concerning the symbols, we promise that $\alpha_i^0 \alpha_j^0 = \overline{\alpha}_i^{\,0} 
\,\overline{\alpha}_j^{\,0} =\alpha_{ij}^{\,0}$, $T_{ij} = -T_{ji}$ and $T_{ij}T_{kl}=T_{ijkl}$.
We can express $\gamma_{ij}^1$ and $\gamma_{ij}^2$ explicitly as 
\begin{align*}
\gamma^1_{ij} &= \frac{1}{2|A_0|} 
\begin{vmatrix}
\alpha_1^0 & \overline{\alpha}_1^{\,0} & T_{1i} \\ 
\rule{0cm}{0.4cm} \alpha_2^0 & \overline{\alpha}_2^{\,0} & T_{2i} \\ 
\rule{0cm}{0.4cm} \alpha_3^0 & \overline{\alpha}_3^{\,0} & T_{3i} 
\end{vmatrix} 
\begin{vmatrix}
\alpha_1^0 & \overline{\alpha}_1^{\,0} & T_{1j} \\ 
\rule{0cm}{0.4cm} \alpha_2^0 & \overline{\alpha}_2^{\,0} & T_{2j} \\ 
\rule{0cm}{0.4cm} \alpha_3^0 & \overline{\alpha}_3^{\,0} & T_{3j} 
\end{vmatrix}, \\
\gamma^2_{ij} & 
=\frac{1}{2|A_0|}\,\alpha_{ij}^0 \,(\overline{\alpha}_1^{\,0} T_{23} 
-\overline{\alpha}_2^{\,0} T_{13}+\overline{\alpha}_3^{\,0} T_{12})^2.
\end{align*}
For example, we have 
\begin{align*}
\gamma_{12}^1 & = \frac{1}{2|A_0|}
(\alpha_1^0 \,\overline{\alpha}_2^{\,0}T_{31}+\alpha_3^0 \,\overline{\alpha}_1^{\,0}T_{21}
-\alpha_1^0 \,\overline{\alpha}_3^{\,0}T_{21}-\alpha_2^0 \,\overline{\alpha}_1^{\,0}T_{31}) \\
& \hspace{1cm} \times
(\alpha_1^0 \,\overline{\alpha}_2^{\,0}T_{32}+\alpha_2^0 \,\overline{\alpha}_3^{\,0}T_{12}
-\alpha_2^0 \,\overline{\alpha}_1^{\,0}T_{32}-\alpha_3^0 \,\overline{\alpha}_2^{\,0}T_{12}) \\
& = \frac{1}{|A_0|}
\{ (\alpha_{11}^0\alpha_{22}^0-\alpha_{12}^0{}^2)T_{13}T_{23}
+ (\alpha_{13}^0\alpha_{22}^0-\alpha_{12}^0\alpha_{23}^0)T_{12}T_{13} \\
& \hspace{1cm} + (\alpha_{12}^0\alpha_{13}^0-\alpha_{11}^0\alpha_{23}^0)T_{12}T_{23}
+ (\alpha_{12}^0\alpha_{33}^0-\alpha_{13}^0\alpha_{23}^0)T_{12}{}^2 \} \\
& = \frac{1}{|A_0|}
\{ (\alpha_{11}^0\alpha_{22}^0-\alpha_{12}^0{}^2)T_{1323}
+ (\alpha_{13}^0\alpha_{22}^0-\alpha_{12}^0\alpha_{23}^0)T_{1213} \\
& \hspace{1cm} + (\alpha_{12}^0\alpha_{13}^0-\alpha_{11}^0\alpha_{23}^0)T_{1223}
+ (\alpha_{12}^0\alpha_{33}^0-\alpha_{13}^0\alpha_{23}^0)T_{1212} \},
\end{align*}
and
\begin{align*}
\gamma_{12}^2 & = \frac{1}{2|A_0|} \alpha_{12}^0
(\overline{\alpha}_1^{\,0}T_{23}-\overline{\alpha}_2^{\,0}T_{13}
+\overline{\alpha}_3^{\,0}T_{12}) (\overline{\alpha}_1^{\,0}T_{23}
-\overline{\alpha}_2^{\,0}T_{13}+\overline{\alpha}_3^{\,0}T_{12}) \\
& = \frac{1}{2|A_0|} \alpha_{12}^0
(\alpha_{11}^0T_{23}{}^2-2\alpha_{12}^0T_{13}T_{23}+2\alpha_{13}^0T_{12}T_{23} \\
& \hspace{2.1cm} +\alpha_{22}^0T_{13}{}^2-2\alpha_{23}^0T_{12}T_{13}
+\alpha_{33}^0T_{12}{}^2) \\
& = \frac{1}{2|A_0|} \alpha_{12}^0
(\alpha_{11}^0T_{2323}-2\alpha_{12}^0T_{1323}+2\alpha_{13}^0T_{1223} \\
& \hspace{2.1cm} +\alpha_{22}^0T_{1313}-2\alpha_{23}^0T_{1213}
+\alpha_{33}^0T_{1212}).
\end{align*}
The solution $\gamma_{12}$ is obtained by taking the difference, which becomes rather a 
long and complicated expression, though its symbolic expression is relatively simple.

To prove Lemma~\ref{Lem_dG}, we prepare two identities.
We promise $\overline{\overline{\alpha}}_i^{\,0} \,\overline{\overline{\alpha}}_j^{\,0} 
= \alpha_{ij}^0$ in the following.

\begin{lem}\label{Lem_identity}
The following two identities hold:
$$
\begin{vmatrix}
\alpha_1^0 & \overline{\alpha}_1^{\,0} & \overline{\overline{\alpha}}_1^{\,0} \\ 
\rule{0cm}{0.5cm} \alpha_2^0 & \overline{\alpha}_2^{\,0} & 
\overline{\overline{\alpha}}_2^{\,0} \\ 
\rule{0cm}{0.5cm} \alpha_3^0 & \overline{\alpha}_3^{\,0} & 
\overline{\overline{\alpha}}_3^{\,0}
\end{vmatrix}^2 = 3!\, |A_0|. \hspace{7cm} \leqno{(1)}
$$
$$
\begin{vmatrix}
\alpha_1^0 & \overline{\alpha}_1^{\,0} & \overline{\overline{\alpha}}_1^{\,0} \\ 
\rule{0cm}{0.5cm} \alpha_2^0 & \overline{\alpha}_2^{\,0} & 
\overline{\overline{\alpha}}_2^{\,0} \\ 
\rule{0cm}{0.5cm} \alpha_3^0 & \overline{\alpha}_3^{\,0} & 
\overline{\overline{\alpha}}_3^{\,0}
\end{vmatrix}  
\begin{vmatrix}
\alpha_i^0 & \overline{\alpha}_i^{\,0} \\ 
\rule{0cm}{0.4cm} \alpha_j^0 & \overline{\alpha}_j^{\,0}
\end{vmatrix} 
(\overline{\overline{\alpha}}_1^{\,0} T_{23}
-\overline{\overline{\alpha}}_2^{\,0} T_{13}
+\overline{\overline{\alpha}}_3^{\,0} T_{12})T_{kl}
= 2 \, |A_0| \, T_{ijkl}. \leqno{(2)}
$$
\end{lem}

\begin{proof}
(1) 
The first identity immediately follows from the definition of the determinant.
In fact, the left hand side is equal to 
\begin{align*}
& \sum_{\sigma} {\rm sgn}\, \sigma \, 
\alpha_{\sigma(1)}^0 \overline{\alpha}_{\sigma(2)}^{\,0} 
\overline{\overline{\alpha}}_{\sigma(3)}^{\,0} \cdot
\sum_{\tau} {\rm sgn}\, \tau \, 
\alpha_{\tau(1)}^0 \overline{\alpha}_{\tau(2)}^{\,0} 
\overline{\overline{\alpha}}_{\tau(3)}^{\,0} \\
= & \sum_{\sigma, \tau} {\rm sgn}\, \sigma \tau \, 
\alpha_{\sigma(1)\tau(1)}^0 \alpha_{\sigma(2)\tau(2)}^0 
\alpha_{\sigma(3) \tau(3)}^0 \\
= & 3! \, \sum_{\tau} {\rm sgn}\, \tau \, 
\alpha_{1\tau(1)}^0 \alpha_{2\tau(2)}^0 \alpha_{3 \tau(3)}^0, 
\end{align*}
which is nothing but $3! \, |A_0|$.

(2) 
We have 
\begin{align*}
& \begin{vmatrix}
\alpha_1^0 & \overline{\alpha}_1^{\,0} & \overline{\overline{\alpha}}_1^{\,0} \\ 
\rule{0cm}{0.5cm} \alpha_2^0 & \overline{\alpha}_2^{\,0} & 
\overline{\overline{\alpha}}_2^{\,0} \\ 
\rule{0cm}{0.5cm} \alpha_3^0 & \overline{\alpha}_3^{\,0} & 
\overline{\overline{\alpha}}_3^{\,0}
\end{vmatrix}  
\begin{vmatrix}
\alpha_i^0 & \overline{\alpha}_i^{\,0} \\ 
\rule{0cm}{0.4cm} \alpha_j^0 & \overline{\alpha}_j^{\,0}
\end{vmatrix} 
\overline{\overline{\alpha}}_m^{\,0}
= \begin{vmatrix}
\alpha_1^0 & \overline{\alpha}_1^{\,0} & \overline{\overline{\alpha}}_1^{\,0} \\ 
\rule{0cm}{0.5cm} \alpha_2^0 & \overline{\alpha}_2^{\,0} & 
\overline{\overline{\alpha}}_2^{\,0} \\ 
\rule{0cm}{0.5cm} \alpha_3^0 & \overline{\alpha}_3^{\,0} & 
\overline{\overline{\alpha}}_3^{\,0}
\end{vmatrix} 
(\alpha_i^0 \,\overline{\alpha}_j^{\,0}\,\overline{\overline{\alpha}}_m^{\,0}
-\alpha_j^0 \,\overline{\alpha}_i^{\,0}\,\overline{\overline{\alpha}}_m^{\,0}) \\
& \rule{0cm}{1.2cm} = \begin{vmatrix}
\alpha_{1i}^0 & \alpha_{1j}^0 & \alpha_{1m}^0 \\ 
\rule{0cm}{0.5cm} \alpha_{2i}^0 & \alpha_{2j}^0 & \alpha_{2m}^0 \\ 
\rule{0cm}{0.5cm} \alpha_{3i}^0 & \alpha_{3j}^0 & \alpha_{3m}^0
\end{vmatrix}
- \begin{vmatrix}
\alpha_{1j}^0 & \alpha_{1i}^0 & \alpha_{1m}^0 \\ 
\rule{0cm}{0.5cm} \alpha_{2j}^0 & \alpha_{2i}^0 & \alpha_{2m}^0 \\ 
\rule{0cm}{0.5cm} \alpha_{3j}^0 & \alpha_{3i}^0 & \alpha_{3m}^0
\end{vmatrix} 
= 2 \,\begin{vmatrix}
\alpha_{1i}^0 & \alpha_{1j}^0 & \alpha_{1m}^0 \\ 
\rule{0cm}{0.5cm} \alpha_{2i}^0 & \alpha_{2j}^0 & \alpha_{2m}^0 \\ 
\rule{0cm}{0.5cm} \alpha_{3i}^0 & \alpha_{3j}^0 & \alpha_{3m}^0
\end{vmatrix}.
\end{align*}
Hence we have 
\begin{align*}
& \begin{vmatrix}
\alpha_1^0 & \overline{\alpha}_1^{\,0} & \overline{\overline{\alpha}}_1^{\,0} \\ 
\rule{0cm}{0.5cm} \alpha_2^0 & \overline{\alpha}_2^{\,0} & 
\overline{\overline{\alpha}}_2^{\,0} \\ 
\rule{0cm}{0.5cm} \alpha_3^0 & \overline{\alpha}_3^{\,0} & 
\overline{\overline{\alpha}}_3^{\,0}
\end{vmatrix}  
\begin{vmatrix}
\alpha_i^0 & \overline{\alpha}_i^{\,0} \\ 
\rule{0cm}{0.4cm} \alpha_j^0 & \overline{\alpha}_j^{\,0}
\end{vmatrix} 
\overline{\overline{\alpha}}_m^{\,0} 
= 2\,{\rm sgn} \{i,j,m\} \,|A_0|.
\end{align*}
Here, ${\rm sgn} \, \{i,j,m\}$ is the signature of $\{ i,j,m \}$ when $\{ i,j,m \}$ is a 
permutation of $\{1,2,3 \}$, and is $0$ otherwise.
From this equality, we have easily the second identity.
\end{proof}

\begin{rem}\label{Rem_identity}
Concerning the identity (1), we remark that the following is an $\lq\lq$illegal'' symbolic 
calculation of the determinant (with smaller size, for simplicity):
\begin{align*}
\begin{vmatrix}
\alpha_1^0 & \overline{\alpha}_1^{\,0} \\ 
\rule{0cm}{0.4cm} \alpha_2^0 & \overline{\alpha}_2^{\,0} 
\end{vmatrix}^2 
= & \begin{vmatrix}
\alpha_1^0 & \overline{\alpha}_1^{\,0} \\ 
\rule{0cm}{0.4cm} \alpha_2^0 & \overline{\alpha}_2^{\,0} 
\end{vmatrix} 
\begin{vmatrix}
\alpha_1^0 & \alpha_2^0 \\ 
\rule{0cm}{0.4cm} \overline{\alpha}_1^{\,0} & \overline{\alpha}_2^{\,0} 
\end{vmatrix} \\
= & \begin{vmatrix}
\alpha_1^0 \alpha_1^0 + \overline{\alpha}_1^{\,0} \,\overline{\alpha}_1^{\,0} & 
\alpha_1^0 \alpha_2^0 + \overline{\alpha}_1^{\,0} \,\overline{\alpha}_2^{\,0} \\
\rule{0cm}{0.4cm} \alpha_2^0 \alpha_1^0 + \overline{\alpha}_2^{\,0} \,
\overline{\alpha}_1^{\,0} & 
\alpha_2^0 \alpha_2^0 + \overline{\alpha}_2^{\,0} \,\overline{\alpha}_2^{\,0}
\end{vmatrix} \\
= & \begin{vmatrix}
2\alpha_{11}^0 & 2\alpha_{12}^{\,0} \\ 
\rule{0cm}{0.4cm} 2\alpha_{12}^0 & 2\alpha_{22}^{\,0} 
\end{vmatrix} 
= 4 \begin{vmatrix}
\alpha_{11}^0 & \alpha_{12}^{\,0} \\ 
\rule{0cm}{0.4cm} \alpha_{12}^0 & \alpha_{22}^{\,0} 
\end{vmatrix}.
\end{align*}
In the final expression, the correct coefficient is $2$, instead of $4$.
In this calculation, the determinant in the second line has no symbolic meaning.
In fact, if we expand it,  it becomes 
\begin{align*}
& (\alpha_1^0 \alpha_1^0 + \overline{\alpha}_1^{\,0} \,\overline{\alpha}_1^{\,0})
(\alpha_2^0 \alpha_2^0 + \overline{\alpha}_2^{\,0} \,\overline{\alpha}_2^{\,0})
-
(\alpha_1^0 \alpha_2^0 + \overline{\alpha}_1^{\,0} \,\overline{\alpha}_2^{\,0})
(\alpha_2^0 \alpha_1^0 + \overline{\alpha}_2^{\,0} \,\overline{\alpha}_1^{\,0}) \\
= & \alpha_1^0 \alpha_1^0 \alpha_2^0 \alpha_2^0 + \cdots,
\end{align*}
which has no symbolic meaning since four $\alpha_i^0$'s appear.
We must be careful not to fall into such an illegal calculation.
\end{rem}

\medskip

{\it Proof of Lemma~\ref{Lem_dG}}.
We first prove the uniqueness of the solution.
Since the indices $i, j, k, l$ moves in the range $1\sim3$, the system of linear equations can 
be expressed in the matrix form as follows:
$$
\begin{pmatrix}
\alpha_{22}^0 & -2\alpha_{12}^0 & 0 & \alpha_{11}^0 & 0 & 0 \\
\rule{0cm}{0.4cm} \alpha_{23}^0 & -\alpha_{13}^0 & -\alpha_{12}^0 & 0 & \alpha_{11}^0 & 0 \\
\rule{0cm}{0.4cm} 0 & \alpha_{23}^0 & -\alpha_{22}^0 & -\alpha_{13}^0 & \alpha_{12}^0 & 0 \\
\rule{0cm}{0.4cm} \alpha_{33}^0 & 0 & -2\alpha_{13}^0 & 0 & 0 & \alpha_{11}^0 \\
\rule{0cm}{0.4cm} 0 & \alpha_{33}^0 & -\alpha_{23}^0 & 0 & -\alpha_{13}^0 & \alpha_{12}^0 \\
\rule{0cm}{0.4cm} 0 & 0 & 0 & \alpha_{33}^0 & -2\alpha_{23}^0 & \alpha_{22}^0
\end{pmatrix}
\begin{pmatrix}
\gamma_{11} \\ \gamma_{12} \\ \gamma_{13} \\ \gamma_{22} \\ \gamma_{23} \\ 
\gamma_{33}
\end{pmatrix}
= \begin{pmatrix}
T_{1212} \\ T_{1213} \\ T_{1223} \\ T_{1313} \\ T_{1323} \\ T_{2323}
\end{pmatrix}.
$$
By direct calculations we know that the determinant of the above $(6,6)$-matrix is 
equal to $-2|A_0|^2$, which is non-zero from the assumption.
Hence the solution $\gamma_{ij}$ exists uniquely.

Next, we show that $\gamma_{ij}=\gamma_{ij}^1-\gamma_{ij}^2$ actually gives a 
solution.
For this purpose, we first calculate two polynomials 
\begin{align*}
& \gamma^1_{ik}\alpha_{jl}^0+\alpha_{ik}^0\gamma^1_{jl}
-\gamma^1_{il}\alpha_{jk}^0-\alpha_{il}^0\gamma^1_{jk}, \\
\rule{0cm}{0.5cm} & \gamma^2_{ik}\alpha_{jl}^0+\alpha_{ik}^0\gamma^2_{jl}
-\gamma^2_{il}\alpha_{jk}^0-\alpha_{il}^0\gamma^2_{jk}
\end{align*}
separately, and take the difference at the last step.
We decompose the first polynomial symbolically as 
$$
\gamma^1_{ik}\alpha_{jl}^0+\alpha_{ik}^0\gamma^1_{jl}
-\gamma^1_{il}\alpha_{jk}^0-\alpha_{il}^0\gamma^1_{jk}
= \begin{vmatrix} \overline{\overline{\alpha}}_i^{\,0} & \gamma_i^1 \\
\rule{0cm}{0.5cm} \overline{\overline{\alpha}}_j^{\,0} & \gamma_j^1 \end{vmatrix}
\begin{vmatrix} \overline{\overline{\alpha}}_k^{\,0} & \gamma_k^1 \\
\rule{0cm}{0.5cm} \overline{\overline{\alpha}}_l^{\,0} & \gamma_l^1 \end{vmatrix}.
$$
Here we use the symbol $\overline{\overline{\alpha}}_i^{\,0}$, since $\alpha_i^0$ 
and $\overline{\alpha}_i^{\,0}$ are already appeared in the symbol $\gamma_i^1$.
We calculate its half factor 
$$
\begin{vmatrix} \overline{\overline{\alpha}}_i^{\,0} & \gamma_i^1 \\
\rule{0cm}{0.5cm} \overline{\overline{\alpha}}_j^{\,0} & \gamma_j^1 \end{vmatrix}
= \overline{\overline{\alpha}}_i^{\,0} \gamma_j^1
-\overline{\overline{\alpha}}_j^{\,0} \gamma_i^1
$$
symbolically.
We cannot restore this symbolic half factor to a polynomial since its symbolic partner 
is lacked.
But in the final step of calculations, we will take the product with its partner 
$$
\begin{vmatrix} \overline{\overline{\alpha}}_k^{\,0} & \gamma_k^1 \\
\rule{0cm}{0.5cm} \overline{\overline{\alpha}}_l^{\,0} & \gamma_l^1 \end{vmatrix}
= \overline{\overline{\alpha}}_k^{\,0} \gamma_l^1
-\overline{\overline{\alpha}}_l^{\,0} \gamma_k^1,
$$
and we can safely restore it to a polynomial form.
For some time we proceed to calculate the above half factor as if it is a 
polynomial.
The result is given by the following equality:
\begin{align}
& \begin{vmatrix}
\overline{\overline{\alpha}}_i^{\,0} & \gamma^1_i \\ 
\rule{0cm}{0.5cm} \overline{\overline{\alpha}}_j^{\,0} & \gamma^1_j
\end{vmatrix}  \label{eq7} \\
= & \frac{1}{\sqrt{2 |A_0|}} \left(
\begin{vmatrix}
\alpha_1^0 & \overline{\alpha}_1^{\,0} & \overline{\overline{\alpha}}_1^{\,0} \\ 
\rule{0cm}{0.5cm} \alpha_2^0 & \overline{\alpha}_2^{\,0} & 
\overline{\overline{\alpha}}_2^{\,0} \\ 
\rule{0cm}{0.5cm} \alpha_3^0 & \overline{\alpha}_3^{\,0} & 
\overline{\overline{\alpha}}_3^{\,0}
\end{vmatrix} T_{ij}
-\begin{vmatrix}
\alpha_i^0 & \overline{\alpha}_i^{\,0} \\ 
\rule{0cm}{0.4cm} \alpha_j^0 & \overline{\alpha}_j^{\,0}
\end{vmatrix}
(\overline{\overline{\alpha}}_1^{\,0} T_{23}
-\overline{\overline{\alpha}}_2^{\,0} T_{13}
+\overline{\overline{\alpha}}_3^{\,0} T_{12})
\right). \notag
\end{align}
The equality for its symbolic partner can be similarly obtained.
In the case $i=j$, the equality (\ref{eq7}) clearly holds.
If we exchange $i$ and $j$, both sides of (\ref{eq7}) change their signs.
Thus, to prove (\ref{eq7}), we have only to consider the case where $\{i,j,k\}$ is a 
cyclic permutation of $\{1,2,3\}$ by taking a suitable $k$.
Then we have 
\begin{align*}
\begin{vmatrix}
\overline{\overline{\alpha}}_i^{\,0} & \gamma^1_i \\ 
\rule{0cm}{0.5cm} \overline{\overline{\alpha}}_j^{\,0} & \gamma^1_j
\end{vmatrix} 
= & \, \overline{\overline{\alpha}}_i^{\,0} \gamma_j^1 
 - \overline{\overline{\alpha}}_j^{\,0}\gamma_i^1 \\
 = & \frac{1}{\sqrt{2 |A_0|}} \,\overline{\overline{\alpha}}_i^{\,0}
\begin{vmatrix}
\alpha_1^0 & \overline{\alpha}_1^{\,0} & T_{1j} \\ 
\rule{0cm}{0.5cm} \alpha_2^0 & \overline{\alpha}_2^{\,0} & T_{2j} \\ 
\rule{0cm}{0.5cm} \alpha_3^0 & \overline{\alpha}_3^{\,0} & T_{3j} 
\end{vmatrix}
-  \frac{1}{\sqrt{2 |A_0|}} \,\overline{\overline{\alpha}}_j^{\,0}
\begin{vmatrix}
\alpha_1^0 & \overline{\alpha}_1^{\,0} & T_{1i} \\ 
\rule{0cm}{0.5cm} \alpha_2^0 & \overline{\alpha}_2^{\,0} & T_{2i} \\ 
\rule{0cm}{0.5cm} \alpha_3^0 & \overline{\alpha}_3^{\,0} & T_{3i} 
\end{vmatrix} \\
 = & \frac{1}{\sqrt{2 |A_0|}} 
\begin{vmatrix}
\alpha_1^0 & \overline{\alpha}_1^{\,0} & 
\overline{\overline{\alpha}}_i^{\,0}\,T_{1j} - \overline{\overline{\alpha}}_j^{\,0}\,T_{1i}\\ 
\rule{0cm}{0.5cm} \alpha_2^0 & \overline{\alpha}_2^{\,0} & 
\overline{\overline{\alpha}}_i^{\,0}\,T_{2j} - \overline{\overline{\alpha}}_j^{\,0}\,T_{2i}\\ 
\rule{0cm}{0.5cm} \alpha_3^0 & \overline{\alpha}_3^{\,0} & 
\overline{\overline{\alpha}}_i^{\,0}\,T_{3j}  - \overline{\overline{\alpha}}_j^{\,0}\,T_{3i}
\end{vmatrix}.
\end{align*}
Since $\{i,j,k \}$ is a cyclic permutation of $\{ 1,2,3 \}$, this is equal to 
\begin{align*}
& \frac{1}{\sqrt{2 |A_0|}} 
\begin{vmatrix}
\alpha_i^0 & \overline{\alpha}_i^{\,0} & 
\overline{\overline{\alpha}}_i^{\,0}\,T_{ij} - \overline{\overline{\alpha}}_j^{\,0}\,T_{ii}\\ 
\rule{0cm}{0.5cm} \alpha_j^0 & \overline{\alpha}_j^{\,0} & 
\overline{\overline{\alpha}}_i^{\,0}\,T_{jj} - \overline{\overline{\alpha}}_j^{\,0}\,T_{ji}\\ 
\rule{0cm}{0.5cm} \alpha_k^0 & \overline{\alpha}_k^{\,0} & 
\overline{\overline{\alpha}}_i^{\,0}\,T_{kj}  - \overline{\overline{\alpha}}_j^{\,0}\,T_{ki}
\end{vmatrix} \\
 = & \frac{1}{\sqrt{2 |A_0|}} 
\begin{vmatrix}
\alpha_i^0 & \overline{\alpha}_i^{\,0} & \overline{\overline{\alpha}}_i^{\,0}\,T_{ij}\\ 
\rule{0cm}{0.5cm} \alpha_j^0 & \overline{\alpha}_j^{\,0} &
 \overline{\overline{\alpha}}_j^{\,0}\,T_{ij}\\ 
\rule{0cm}{0.5cm} \alpha_k^0 & \overline{\alpha}_k^{\,0} & 
\overline{\overline{\alpha}}_k^{\,0}\,T_{ij}
-(\overline{\overline{\alpha}}_i^{\,0}\,T_{jk}
+ \overline{\overline{\alpha}}_j^{\,0}\,T_{ki}
+ \overline{\overline{\alpha}}_k^{\,0}\,T_{ij})
\end{vmatrix} \\
 = & \frac{1}{\sqrt{2 |A_0|}} \left( 
\begin{vmatrix}
\alpha_i^0 & \overline{\alpha}_i^{\,0} & \overline{\overline{\alpha}}_i^{\,0}\\ 
\rule{0cm}{0.5cm} \alpha_j^0 & \overline{\alpha}_j^{\,0} & \overline{\overline{\alpha}}_j^{\,0}\\ 
\rule{0cm}{0.5cm} \alpha_k^0 & \overline{\alpha}_k^{\,0} & \overline{\overline{\alpha}}_k^{\,0}
\end{vmatrix} T_{ij}
-  
\begin{vmatrix}
\alpha_i^0 & \overline{\alpha}_i^{\,0} & 0\\ 
\rule{0cm}{0.4cm} \alpha_j^0 & \overline{\alpha}_j^{\,0} & 0\\ 
\rule{0cm}{0.4cm} \alpha_k^0 & \overline{\alpha}_k^{\,0} & 1
\end{vmatrix} 
(\overline{\overline{\alpha}}_i^{\,0}\,T_{jk}
+ \overline{\overline{\alpha}}_j^{\,0}\,T_{ki}
+ \overline{\overline{\alpha}}_k^{\,0}\,T_{ij}) \right) \\
= & \frac{1}{\sqrt{2 |A_0|}} \left(
\begin{vmatrix}
\alpha_1^0 & \overline{\alpha}_1^{\,0} & \overline{\overline{\alpha}}_1^{\,0}\\ 
\rule{0cm}{0.5cm} \alpha_2^0 & \overline{\alpha}_2^{\,0} & 
\overline{\overline{\alpha}}_2^{\,0}\\ 
\rule{0cm}{0.5cm} \alpha_3^0 & \overline{\alpha}_3^{\,0} & 
\overline{\overline{\alpha}}_3^{\,0}
\end{vmatrix} T_{ij}
- \begin{vmatrix}
\alpha_i^0 & \overline{\alpha}_i^{\,0} \\ 
\rule{0cm}{0.4cm} \alpha_j^0 & \overline{\alpha}_j^{\,0} 
\end{vmatrix} 
(\overline{\overline{\alpha}}_1^{\,0}\,T_{23}
- \overline{\overline{\alpha}}_2^{\,0}\,T_{13}
+ \overline{\overline{\alpha}}_3^{\,0}\,T_{12}) \right), 
\end{align*}
which shows the equality (\ref{eq7}).

Now we take the product with its partner 
$$ \frac{1}{\sqrt{2 |A_0|}} \left( 
\begin{vmatrix}
\alpha_1^0 & \overline{\alpha}_1^{\,0} & \overline{\overline{\alpha}}_1^{\,0} \\ 
\rule{0cm}{0.5cm} \alpha_2^0 & \overline{\alpha}_2^{\,0} & 
\overline{\overline{\alpha}}_2^{\,0} \\ 
\rule{0cm}{0.5cm} \alpha_3^0 & \overline{\alpha}_3^{\,0} & 
\overline{\overline{\alpha}}_3^{\,0}
\end{vmatrix} T_{kl}
-\begin{vmatrix}
\alpha_k^0 & \overline{\alpha}_k^{\,0} \\ 
\rule{0cm}{0.4cm} \alpha_l^0 & \overline{\alpha}_l^{\,0}
\end{vmatrix}
(\overline{\overline{\alpha}}_1^{\,0} T_{23}
-\overline{\overline{\alpha}}_2^{\,0} T_{13}
+\overline{\overline{\alpha}}_3^{\,0} T_{12})
\right), 
$$
and expand it.
Then we have 
\begin{align*}
\begin{vmatrix}
\overline{\overline{\alpha}}_i^{\,0} & \gamma^1_i \\ 
\rule{0cm}{0.5cm} \overline{\overline{\alpha}}_j^{\,0} & \gamma^1_j
\end{vmatrix} 
\begin{vmatrix}
\overline{\overline{\alpha}}_k^{\,0} & \gamma^1_k \\ 
\rule{0cm}{0.5cm} \overline{\overline{\alpha}}_l^{\,0} & \gamma^1_l
\end{vmatrix} 
= & \frac{1}{2 |A_0|} 
\begin{vmatrix}
\alpha_1^0 & \overline{\alpha}_1^{\,0} & \overline{\overline{\alpha}}_1^{\,0} \\ 
\rule{0cm}{0.5cm} \alpha_2^0 & \overline{\alpha}_2^{\,0} & 
\overline{\overline{\alpha}}_2^{\,0} \\ 
\rule{0cm}{0.5cm} \alpha_3^0 & \overline{\alpha}_3^{\,0} & 
\overline{\overline{\alpha}}_3^{\,0}
\end{vmatrix}^2 T_{ij}T_{kl} \\
& -\frac{1}{2 |A_0|}
\begin{vmatrix}
\alpha_1^0 & \overline{\alpha}_1^{\,0} & \overline{\overline{\alpha}}_1^{\,0} \\ 
\rule{0cm}{0.5cm} \alpha_2^0 & \overline{\alpha}_2^{\,0} & 
\overline{\overline{\alpha}}_2^{\,0} \\ 
\rule{0cm}{0.5cm} \alpha_3^0 & \overline{\alpha}_3^{\,0} & 
\overline{\overline{\alpha}}_3^{\,0}
\end{vmatrix}
\begin{vmatrix}
\alpha_k^0 & \overline{\alpha}_k^{\,0} \\ 
\rule{0cm}{0.4cm} \alpha_l^0 & \overline{\alpha}_l^{\,0}
\end{vmatrix}  
(\overline{\overline{\alpha}}_1^{\,0} T_{23}
-\overline{\overline{\alpha}}_2^{\,0} T_{13}
+\overline{\overline{\alpha}}_3^{\,0} T_{12})T_{ij} \\
& -\frac{1}{2 |A_0|}
\begin{vmatrix}
\alpha_1^0 & \overline{\alpha}_1^{\,0} & \overline{\overline{\alpha}}_1^{\,0} \\ 
\rule{0cm}{0.5cm} \alpha_2^0 & \overline{\alpha}_2^{\,0} & 
\overline{\overline{\alpha}}_2^{\,0} \\ 
\rule{0cm}{0.5cm} \alpha_3^0 & \overline{\alpha}_3^{\,0} & 
\overline{\overline{\alpha}}_3^{\,0}
\end{vmatrix}
\begin{vmatrix}
\alpha_i^0 & \overline{\alpha}_i^{\,0} \\ 
\rule{0cm}{0.4cm} \alpha_j^0 & \overline{\alpha}_j^{\,0}
\end{vmatrix} 
(\overline{\overline{\alpha}}_1^{\,0} T_{23}
-\overline{\overline{\alpha}}_2^{\,0} T_{13}
+\overline{\overline{\alpha}}_3^{\,0} T_{12})T_{kl} \\
& + \frac{1}{2 |A_0|}
\begin{vmatrix}
\alpha_i^0 & \overline{\alpha}_i^{\,0} \\ 
\rule{0cm}{0.4cm} \alpha_j^0 & \overline{\alpha}_j^{\,0}
\end{vmatrix}
\begin{vmatrix}
\alpha_k^0 & \overline{\alpha}_k^{\,0} \\ 
\rule{0cm}{0.4cm} \alpha_l^0 & \overline{\alpha}_l^{\,0}
\end{vmatrix}
(\overline{\overline{\alpha}}_1^{\,0} T_{23}
-\overline{\overline{\alpha}}_2^{\,0} T_{13}
+\overline{\overline{\alpha}}_3^{\,0} T_{12})^2.
\end{align*}
From Lemma~\ref{Lem_identity}, this is equal to 
\begin{align*}
& 3\,T_{ijkl} -T_{kl}T_{ij} -T_{ij}T_{kl}
+ \frac{1}{2 |A_0|}
\begin{vmatrix}
\alpha_i^0 & \overline{\alpha}_i^{\,0} \\ 
\rule{0cm}{0.4cm} \alpha_j^0 & \overline{\alpha}_j^{\,0}
\end{vmatrix}
\begin{vmatrix}
\alpha_k^0 & \overline{\alpha}_k^{\,0} \\ 
\rule{0cm}{0.4cm} \alpha_l^0 & \overline{\alpha}_l^{\,0}
\end{vmatrix}
(\overline{\overline{\alpha}}_1^{\,0} T_{23}
-\overline{\overline{\alpha}}_2^{\,0} T_{13}
+\overline{\overline{\alpha}}_3^{\,0} T_{12})^2 \\
= & T_{ijkl}
+ \frac{1}{2 |A_0|}
\begin{vmatrix}
\alpha_i^0 & \overline{\alpha}_i^{\,0} \\ 
\rule{0cm}{0.4cm} \alpha_j^0 & \overline{\alpha}_j^{\,0}
\end{vmatrix}
\begin{vmatrix}
\alpha_k^0 & \overline{\alpha}_k^{\,0} \\ 
\rule{0cm}{0.4cm} \alpha_l^0 & \overline{\alpha}_l^{\,0}
\end{vmatrix}
(\overline{\overline{\alpha}}_1^{\,0} T_{23}
-\overline{\overline{\alpha}}_2^{\,0} T_{13}
+\overline{\overline{\alpha}}_3^{\,0} T_{12})^2.
\end{align*}
Thus, we have 
\begin{align*}
& \gamma^1_{ik}\alpha_{jl}^0+\alpha_{ik}^0\gamma^1_{jl}
-\gamma^1_{il}\alpha_{jk}^0-\alpha_{il}^0\gamma^1_{jk} \\
= & T_{ijkl}
+ \frac{1}{|A_0|}
(\alpha_{ik}^0 \alpha_{jl}^0-\alpha_{il}^0 \alpha_{jk}^0) 
(\overline{\overline{\alpha}}_1^{\,0} T_{23}
-\overline{\overline{\alpha}}_2^{\,0} T_{13}
+\overline{\overline{\alpha}}_3^{\,0} T_{12})^2.
\end{align*}

As for the second polynomial 
$
\gamma^2_{ik}\alpha_{jl}^0+\alpha_{ik}^0\gamma^2_{jl}
-\gamma^2_{il}\alpha_{jk}^0-\alpha_{il}^0\gamma^2_{jk},
$
we substitute 
\begin{align*}
\gamma^2_{ij} & =\frac{1}{2|A_0|}\,\alpha_{ij}^0 \,(\overline{\alpha}_1^{\,0} T_{23} 
-\overline{\alpha}_2^{\,0} T_{13}+\overline{\alpha}_3^{\,0} T_{12})^2 
\end{align*}
into it.
Then we have directly 
\begin{align*}
& \gamma^2_{ik}\alpha_{jl}^0+\alpha_{ik}^0\gamma^2_{jl}
-\gamma^2_{il}\alpha_{jk}^0-\alpha_{il}^0\gamma^2_{jk}  \\
= & \frac{2}{2|A_0|} (\alpha_{ik}^0 \alpha_{jl}^0-\alpha_{il}^0 \alpha_{jk}^0)
(\overline{\alpha}_1^{\,0} T_{23} 
-\overline{\alpha}_2^{\,0} T_{13}+\overline{\alpha}_3^{\,0} T_{12})^2 \\
= & \frac{1}{|A_0|} (\alpha_{ik}^0 \alpha_{jl}^0-\alpha_{il}^0 \alpha_{jk}^0)
(\overline{\overline{\alpha}}_1^{\,0} T_{23} 
-\overline{\overline{\alpha}}_2^{\,0} T_{13}+\overline{\overline{\alpha}}_3^{\,0} T_{12})^2.
\end{align*}
Hence, by taking the difference, we have 
\begin{align*}
& \gamma_{ik}\alpha_{jl}^0+\alpha_{ik}^0\gamma_{jl}
-\gamma_{il}\alpha_{jk}^0-\alpha_{il}^0\gamma_{jk} \\
= & \rule{0cm}{0.4cm} (\gamma^1_{ik}\alpha_{jl}^0+\alpha_{ik}^0\gamma^1_{jl}
-\gamma^1_{il}\alpha_{jk}^0-\alpha_{il}^0\gamma^1_{jk}) 
 - (\gamma^2_{ik}\alpha_{jl}^0+\alpha_{ik}^0\gamma^2_{jl}
-\gamma^2_{il}\alpha_{jk}^0-\alpha_{il}^0\gamma^2_{jk})  \\
= & \rule{0cm}{0.4cm} T_{ijkl},
\end{align*}
which completes the proof of Lemma~\ref{Lem_dG}.
\hfill{$\Box$}

\medskip

\begin{rem}\label{Rem_S4}
(1) 
To express the solution $\gamma_{ij}$, we used two types of symbols $\gamma_i^1$ 
and $\gamma_i^2$ and take the difference $\gamma_i^1 \gamma_j^1 
- \gamma_i^2 \gamma_j^2$.
It is natural to ask whether we can decompose $\gamma_{ij}$ simply as $\gamma_i 
\gamma_j$ symbolically.
If we are allowed to use complex numbers in symbolic calculations, we have such a 
decomposition.
In fact, by putting 
$$
\gamma_i = \frac{\sqrt{-1}}{\sqrt{2|A_0|}} \{ \omega \,\alpha_i^0 \,
(\overline{\alpha}_1^{\,0} T_{23}
-\overline{\alpha}_2^{\,0} T_{13} + \overline{\alpha}_3^{\,0}T_{12}) 
+\omega^2 \,\overline{\alpha}_i^{\,0} \,(\alpha_1^0 T_{23}-\alpha_2^0T_{13}
+\alpha_3^0 T_{12}) \}
$$
($\omega$ is a complex number satisfying $\omega^2+\omega+1=0$), 
we have $\gamma_{ij} = \gamma_i \gamma_j$ for $i,j=1,2,3$.
We can directly prove that this $\gamma_{ij}$ gives the solution of the system of linear 
equations in Lemma~\ref{Lem_dG} essentially in the same way as in the above proof.
We must use the same technique and identities during the proof.
We once mention that polynomials may be decomposed into symbols in several ways 
(cf. Remark~\ref{subst_exchange} (2)).

(2) 
We can write down the solution $\gamma_{ij}$ explicitly in the polynomial form by restoring 
the formula given in Lemma~\ref{Lem_dG}, but it becomes a long complicated expression 
as we gave one example $\gamma_{12}$ above.
It is almost impossible to check that this long expression actually gives a solution of the system 
of linear equations by direct substitution, in particular by hand calculations.
This fact shows the power of the symbolic method in treating long expressions.
\end{rem}

\medskip

Now, we return to the system of linear equations (\ref{moddG})
\begin{align*}
\overline{\beta}_{ikm}\alpha_{jl}^0+\alpha_{ik}^0\overline{\beta}_{jlm}
-\overline{\beta}_{ilm}\alpha_{jk}^0-\alpha_{il}^0\overline{\beta}_{jkm}
= \overline{S}_{ijklm}.
\end{align*}
As stated before, we have to solve the same system of equations as in 
Lemma~\ref{Lem_dG} three times.
Thus, for each $m$, we put $\overline{\beta}_{ijm}= \gamma_{ij}S_m$, 
$\overline{S}_{ijklm} = T_{ijkl}S_m$, and apply Lemma~\ref{Lem_dG}, which leads us to the 
formula of $\overline{\beta}_{ijm}$.
Then, we may apply the result to the derived Gauss equation for $\alpha_{ij}^0$ by putting 
$\overline{\beta}_{ijm} = \beta_{ijm}^0$ and $\overline{S}_{ijklm} = S_{ijklm}$.
Thus, we obtain the formula of $\beta_{ijm}^0$.

\begin{prop}\label{Prop_beta}
Notations being as above, we have 
\begin{align*}
& \beta_{ijm}^0 = 
\frac{1}{2|A_0|} \left( 
\begin{vmatrix}
\alpha_1^0 & \overline{\alpha}_1^{\,0} & S_{1i} \\ 
\rule{0cm}{0.4cm} \alpha_2^0 & \overline{\alpha}_2^{\,0} & S_{2i} \\ 
\rule{0cm}{0.4cm} \alpha_3^0 & \overline{\alpha}_3^{\,0} & S_{3i} 
\end{vmatrix} 
\begin{vmatrix}
\alpha_1^0 & \overline{\alpha}_1^{\,0} & S_{1j} \\ 
\rule{0cm}{0.4cm} \alpha_2^0 & \overline{\alpha}_2^{\,0} & S_{2j} \\ 
\rule{0cm}{0.4cm} \alpha_3^0 & \overline{\alpha}_3^{\,0} & S_{3j} 
\end{vmatrix}
-\alpha_{ij}^0 \,(\overline{\alpha}_1^{\,0} S_{23} 
-\overline{\alpha}_2^{\,0} S_{13}+\overline{\alpha}_3^{\,0} S_{12})^2 \right) S_m.
\end{align*}
\end{prop}

\noindent
This gives the explicit formula of $\beta^0 = \nabla \alpha^0$, which is represented as a 
function of $\alpha^0$ and $S=\nabla R$.
We thus settled the problem in step (A).
It seems difficult to obtain the above formula on $\beta^0$ by directly taking the covariant 
derivative of $\alpha^0$ given in Proposition~\ref{RA2} (2), except for some specific cases.
(See for example, Petersen \cite{Pe}, p.95.
Also, compare with the general formula (74) in Rosenson \cite{Ro}, p.345.)

\medskip

\subsection{A new condition equivalent to Rivertz's six equalities}
We prove the next fact in step (B).
For this purpose, we introduce a quadratic polynomial $H_0(x_1,x_2,x_3)$ on $x_1,x_2,x_3$ 
symbolically by 
\begin{align*}
H_0(x_1,x_2,x_3) & : = \begin{vmatrix} x_1 & \alpha_1^0 & S_1 \\
\rule{0cm}{0.4cm} x_2 & \alpha_2^0 & S_2 \\
\rule{0cm}{0.4cm} x_3 & \alpha_3^0 & S_3 \end{vmatrix} 
(x_1S_{23}-x_2S_{13}+x_3S_{12}) 
(\alpha_1^0S_{23} - \alpha_2^0S_{13} + \alpha_3^0S_{12}).
\end{align*}
By putting 
$$
h_{ijkl}^0 := \frac{1}{\,2\,} \left( S_{ij} \begin{vmatrix} \alpha_k^0 & S_k \\
\rule{0cm}{0.4cm} \alpha_l^0 & S_l \end{vmatrix} 
+ S_{kl} \begin{vmatrix} \alpha_i^0 & S_i \\
\rule{0cm}{0.4cm} \alpha_j^0 & S_j \end{vmatrix} \right) 
(\alpha_1^0 S_{23} - \alpha_2^0 S_{13}
+\alpha_3^0 S_{12}), 
$$
the polynomial $H_0(x_1,x_2,x_3)$ is expanded as follows 
$$
H_0(x_1,x_2,x_3)  =  h_{1212}^0\,x_3{}^2 - 2 h_{1213}^0\,x_2x_3 + 2 h_{1223}^0\,x_1x_3 
+  h_{1313}^0\,x_2{}^2 - 2 h_{1323}^0\,x_1x_2 +  h_{2323}^0\,x_1{}^2.
$$
It is easy to see that the coefficients $h_{ijkl}^0$ possess the following curvature like 
properties 
\begin{align*}
& h_{ijkl}^0 = -h_{jikl}^0 = -h_{ijlk}^0, \quad h_{ijkl}^0 = h_{klij}^0.
\end{align*}
They are explicitly given by 

\begin{align*}
& h_{1212}^0=\alpha_{11}^0S_{12232}-\alpha_{12}^0S_{12132}-\alpha_{12}^0S_{12231}
+\alpha_{13}^0S_{12122}+\alpha_{22}^0S_{12131}-\alpha_{23}^0S_{12121},\\
& h_{1213}^0 = \frac{1}{\,2\,} (\alpha_{11}^0S_{13232}+\alpha_{11}^0S_{12233}
-\alpha_{12}^0S_{12133}-\alpha_{12}^0S_{13132}-\alpha_{12}^0S_{13231} \\
& \qquad \quad +\alpha_{13}^0S_{12123} +\alpha_{13}^0S_{12132}-\alpha_{13}^0S_{12231}
+\alpha_{22}^0S_{13131}-\alpha_{33}^0S_{12121}),\\
& h_{1223}^0 = \frac{1}{\,2\,} (\alpha_{11}^0S_{23232}+\alpha_{12}^0S_{12233}
-\alpha_{12}^0S_{13232}-\alpha_{12}^0S_{23231}+\alpha_{22}^0S_{13231} \\
& \qquad \quad -\alpha_{22}^0S_{12133}+\alpha_{23}^0S_{12123}+\alpha_{23}^0S_{12132}-
\alpha_{23}^0S_{12231}-\alpha_{33}^0S_{12122}),\\
& h_{1313}^0 = \alpha_{11}^0S_{13233}-\alpha_{12}^0S_{13133}-\alpha_{13}^0S_{13231}
+\alpha_{13}^0S_{12133}+\alpha_{23}^0S_{13131}-\alpha_{33}^0S_{12131},\\
& h_{1323}^0 = \frac{1}{\,2\,} (\alpha_{11}^0S_{23233}+\alpha_{13}^0S_{12233}
-\alpha_{13}^0S_{13232}-\alpha_{13}^0S_{23231}-\alpha_{22}^0S_{13133} \\
& \qquad \quad+\alpha_{23}^0S_{12133}+\alpha_{23}^0S_{13132}+\alpha_{23}^0S_{13231}-
\alpha_{33}^0S_{12231}-\alpha_{33}^0S_{12132}),\\
& h_{2323}^0 = \alpha_{12}^0S_{23233}-\alpha_{13}^0S_{23232}-\alpha_{22}^0S_{13233}
+\alpha_{23}^0S_{13232}+\alpha_{23}^0S_{12233}-\alpha_{33}^0S_{12232}.
\end{align*}

\noindent
Then, we have the following proposition, which will connect Rivertz's six polynomials and the 
Codazzi equation through $H_0(x_1,x_2,x_3)$.
(See Proposition~\ref{Prop_Codazzi_2}.)

\begin{prop}\label{Prop_Ralpha}
Under the above notations, the identity 
$$
\frac{1}{\,2\,} RS(x_1,x_2,x_3) = -|A_0| H_0(x_1,x_2,x_3)
$$
holds.
In other words, we have the equalities 
\begin{align*}
& r_1=-|A_0| h_{1212}^0, \quad r_2=-2|A_0| h_{1213}^0, \quad r_3=-2|A_0| h_{1223}^0, \\
& r_4=-|A_0| h_{1313}^0, \quad r_5=-2|A_0| h_{1323}^0, \quad r_6=-|A_0| h_{2323}^0.
\end{align*}
In particular, the equality $RS(x_1,x_2,x_3)=0$ is equivalent to $H_0(x_1,x_2,x_3)=0$.
\end{prop}

\begin{proof}
By substituting the inverse formula of the Gauss equation (Proposition~\ref{RA2} (2)) into 
the polynomials $r_1$, $r_4$, $r_6$ in Introduction, we can directly obtain the above 
equalities for $r_1$, $r_4$, $r_6$.
(Remind that $\varepsilon/ \sqrt{|\widetilde{R}|}$ is equal to $1/|A_0|$. 
See Remark~\ref{Rem_const} (1)).

Concerning $r_2$, $r_3$, $r_5$, we define the following polynomial 
$H_0^{\prime}(x_1,x_2,x_3)$, which is actually 
zero on account of the second Bianchi identity 
\begin{align*}
H_0^{\prime}(x_1,x_2,x_3) : = &
\rule{0cm}{0.5cm} \{ x_1x_2(\alpha_{13}^0S_{23}+\alpha_{23}^0S_{13})
+x_1x_3(\alpha_{23}^0S_{12}-\alpha_{12}^0S_{23}) \\
& \hspace{0.5cm}  -x_2x_3(\alpha_{13}^0S_{12}+\alpha_{12}^0S_{13}) \} 
(S_{12}S_3-S_{13}S_2+S_{23}S_1), 
\end{align*}
and add it to $H_0(x_1,x_2,x_3)$.
Then the coefficient of $x_2x_3$ in $H_0(x_1,x_2,x_3)+H_0^{\prime}(x_1,x_2,x_3)$ is 
equal to 
\begin{align*}
& -2h_{1213}^0-(\alpha_{13}^0S_{12}+\alpha_{12}^0S_{13}) 
(S_{12}S_3-S_{13}S_2+S_{23}S_1) \\
= &  -(\alpha_{11}^0S_{13232}+\alpha_{11}^0S_{12233}-\alpha_{12}^0S_{12133}
 -\alpha_{12}^0S_{13132}-\alpha_{12}^0S_{13231} \\
& \hspace{0.6cm} +\alpha_{13}^0S_{12123}+\alpha_{13}^0S_{12132} -\alpha_{13}^0S_{12231}
+\alpha_{22}^0S_{13131}-\alpha_{33}^0S_{12121}) \\
& -(\alpha_{13}^0S_{12123}-\alpha_{13}^0S_{12132}+\alpha_{13}^0S_{12231}
+\alpha_{12}^0S_{12133}-\alpha_{12}^0S_{13132}+\alpha_{12}^0S_{13231}) \\
= & -(\alpha_{11}^0S_{12233}+\alpha_{11}^0S_{13232}-2\alpha_{12}^0S_{13132}
+\alpha_{22}^0S_{13131}+2\alpha_{13}^0S_{12123}  -\alpha_{33}^0S_{12121}).
\end{align*}
By substituting the inverse formula of the Gauss equation into this expression, we know that 
it is equal to  $r_2/|A_0|$.
This shows the equality $r_2=-2|A_0| h_{1213}^0$.
Remaining two equalities for $r_3$ and $r_5$ can be obtained in the same way.
\end{proof}

\begin{rem}\label{dependent}
Six polynomials $h_{ijkl}^0$ are not independent.
In fact, they satisfy the following relation:
$$
\alpha_{33}^0 h_{1212}^0  - 2\alpha_{23}^0 h_{1213}^0 + 2\alpha_{13}^0 h_{1223}^0 
+ \alpha_{22}^0 h_{1313}^0 - 2\alpha_{12}^0 h_{1323}^0 + \alpha_{11}^0 h_{2323}^0=0.
$$
It takes a little labor to examine this relation by direct calculations, since $h_{ijkl}^0$ are lengthy.
But this fact can be easily verified as follows.
We substitute $x_i=\overline{\alpha}_i^{\,0}$ $(i=1,2,3)$ into 
\begin{align*}
& H_0(x_1,x_2,x_3) \\
= & h_{1212}^0 x_3{}^2  - 2h_{1213}^0 x_2x_3 + 2h_{1223}^0 x_1x_3 
+ h_{1313}^0x_2{}^2 - 2h_{1323}^0x_1x_2  + h_{2323}^0x_1{}^2  \\
= & \begin{vmatrix} x_1 & \alpha_1^0 & S_1 \\
\rule{0cm}{0.4cm} x_2 & \alpha_2^0 & S_2 \\
\rule{0cm}{0.4cm} x_3 & \alpha_3^0 & S_3 \end{vmatrix} 
(x_1S_{23}-x_2S_{13}+x_3S_{12}) (\alpha_1^0S_{23} - \alpha_2^0S_{13} + \alpha_3^0S_{12}).
\end{align*}
Then we have 
\begin{align*}
& \alpha_{33}^0 h_{1212}^0  - 2\alpha_{23}^0 h_{1213}^0 + 2\alpha_{13}^0 h_{1223}^0 
+ \alpha_{22}^0 h_{1313}^0 - 2\alpha_{12}^0 h_{1323}^0 + \alpha_{11}^0 h_{2323}^0 \\
= & \begin{vmatrix} \overline{\alpha}_1^{\,0} & \alpha_1^0 & S_1 \\
\rule{0cm}{0.4cm} \overline{\alpha}_2^{\,0} & \alpha_2^0 & S_2 \\
\rule{0cm}{0.4cm} \overline{\alpha}_3^{\,0} & \alpha_3^0 & S_3 \end{vmatrix} 
(\overline{\alpha}_1^{\,0}S_{23}-\overline{\alpha}_2^{\,0}S_{13}
+\overline{\alpha}_3^{\,0}S_{12}) 
(\alpha_1^0S_{23} - \alpha_2^0S_{13} + \alpha_3^0S_{12}).
\end{align*}
In the last expression, by the exchanging method, we can see that it is zero, which shows 
the desired equality.
\end{rem}

\medskip

\subsection{Verification of the Codazzi equation}
Now, we are in a final position to prove Proposition~\ref{Prop_Codazzi_1}.
We prove the fact in step (C), and show the following proposition.

\begin{prop}\label{Prop_Codazzi_2}
Under the same notations and assumptions as above, the following equality holds:
$$
|A_0|(\beta_{ipq}^0-\beta_{iqp}^0) 
= -(\alpha_{i1}^0h_{23pq}^0-\alpha_{i2}^0h_{13pq}^0+\alpha_{i3}^0h_{12pq}^0).
$$
\end{prop}

\noindent
Since Rivertz's six equalities $r_1 = \cdots =r_6=0$ are equivalent to $h_{ijkl}^0=0$ 
(Proposition~\ref{Prop_Ralpha}), we have immediately the Codazzi equation 
$\beta_{ipq}^0= \beta_{iqp}^0$ from this proposition, which completes the proof of 
Proposition~\ref{Prop_Codazzi_1}.

We can directly verify the equalities in this proposition by restoring the formula 
in Proposition~\ref{Prop_beta}, and by using the second Bianchi identity.
But actually it requires tremendous labor, and it is hard to verify these equalities for 
all indices $i, p, q$ only by hand calculations.
In the following, to prove Proposition~\ref{Prop_Codazzi_2}, we once use the symbolic method.
By this device, we can prove the proposition by using only elementary properties of 
determinants, though to complete the calculations, it requires some technical ideas, 
which can be seen below.
We may say that the subsequent symbolic calculations constitute an essential part in 
the proof of our main theorem.

We also remark that the representation of $|A_0|(\beta_{ipq}^0-\beta_{iqp}^0)$ as a 
linear combination of $\alpha_{ij}^0 h_{pqrs}^0$ given in 
Proposition~\ref{Prop_Codazzi_2} is not uniquely determined, because there exists a 
non-trivial identity between $h_{ijkl}^0$ (cf. Remark~\ref{dependent}).

To prove Proposition~\ref{Prop_Codazzi_2}, we first prepare the following lemma.

\begin{lem}\label{Lem_det}
Assume that $\{i,j,k\}$ and $\{ p,q,r \}$ are cyclic permutations of $\{1,2,3\}$.
Then the following identities on symbols hold.
\begin{align}
\tag{1} & \alpha_i^0  \begin{vmatrix} \alpha_1^0 & \overline{\alpha}_1^{\,0} & S_1 \\
\rule{0cm}{0.4cm} \alpha_2^0 & \overline{\alpha}_2^{\,0} & S_2 \\
\rule{0cm}{0.4cm} \alpha_3^0 & \overline{\alpha}_3^{\,0} & S_3 
\end{vmatrix} 
(\overline{\alpha}_1^{\,0} S_{23} - \overline{\alpha}_2^{\,0} S_{13}
+\overline{\alpha}_3^{\,0} S_{12}) \\
& \hspace{3cm} =  -\begin{vmatrix} R_{ij} & S_{ij} \\
R_{ik} & S_{ik} \end{vmatrix} (S_1R_{23}-S_2R_{13}+S_3R_{12}). \notag
\end{align}

\begin{equation}
\tag{2} 
\begin{vmatrix}
\alpha_1^0 & \overline{\alpha}_1^{\,0} & S_{1i} \\ 
\rule{0cm}{0.4cm} \alpha_2^0 & \overline{\alpha}_2^{\,0} & S_{2i} \\ 
\rule{0cm}{0.4cm} \alpha_3^0 & \overline{\alpha}_3^{\,0} & S_{3i} 
\end{vmatrix} 
\begin{vmatrix}
\alpha_1^0 & \overline{\alpha}_1^{\,0} & S_{1p} \\ 
\rule{0cm}{0.4cm} \alpha_2^0 & \overline{\alpha}_2^{\,0} & S_{2p} \\ 
\rule{0cm}{0.4cm} \alpha_3^0 & \overline{\alpha}_3^{\,0} & S_{3p} 
\end{vmatrix}
 =  2 \begin{vmatrix} R_{ij} & S_{ij} \\
R_{ik} & S_{ik} \end{vmatrix}
 \begin{vmatrix} R_{pq} & S_{pq} \\
 R_{pr} & S_{pr} \end{vmatrix}.
\end{equation}
\end{lem}

\begin{proof}
(1) 
We first calculate 
$$
\alpha_i^0 \begin{vmatrix} \alpha_p^0 & \overline{\alpha}_p^{\,0} \\
\rule{0cm}{0.4cm} \alpha_q^0 & \overline{\alpha}_q^{\,0} \end{vmatrix}
(\overline{\alpha}_1^{\,0}S_{23}- \overline{\alpha}_2^{\,0} S_{13} 
+ \overline{\alpha}_3^{\,0}S_{12}).
$$
Since $\{ i,j,k \}$ is a cyclic permutation of $\{ 1,2,3 \}$, it is equal to 
\begin{align*}
& \rule{0cm}{0.4cm} \alpha_i^0 (\alpha_p^0 \,\overline{\alpha}_q^{\,0} 
- \alpha_q^0 \,\overline{\alpha}_p^{\,0})
(\overline{\alpha}_i^{\,0}S_{jk}+ \overline{\alpha}_j^{\,0} S_{ki} 
+ \overline{\alpha}_k^{\,0}S_{ij}) \\
= & \rule{0cm}{0.4cm} \alpha_{ip}^0\alpha_{iq}^0S_{jk}+\alpha_{ip}^0\alpha_{jq}^0S_{ki}
+\alpha_{ip}^0\alpha_{kq}^0S_{ij}
-\alpha_{iq}^0\alpha_{ip}^0S_{jk}-\alpha_{iq}^0\alpha_{jp}^0S_{ki}
-\alpha_{iq}^0\alpha_{kp}^0S_{ij} \\
= & \rule{0cm}{0.4cm} (\alpha_{ip}^0\alpha_{jq}^0-\alpha_{iq}^0\alpha_{jp}^0)S_{ki}
-(\alpha_{iq}^0\alpha_{kp}^0-\alpha_{ip}^0\alpha_{kq}^0)S_{ij} \\
= & \rule{0cm}{0.4cm} R_{ijpq}S_{ki}-R_{kipq}S_{ij} \\
= & R_{pq} \begin{vmatrix} R_{ij} & S_{ij} \\
R_{ki} & S_{ki} \end{vmatrix} 
=  -R_{pq} \begin{vmatrix} R_{ij} & S_{ij} \\
R_{ik} & S_{ik} \end{vmatrix}.
\end{align*}
Then, we have 
\begin{align*}
& \alpha_i^0  \begin{vmatrix} \alpha_1^0 & \overline{\alpha}_1^{\,0} & S_1 \\
\rule{0cm}{0.4cm} \alpha_2^0 & \overline{\alpha}_2^{\,0} & S_2 \\
\rule{0cm}{0.4cm} \alpha_3^0 & \overline{\alpha}_3^{\,0} & S_3 
\end{vmatrix} 
(\overline{\alpha}_1^{\,0} S_{23} - \overline{\alpha}_2^{\,0} S_{13}
+\overline{\alpha}_3^{\,0} S_{12}) \\
= & \alpha_i^0  \left( \begin{vmatrix} \alpha_2^0 & \overline{\alpha}_2^{\,0} \\
\rule{0cm}{0.4cm} \alpha_3^0 & \overline{\alpha}_3^{\,0} 
\end{vmatrix} S_1
- \begin{vmatrix} \alpha_1^0 & \overline{\alpha}_1^{\,0} \\
\rule{0cm}{0.4cm} \alpha_3^0 & \overline{\alpha}_3^{\,0} 
\end{vmatrix} S_2
+ \begin{vmatrix} \alpha_1^0 & \overline{\alpha}_1^{\,0} \\
\rule{0cm}{0.4cm} \alpha_2^0 & \overline{\alpha}_2^{\,0} 
\end{vmatrix} S_3 \right)
(\overline{\alpha}_1^{\,0} S_{23} - \overline{\alpha}_2^{\,0} S_{13}
+\overline{\alpha}_3^{\,0} S_{12}) \\
= & - \begin{vmatrix} R_{ij} & S_{ij} \\
R_{ik} & S_{ik} \end{vmatrix} (S_1R_{23}-S_2R_{13}+S_3R_{12}).
\end{align*}

(2) 
Since $\{i,j,k\}$ and $\{ p,q,r \}$ are cyclic permutations of $\{1,2,3\}$, 
the left hand side is equal to 
\begin{align*}
& \begin{vmatrix}
\alpha_i^0 & \overline{\alpha}_i^{\,0} & S_{ii} \\ 
\rule{0cm}{0.4cm} \alpha_j^0 & \overline{\alpha}_j^{\,0} & S_{ji} \\ 
\rule{0cm}{0.4cm} \alpha_k^0 & \overline{\alpha}_k^{\,0} & S_{ki} 
\end{vmatrix} 
\begin{vmatrix}
\alpha_p^0 & \overline{\alpha}_p^{\,0} & S_{pp} \\ 
\rule{0cm}{0.4cm} \alpha_q^0 & \overline{\alpha}_q^{\,0} & S_{qp} \\ 
\rule{0cm}{0.4cm} \alpha_r^0 & \overline{\alpha}_r^{\,0} & S_{rp} 
\end{vmatrix} \\
= & \left( \begin{vmatrix}
\alpha_i^0 & \overline{\alpha}_i^{\,0} \\ 
\rule{0cm}{0.4cm} \alpha_j^0 & \overline{\alpha}_j^{\,0} 
\end{vmatrix} S_{ki} 
- \begin{vmatrix}
\alpha_i^0 & \overline{\alpha}_i^{\,0} \\ 
\rule{0cm}{0.4cm} \alpha_k^0 & \overline{\alpha}_k^{\,0} 
\end{vmatrix} S_{ji} \right)
\left( \begin{vmatrix}
\alpha_p^0 & \overline{\alpha}_p^{\,0} \\ 
\rule{0cm}{0.4cm} \alpha_q^0 & \overline{\alpha}_q^{\,0} 
\end{vmatrix} S_{rp} 
- \begin{vmatrix}
\alpha_p^0 & \overline{\alpha}_p^{\,0} \\ 
\rule{0cm}{0.4cm} \alpha_r^0 & \overline{\alpha}_r^{\,0} 
\end{vmatrix} S_{qp} \right).
\end{align*}
Here, the coefficient of $S_{ki}S_{rp}$ is given by  
\begin{align*}
\begin{vmatrix}
\alpha_i^0 & \overline{\alpha}_i^{\,0} \\ 
\rule{0cm}{0.4cm} \alpha_j^0 & \overline{\alpha}_j^{\,0} 
\end{vmatrix} 
\begin{vmatrix}
\alpha_p^0 & \overline{\alpha}_p^{\,0} \\ 
\rule{0cm}{0.4cm} \alpha_q^0 & \overline{\alpha}_q^{\,0} 
\end{vmatrix} 
= & (\alpha_i^0 \overline{\alpha}_j^{\,0}- \alpha_j^0 \overline{\alpha}_i^{\,0})
(\alpha_p^0 \overline{\alpha}_q^{\,0}- \alpha_q^0 \overline{\alpha}_p^{\,0}) \\
= & 2(\alpha_{ip}^0\alpha_{jq}^0-\alpha_{iq}^0\alpha_{jp}^0) \\
= & 2R_{ijpq} = 2R_{ij}R_{pq}.
\end{align*}
Remaining coefficients are similarly calculated, and finally we know that the above is equal to 
$$
2(R_{ij}S_{ki}-R_{ik}S_{ji})(R_{pq}S_{rp}-R_{pr}S_{qp})
 =  2 \begin{vmatrix} R_{ij} & S_{ij} \\
R_{ik} & S_{ik} \end{vmatrix}
 \begin{vmatrix} R_{pq} & S_{pq} \\
 R_{pr} & S_{pr} \end{vmatrix}, 
$$
which is the desired equality.
\end{proof}

{\it Proof of Proposition~\ref{Prop_Codazzi_2}.}
In case $p=q$, the equality clearly holds, and if we exchange $p$ and $q$, then the both 
sides change their signs.
Hence we may assume that $\{p,q,r\}$ is a cyclic permutation of $\{1,2,3\}$ by taking 
a suitable $r$.
We also assume that $\{i,j,k\}$ is a cyclic permutation of $\{1,2,3\}$ by taking suitable 
$j$, $k$.
We fix such $\{ p,q,r \}$ and $\{ i,j,k \}$ throughout this proof.

First of all, we calculate 
$$
-2(\alpha_{i1}^0h_{23pq}^0-\alpha_{i2}^0h_{13pq}^0+\alpha_{i3}^0h_{12pq}^0),
$$
which is twice the right hand side of the equality.
It is equal to 
$$
-2(\alpha_{ii}^0h_{jkpq}^0+\alpha_{ij}^0h_{kipq}^0+\alpha_{ik}^0h_{ijpq}^0).
$$
We substitute 
$$
h_{ijkl}^0 = \frac{1}{\,2\,} \left( S_{ij} \begin{vmatrix} \overline{\alpha}_k^{\,0} & S_k \\
\rule{0cm}{0.4cm} \overline{\alpha}_l^{\,0} & S_l \end{vmatrix} 
+ S_{kl} \begin{vmatrix} \overline{\alpha}_i^{\,0} & S_i \\
\rule{0cm}{0.4cm} \overline{\alpha}_j^{\,0} & S_j \end{vmatrix} \right) 
(\overline{\alpha}_1^{\,0} S_{23} - \overline{\alpha}_2^{\,0} S_{13}
+\overline{\alpha}_3^{\,0} S_{12}), 
$$
where we replace $\alpha_i^0$ in the definition of $h_{ijkl}^0$ by another symbol 
$\overline{\alpha}_i^{\,0}$.
Then it becomes 
\begin{align}
& -\left\{ \alpha_{ii}^0 \left( S_{jk} \begin{vmatrix} \overline{\alpha}_p^{\,0} & S_p \\
\rule{0cm}{0.4cm} \overline{\alpha}_q^{\,0} & S_q \end{vmatrix} 
+ S_{pq} \begin{vmatrix} \overline{\alpha}_j^{\,0} & S_j \\
\rule{0cm}{0.4cm} \overline{\alpha}_k^{\,0} & S_k \end{vmatrix} \right) 
+ \alpha_{ij}^0 \left( S_{ki} \begin{vmatrix} \overline{\alpha}_p^{\,0} & S_p \\
\rule{0cm}{0.4cm} \overline{\alpha}_q^{\,0} & S_q \end{vmatrix} 
+ S_{pq} \begin{vmatrix} \overline{\alpha}_k^{\,0} & S_k \\
\rule{0cm}{0.4cm} \overline{\alpha}_i^{\,0} & S_i \end{vmatrix} \right) \right. \label{eq8} \\
& \hspace{0.5cm} + \left. \alpha_{ik}^0 \left( S_{ij} \begin{vmatrix} \overline{\alpha}_p^{\,0} 
& S_p \\ \rule{0cm}{0.4cm} \overline{\alpha}_q^{\,0} & S_q \end{vmatrix} 
+ S_{pq} \begin{vmatrix} \overline{\alpha}_i^{\,0} & S_i \\
\rule{0cm}{0.4cm} \overline{\alpha}_j^{\,0} & S_j \end{vmatrix} \right) \right\}
(\overline{\alpha}_1^{\,0} S_{23} - \overline{\alpha}_2^{\,0} S_{13}
+\overline{\alpha}_3^{\,0} S_{12}) \notag \\
= & -S_{pq} \left( \alpha_{ii}^0 \begin{vmatrix} \overline{\alpha}_j^{\,0} & S_j \\
\rule{0cm}{0.4cm} \overline{\alpha}_k^{\,0} & S_k \end{vmatrix} 
+\alpha_{ij}^0 \begin{vmatrix} \overline{\alpha}_k^{\,0} & S_k \\
\rule{0cm}{0.4cm} \overline{\alpha}_i^{\,0} & S_i \end{vmatrix}
+\alpha_{ik}^0 \begin{vmatrix} \overline{\alpha}_i^{\,0} & S_i \\
\rule{0cm}{0.4cm} \overline{\alpha}_j^{\,0} & S_j \end{vmatrix} \right) 
(\overline{\alpha}_1^{\,0} S_{23} - \overline{\alpha}_2^{\,0} S_{13}
+\overline{\alpha}_3^{\,0} S_{12}) \notag \\
& - \alpha_i^0 \begin{vmatrix} \overline{\alpha}_p^{\,0} & S_p \\
\rule{0cm}{0.4cm} \overline{\alpha}_q^{\,0} & S_q \end{vmatrix}
(\alpha_i^0 S_{jk}+\alpha_j^0 S_{ki}+\alpha_k^0 S_{ij})
(\overline{\alpha}_1^{\,0} S_{23} - \overline{\alpha}_2^{\,0} S_{13}
+\overline{\alpha}_3^{\,0} S_{12}) \notag \\
= & -\alpha_i^0 S_{pq}  \begin{vmatrix} \alpha_1^0 & \overline{\alpha}_1^{\,0} & S_1 \\
\rule{0cm}{0.4cm} \alpha_2^0 & \overline{\alpha}_2^{\,0} & S_2 \\
\rule{0cm}{0.4cm} \alpha_3^0 & \overline{\alpha}_3^{\,0} & S_3 
\end{vmatrix} 
(\overline{\alpha}_1^{\,0} S_{23} - \overline{\alpha}_2^{\,0} S_{13}
+\overline{\alpha}_3^{\,0} S_{12}) \notag \\
& -\alpha_i^0  \begin{vmatrix} \overline{\alpha}_p^{\,0} & S_p \\
\rule{0cm}{0.4cm} \overline{\alpha}_q^{\,0} & S_q \end{vmatrix}
(\alpha_1^0 S_{23}-\alpha_2^0 S_{13}+\alpha_3^0 S_{12})
(\overline{\alpha}_1^{\,0} S_{23} - \overline{\alpha}_2^{\,0} S_{13}
+\overline{\alpha}_3^{\,0} S_{12}). \notag
\end{align}

Next, we calculate $2|A_0|(\beta_{ipq}^0-\beta_{iqp}^0)$ by using the formula in  
Proposition~\ref{Prop_beta}.
Then we have 
\begin{align*}
& 2|A_0|(\beta_{ipq}^0-\beta_{iqp}^0) \\
= & \begin{vmatrix}
\alpha_1^0 & \overline{\alpha}_1^{\,0} & S_{1i} \\ 
\rule{0cm}{0.4cm} \alpha_2^0 & \overline{\alpha}_2^{\,0} & S_{2i} \\ 
\rule{0cm}{0.4cm} \alpha_3^0 & \overline{\alpha}_3^{\,0} & S_{3i} 
\end{vmatrix} 
\begin{vmatrix}
\alpha_1^0 & \overline{\alpha}_1^{\,0} & S_{1p} \\ 
\rule{0cm}{0.4cm} \alpha_2^0 & \overline{\alpha}_2^{\,0} & S_{2p} \\ 
\rule{0cm}{0.4cm} \alpha_3^0 & \overline{\alpha}_3^{\,0} & S_{3p} 
\end{vmatrix} 
S_q
- 
\begin{vmatrix}
\alpha_1^0 & \overline{\alpha}_1^{\,0} & S_{1i} \\ 
\rule{0cm}{0.4cm} \alpha_2^0 & \overline{\alpha}_2^{\,0} & S_{2i} \\ 
\rule{0cm}{0.4cm} \alpha_3^0 & \overline{\alpha}_3^{\,0} & S_{3i} 
\end{vmatrix}  
\begin{vmatrix}
\alpha_1^0 & \overline{\alpha}_1^{\,0} & S_{1q} \\ 
\rule{0cm}{0.4cm} \alpha_2^0 & \overline{\alpha}_2^{\,0} & S_{2q} \\ 
\rule{0cm}{0.4cm} \alpha_3^0 & \overline{\alpha}_3^{\,0} & S_{3q} 
\end{vmatrix} 
S_p \\
&- \begin{vmatrix} \alpha_{ip}^0 & S_p \\
\rule{0.0cm}{0.4cm} \alpha_{iq}^0 & S_q \end{vmatrix}
(\overline{\alpha}_1^{\,0}S_{23}- \overline{\alpha}_2^{\,0} S_{13} 
+ \overline{\alpha}_3^{\,0}S_{12})^2.
\end{align*}
By Lemma~\ref{Lem_det} (2), it is  equal to 
\begin{align}
& 2 \begin{vmatrix} R_{ij} & S_{ij} \\
R_{ik} & S_{ik} \end{vmatrix}
 \begin{vmatrix} R_{pq} & S_{pq} \\
 R_{pr} & S_{pr} \end{vmatrix} S_q
 - 
 2 \begin{vmatrix} R_{ij} & S_{ij} \\
R_{ik} & S_{ik} \end{vmatrix}
 \begin{vmatrix} R_{qr} & S_{qr} \\
 R_{qp} & S_{qp} \end{vmatrix} S_p \label{eq9} \\
&- \begin{vmatrix} \alpha_{ip}^0 & S_p \\
\rule{0.0cm}{0.4cm} \alpha_{iq}^0 & S_p \end{vmatrix} 
(\overline{\alpha}_1^{\,0}S_{23}- \overline{\alpha}_2^{\,0} S_{13} 
+ \overline{\alpha}_3^{\,0}S_{12})^2. \notag
\end{align}
The first line of (\ref{eq9}) is equal to 
\begin{align}
& 2 \begin{vmatrix} R_{ij} & S_{ij} \\
 R_{ik} & S_{ik} \end{vmatrix} 
 \left( \begin{vmatrix} R_{pq} & S_{pq} \\
 R_{pr} & S_{pr} \end{vmatrix}S_q
 + 
 \begin{vmatrix} R_{qr} & S_{qr} \\
 R_{pq} & S_{pq} \end{vmatrix}S_p \right) \label{eq3} \\
 = & 2 \begin{vmatrix} R_{ij} & S_{ij} \\
 R_{ik} & S_{ik} \end{vmatrix} 
 \{ S_{pq}(S_pR_{qr}+S_qR_{rp}+S_rR_{pq}) \notag \\
& \hspace{3cm}  - R_{pq} (S_pS_{qr}+S_qS_{rp}+S_rS_{pq}) \}  \notag \\
  = & 2 S_{pq} \begin{vmatrix} R_{ij} & S_{ij} \\
 R_{ik} & S_{ik} \end{vmatrix} 
 (S_pR_{qr}+S_qR_{rp}+S_rR_{pq})  \notag \\
  = & 2 S_{pq} \begin{vmatrix} R_{ij} & S_{ij} \\
 R_{ik} & S_{ik} \end{vmatrix} 
 (S_1R_{23}-S_2R_{13}+S_3R_{12}) \notag \\
= &  -2\alpha_i^0 S_{pq}  \begin{vmatrix} \alpha_1^0 & \overline{\alpha}_1^{\,0} & S_1 \\
\rule{0cm}{0.4cm} \alpha_2^0 & \overline{\alpha}_2^{\,0} & S_2 \\
\rule{0cm}{0.4cm} \alpha_3^0 & \overline{\alpha}_3^{\,0} & S_3 
\end{vmatrix} 
(\overline{\alpha}_1^{\,0} S_{23} - \overline{\alpha}_2^{\,0} S_{13}
+\overline{\alpha}_3^{\,0} S_{12}), \notag
 \end{align}
on account of the second Bianchi identity and Lemma~\ref{Lem_det} (1).

Next, we calculate the second line of (\ref{eq9}).
We show the equality 
\begin{align}
& - \alpha_i^0 \begin{vmatrix} \alpha_{p}^0 & S_p \\
\rule{0.0cm}{0.4cm} \alpha_{q}^0 & S_q \end{vmatrix}
(\overline{\alpha}_1^{\,0}S_{23}- \overline{\alpha}_2^{\,0} S_{13} 
+ \overline{\alpha}_3^{\,0}S_{12})^2 \label{eq4} \\
= & - \alpha_i^0 \begin{vmatrix} 
\overline{\alpha}_p^{\,0} & S_p \\ 
\rule{0cm}{0.4cm} \overline{\alpha}_q^{\,0} & S_q 
\end{vmatrix}
(\alpha_1^0S_{23}- \alpha_2^0 S_{13} + \alpha_3^0S_{12}) 
(\overline{\alpha}_1^{\,0}S_{23}- \overline{\alpha}_2^{\,0} S_{13} 
+ \overline{\alpha}_3^{\,0}S_{12}) \notag \\
& \hspace{1cm} + \alpha_i^0 S_{pq}  \begin{vmatrix} \alpha_1^0 & 
\overline{\alpha}_1^{\,0} & S_1 \\
\rule{0cm}{0.4cm} \alpha_2^0 & \overline{\alpha}_2^{\,0} & S_2 \\
\rule{0cm}{0.4cm} \alpha_3^0 & \overline{\alpha}_3^{\,0} & S_3
\end{vmatrix} 
(\overline{\alpha}_1^{\,0} S_{23} - \overline{\alpha}_2^{\,0} S_{13}
+\overline{\alpha}_3^{\,0} S_{12}). \notag
\end{align}
We calculate 
$$
\begin{vmatrix}
\alpha_p^0 & S_p \\ \rule{0cm}{0.4cm} \alpha_q^0 & S_q 
\end{vmatrix}
(\overline{\alpha}_1^{\,0}S_{23}- \overline{\alpha}_2^{\,0} S_{13} 
+ \overline{\alpha}_3^{\,0}S_{12}), 
$$
which is a symbolic piece of the left hand side of (\ref{eq4}).
We can easily examine the equality 
\begin{align*}
& \begin{vmatrix}
\alpha_p^0 & S_p \\ \rule{0cm}{0.4cm} \alpha_q^0 & S_q 
\end{vmatrix}
(\overline{\alpha}_p^{\,0}S_{qr}+ \overline{\alpha}_q^{\,0} S_{rp} 
+ \overline{\alpha}_r^{\,0}S_{pq}) 
- 
\begin{vmatrix} 
\overline{\alpha}_p^{\,0} & S_p \\ 
\rule{0cm}{0.4cm} \overline{\alpha}_q^{\,0} & S_q 
\end{vmatrix}
(\alpha_p^0S_{qr}+ \alpha_q^0 S_{rp} + \alpha_r^0S_{pq}) \\
& = \begin{vmatrix}
\alpha_p^0 & \overline{\alpha}_p^{\,0} \\
\rule{0cm}{0.4cm} \alpha_q^0 & \overline{\alpha}_q^{\,0} 
\end{vmatrix} S_pS_{qr}
+\begin{vmatrix}
\alpha_p^0 & \overline{\alpha}_p^{\,0} \\
\rule{0cm}{0.4cm} \alpha_q^0 & \overline{\alpha}_q^{\,0}
\end{vmatrix} S_qS_{rp}
+ \begin{vmatrix}
\alpha_p^0 & \overline{\alpha}_p^{\,0} \\
\rule{0cm}{0.4cm} \alpha_r^0 & \overline{\alpha}_r^{\,0} 
\end{vmatrix} S_qS_{pq}
- \begin{vmatrix}
\alpha_q^0 & \overline{\alpha}_q^{\,0} \\
\rule{0cm}{0.4cm} \alpha_r^0 & \overline{\alpha}_r^{\,0} 
\end{vmatrix} S_pS_{pq}.
\end{align*}
From this equality we have 
\begin{align*}
& \begin{vmatrix}
\alpha_p^0 & S_p \\ \rule{0cm}{0.4cm} \alpha_q^0 & S_q 
\end{vmatrix}
(\overline{\alpha}_p^{\,0}S_{qr}+ \overline{\alpha}_q^{\,0} S_{rp} 
+ \overline{\alpha}_r^{\,0}S_{pq}) \\
= &  \begin{vmatrix} 
\overline{\alpha}_p^{\,0} & S_p \\ 
\rule{0cm}{0.4cm} \overline{\alpha}_q^{\,0} & S_q 
\end{vmatrix}
(\alpha_p^0S_{qr}+ \alpha_q^0 S_{rp} + \alpha_r^0S_{pq}) \\
 & + \begin{vmatrix}
\alpha_p^0 & \overline{\alpha}_p^{\,0} \\
\rule{0cm}{0.4cm} \alpha_q^0 & \overline{\alpha}_q^{\,0} 
\end{vmatrix} S_pS_{qr}
+\begin{vmatrix}
\alpha_p^0 & \overline{\alpha}_p^{\,0} \\
\rule{0cm}{0.4cm} \alpha_q^0 & \overline{\alpha}_q^{\,0}
\end{vmatrix} S_qS_{rp}
+ \begin{vmatrix}
\alpha_p^0 & \overline{\alpha}_p^{\,0} \\
\rule{0cm}{0.4cm} \alpha_r^0 & \overline{\alpha}_r^{\,0} 
\end{vmatrix} S_qS_{pq}
- \begin{vmatrix}
\alpha_q^0 & \overline{\alpha}_q^{\,0} \\
\rule{0cm}{0.4cm} \alpha_r^0 & \overline{\alpha}_r^{\,0} 
\end{vmatrix} S_pS_{pq} \\
& + \begin{vmatrix}
\alpha_p^0 & \overline{\alpha}_p^{\,0} \\
\rule{0cm}{0.4cm} \alpha_q^0 & \overline{\alpha}_q^{\,0} 
\end{vmatrix} S_rS_{pq} 
- \begin{vmatrix}
\alpha_p^0 & \overline{\alpha}_p^{\,0} \\
\rule{0cm}{0.4cm} \alpha_q^0 & \overline{\alpha}_q^{\,0} 
\end{vmatrix} S_rS_{pq} \\
= &  \begin{vmatrix} 
\overline{\alpha}_p^{\,0} & S_p \\ 
\rule{0cm}{0.4cm} \overline{\alpha}_q^{\,0} & S_q 
\end{vmatrix}
(\alpha_p^0S_{qr}+ \alpha_q^0 S_{rp} + \alpha_r^0S_{pq}) 
+ \begin{vmatrix}
\alpha_p^0 & \overline{\alpha}_p^{\,0} \\
\rule{0cm}{0.4cm} \alpha_q^0 & \overline{\alpha}_q^{\,0} 
\end{vmatrix} (S_pS_{qr}+S_qS_{rp}+S_rS_{pq}) \\
& - S_{pq} \left( \begin{vmatrix}
\alpha_q^0 & \overline{\alpha}_q^{\,0} \\
\rule{0cm}{0.4cm} \alpha_r^0 & \overline{\alpha}_r^{\,0} 
\end{vmatrix} S_p 
- \begin{vmatrix}
\alpha_p^0 & \overline{\alpha}_p^{\,0} \\
\rule{0cm}{0.4cm} \alpha_r^0 & \overline{\alpha}_r^{\,0} 
\end{vmatrix} S_q
+ \begin{vmatrix}
\alpha_p^0 & \overline{\alpha}_p^{\,0} \\
\rule{0cm}{0.4cm} \alpha_q^0 & \overline{\alpha}_q^{\,0} 
\end{vmatrix} S_r \right) \\
= &  \begin{vmatrix} 
\overline{\alpha}_p^{\,0} & S_p \\ 
\rule{0cm}{0.4cm} \overline{\alpha}_q^{\,0} & S_q 
\end{vmatrix}
(\alpha_p^0S_{qr}+ \alpha_q^0 S_{rp} + \alpha_r^0S_{pq}) 
 - S_{pq} \begin{vmatrix}
\alpha_p^0 & \overline{\alpha}_p^{\,0} & S_p \\
\rule{0cm}{0.4cm} \alpha_q^0 & \overline{\alpha}_q^{\,0} & S_q \\
\rule{0cm}{0.4cm} \alpha_r^0 & \overline{\alpha}_r^{\,0} & S_r
\end{vmatrix}, 
\end{align*}
by using the second Bianchi identity.
Hence, by multiplying the remaining symbolic piece 
$$
- \alpha_i^0 (\overline{\alpha}_1^{\,0}S_{23}- \overline{\alpha}_2^{\,0} S_{13} 
+ \overline{\alpha}_3^{\,0}S_{12})
= - \alpha_i^0 (\overline{\alpha}_p^{\,0}S_{qr}+ \overline{\alpha}_q^{\,0} S_{rp} 
+ \overline{\alpha}_r^{\,0}S_{pq})
$$
to the above, the second line of (\ref{eq9}) is equal to 
\begin{align*}
& - \alpha_i^0 \begin{vmatrix} 
\overline{\alpha}_p^{\,0} & S_p \\ 
\rule{0cm}{0.4cm} \overline{\alpha}_q^{\,0} & S_q 
\end{vmatrix}
(\alpha_p^0S_{qr}+ \alpha_q^0 S_{rp} + \alpha_r^0S_{pq}) 
(\overline{\alpha}_p^{\,0}S_{qr}+ \overline{\alpha}_q^{\,0} S_{rp} 
+ \overline{\alpha}_r^{\,0}S_{pq}) \\
& + \alpha_i^0 S_{pq} \begin{vmatrix}
\alpha_p^0 & \overline{\alpha}_p^{\,0} & S_p \\
\rule{0cm}{0.4cm} \alpha_q^0 & \overline{\alpha}_q^{\,0} & S_q \\
\rule{0cm}{0.4cm} \alpha_r^0 & \overline{\alpha}_r^{\,0} & S_r
\end{vmatrix} 
(\overline{\alpha}_p^{\,0}S_{qr}+ \overline{\alpha}_q^{\,0} S_{rp} 
+ \overline{\alpha}_r^{\,0}S_{pq}) \\
= & - \alpha_i^0 \begin{vmatrix} 
\overline{\alpha}_p^{\,0} & S_p \\ 
\rule{0cm}{0.4cm} \overline{\alpha}_q^{\,0} & S_q 
\end{vmatrix}
(\alpha_1^0S_{23}- \alpha_2^0 S_{13} + \alpha_3^0S_{12}) 
(\overline{\alpha}_1^{\,0}S_{23}- \overline{\alpha}_2^{\,0} S_{13} 
+ \overline{\alpha}_3^{\,0}S_{12}) \\
& + \alpha_i^0 S_{pq} \begin{vmatrix}
\alpha_1^0 & \overline{\alpha}_1^{\,0} & S_1 \\
\rule{0cm}{0.4cm} \alpha_2^0 & \overline{\alpha}_2^{\,0} & S_2 \\
\rule{0cm}{0.4cm} \alpha_3^0 & \overline{\alpha}_3^{\,0} & S_3
\end{vmatrix} 
(\overline{\alpha}_1^{\,0}S_{23}- \overline{\alpha}_2^{\,0} S_{13} 
+ \overline{\alpha}_3^{\,0}S_{12}), 
\end{align*}
which shows the equality (\ref{eq4}).

Taking the sum of (\ref{eq3}) and (\ref{eq4}), we have 
\begin{align*}
& 2|A_0|(\beta_{112}^0-\beta_{121}^0) \\
& = - \alpha_i^0 S_{pq} \begin{vmatrix}
\alpha_1^0 & \overline{\alpha}_1^{\,0} & S_1 \\
\rule{0cm}{0.4cm} \alpha_2^0 & \overline{\alpha}_2^{\,0} & S_2 \\
\rule{0cm}{0.4cm} \alpha_3^0 & \overline{\alpha}_3^{\,0} & S_3
\end{vmatrix} 
(\overline{\alpha}_1^{\,0}S_{23}- \overline{\alpha}_2^{\,0} S_{13} 
+ \overline{\alpha}_3^{\,0}S_{12}) \\
&  \hspace{0.5cm} - \alpha_i^0 \begin{vmatrix} 
\overline{\alpha}_p^{\,0} & S_p \\ 
\rule{0cm}{0.4cm} \overline{\alpha}_q^{\,0} & S_q 
\end{vmatrix}
(\alpha_1^0S_{23}- \alpha_2^0 S_{13} + \alpha_3^0S_{12}) 
(\overline{\alpha}_1^{\,0}S_{23}- \overline{\alpha}_2^{\,0} S_{13} 
+ \overline{\alpha}_3^{\,0}S_{12}). 
 \end{align*}
Therefore, combined with the equality (\ref{eq8}), we have 
$$
2|A_0|(\beta_{ipq}^0-\beta_{iqp}^0) 
= -2(\alpha_{i1}^0h_{23pq}^0-\alpha_{i2}^0h_{13pq}^0+\alpha_{i3}^0h_{12pq}^0), 
$$
which completes the proof of Proposition~\ref{Prop_Codazzi_2}.
\hfill{$\Box$}

\medskip

\begin{rem}\label{Rem_rep}
Several expressions that were appeared in the proof of the main theorem possess their own 
representation theoretic meaning.
In this remark let $V$ be an abstract $3$-dimensional real vector space, 
and we denote by 
$V^*$ its dual space.
Let $S_{\lambda_1 \lambda_2 \lambda_3}(x_1,x_2,x_3)$ be the Schur function 
corresponding to the partition $\lambda=(\lambda_1,\lambda_2,\lambda_3)$ 
($\lambda_1 \geq \lambda_2 \geq \lambda_3 \geq 0$).
In the following we abbreviate it simply as $S_{\lambda_1 \lambda_2 \lambda_3}$.
It is well known that there is a one-to-one correspondence between the set of 
equivalence classes of $GL(V)$-irreducible polynomial representations and the set of 
Schur functions.
(For the definition of the Schur function, and the representation theory of the group 
$GL(V)$, see for example \cite{Mac}, \cite{A1}.)
We often confuse the Schur function with the corresponding representation space of 
$GL(V)$, by the abuse use of notation.

(1) 
Curvature tensor $R$ and its covariant derivative $S = \nabla R$ can be considered as elements 
of the spaces 
\begin{align*}
& \{ R \in S^2(\wedge^2 V^*)\: | \: \mathfrak{S}_{Y,Z,W} R(X,Y,Z,W)=0 \} \\
& \{ S \in S^2(\wedge^2 V^*) \otimes V^* \: | \: \mathfrak{S}_{Z,W,U} S(X,Y,Z,W,U)=0 \}, 
\end{align*}
respectively, where $S^2(\wedge^2 V^*)$ is a symmetric tensor product of the space 
$\wedge^2 V^*$ and $\mathfrak{S}_{Y,Z,W}$ means the cyclic sum with respect to $Y,Z,W$.
(Actually, in dimension three, the first Bianchi identity trivially holds, and we may omit it in 
the above definition.)
The general linear group $GL(V)$ acts naturally on these spaces.
These representations are both irreducible, and the corresponding Schur functions are 
$S_{22}$ and $S_{32}$, respectively.
(cf. \cite{KT}. 
Do not confuse the Schur function with the symbols $S_{ij}$ appeared in 
the decomposition of the tensor $S_{ijklm}=S_{ij}S_{kl}S_m$.)

The determinant $|\widetilde{R}|$ is a cubic polynomial of the curvature tensor $R$, 
and hence it is an element of the space $S^3(S_{22})$.
The $GL(V)$-irreducible decomposition of the space $S^3(S_{22})$ is given by the 
plethysm $S_3 \circ S_{22}$, which is explicitly given by 
$$
S_3 \circ S_{22} = S_{66} + S_{642} + S_{444}, 
$$
in case ${\rm dim} \: V = 3$ (cf. \cite{A1}, p.112).
Namely, it splits into three $GL(V)$-irreducible invariant subspaces.
(For the plethysm of Schur functions, see also \cite{Mac}.)
Among these spaces, the cubic polynomial $|\widetilde{R}|$ gives a basis of the 
$1$-dimensional representation space $S_{444}$ (\cite{A1}, p.131$\sim$132).
This fact implies that $|\widetilde{R}|$ is a relative invariant of $GL(V)$. 
(Note that the representation space $S_{\lambda_1 \lambda_2 \lambda_3}$ is 
$1$-dimensional if and only if $\lambda_1=\lambda_2=\lambda_3$.)

The determinant $|A|$ is a cubic polynomial of the symmetric $2$-tensor $\alpha_{ij}$.
Since $\alpha_{ij}$ is an element of the space $S_2$, $|A|$ is contained in the 
space $S^3(S_2)$.
Its $GL(V)$-irreducible decomposition is given by the plethysm 
$$
S_3 \circ S_2 = S_6 + S_{42} + S_{222}, 
$$
and actually $|A|$ is a basis of the $1$-dimensional space $S_{222}$, which implies that 
$|A|$ is also a relative invariant of $GL(V)$.

(2) 
Rivertz's six polynomials $r_1 \sim r_6$ are elements of the space $S^2(S_{22}) \otimes 
S_{32}$, since they are quadratic on $R$ and linear on $S$.
The $GL(V)$-irreducible decomposition of the space $S^2(S_{22})$ is given by the 
plethysm 
$$
S_2 \circ S_{22} = S_{44} + S_{422},
$$
and the tensor product with $S_{32}$ can be calculated by using the Littlewood-Richardson rule.
The result is given by 
\begin{align}
(S_{44} + S_{422}) S_{32} = & 
(S_{76} + S_{751} + S_{742} + S_{661} + S_{652} + S_{643}) \label{LR1} \\
& + (S_{742} + S_{652} + S_{643} + S_{553} + S_{544}) \notag \\
= & S_{76} + S_{751} + 2S_{742} + S_{661} + 2S_{652} + 2S_{643} \notag \\
& \hspace{1cm} + S_{553} + S_{544}, \notag
\end{align}
where the coefficients imply their multiplicities.
Among these spaces, Rivertz's six polynomials are contained in $S_{553}$, and give a basis 
of this $6$-dimensional representation space.
We can examine this fact by applying the method developed in \cite{A1}.
(See also \cite{R1}.)

We constructed the polynomial $RS(x_1,x_2,x_3)$, which plays an essential role in the proof 
of the main theorem.
This polynomial can be considered as an element of the  tensor product of the space $S_{553}$ 
and the space of quadratic polynomials $S_2$, consisting of $x_ix_j$.
The tensor product is decomposed into $GL(V)$-irreducible factors as 
$$
S_{553}S_2 = S_{753} + S_{654} + S_{555}.
$$ 
Among these three spaces $RS(x_1,x_2,x_3)$ is contained in $S_{555}$, and it is also 
a relative invariant of $GL(V)$.

(3) 
By substituting the Gauss equation and the derived Gauss equation into six 
polynomials $r_1 \sim r_6$, we obtain homogeneous polynomials of $\alpha_{ij}$ and 
$\beta_{ijk}$ with degrees five and one, respectively.
Hence, these polynomials are contained in the space $S^5(S_2) \otimes S_3$.
Its $GL(V)$-irreducible decomposition is given by 
\begin{align}
(S_5 \circ S_2)S_3 = & (S_{10} + S_{82} + S_{64} + S_{622} + S_{442}) S_3 \label{LR2} \\
= & S_{13} + S_{12\,1} + 2S_{11\,2} + 2S_{10\,3} + S_{10\,21} + 2S_{94} + S_{931} \notag \\ 
& \hspace{0.2cm} + 2S_{922} + 2S_{85} + 2S_{841} + 2S_{832} + S_{76} + S_{751} 
+ 3S_{742} \notag \\ 
& \hspace{0.2cm}  + S_{661} + 2S_{652} + 2S_{643} +S_{544}. \notag
\end{align}
The above substitution defines a $GL(V)$-equivariant map from the space $S^2(S_{22}) 
\otimes S_{32}$ to $(S_5 \circ S_2) \otimes S_3$.
Hence each irreducible component of $S^2(S_{22}) \otimes S_{32}$ is mapped to the 
irreducible subspace of $(S_5 \circ S_2) \otimes S_3$ with the same Schur function, or is 
mapped to zero.
Comparing (\ref{LR1}) with (\ref{LR2}), we know that the space $S_{553}$ in 
$S^2(S_{22}) \otimes S_{32}$ must be mapped to zero, since the space 
$S_{553}$ does not appear 
in $(S_5 \circ S_2) \otimes S_3$.
This gives a representation theoretic proof of Rivertz's result (Theorem~\ref{WTR} (2)).
The inverse formula of the Gauss equation is also explained in \cite{A1}, p.132 from the 
representation theoretic viewpoint.
\end{rem}

\bigskip

We finally state a corollary concerning local isometric embeddings of $M^3$ into the 
$4$-dimensional space of constant curvature $c$, as a direct consequence of our main 
theorem.
We put 
$$
\overline{R}_{ijkl} := R_{ijkl} - c(\delta_{ik}\delta_{jl}-\delta_{il}\delta_{jk}), 
$$
and define $|\overline{R}|$ by 
$$
|\overline{R}|:= \begin{vmatrix}
\overline{R}_{1212} & \overline{R}_{1213} & \overline{R}_{1223}\\
\overline{R}_{1213} & \overline{R}_{1313} & \overline{R}_{1323}\\
\overline{R}_{1223} & \overline{R}_{1323} & \overline{R}_{2323}
\end{vmatrix} 
= \begin{vmatrix}
R_{1212}-c & R_{1213} & R_{1223}\\
R_{1213} & R_{1313}-c & R_{1323}\\
R_{1223} & R_{1323} & R_{2323}-c
\end{vmatrix}.
$$
In addition, we substitute the above $\overline{R}_{ijkl}$ into $R_{ijkl}$ of Rivertz's six 
polynomials, and express them as $\overline{r}_1 \sim \overline{r}_6$.
For example, we have 
{\small 
\begin{align*}
& \overline{r}_1=\{R_{1213}(R_{2323}-c) - R_{1223}R_{1323}\}S_{12121}
- \{R_{1213}R_{1323} - R_{1223}(R_{1313}-c)\}S_{12122} \\
& - \{(R_{1212}-c)(R_{2323}-c)- R_{1223}{}^2\}S_{12131}
+ \{(R_{1212}-c)R_{1323} - R_{1213}R_{1223}\}S_{12132} \\
& +\{(R_{1212}-c)R_{1323} - R_{1213}R_{1223}\}S_{12231}
- \{(R_{1212}-c)(R_{1313}-c) - R_{1213}{}^2\}S_{12232}.
\end{align*}
}
Then we have the following corollary.

\begin{cor}\label{MainCor}
Let $(M^3,g)$ be a $3$-dimensional simply connected Riemannian manifold.

\smallskip

$(1)$ 
If $M^3$ can be isometrically embedded into the $4$-dimensional space of constant 
curvature $c$, then the inequality $|\overline{R}| \geq 0$ and the six equalities $\overline{r}_1 = 
\cdots = \overline{r}_6=0$ hold everywhere.

\smallskip

$(2)$ 
If $|\overline{R}| > 0$ and $\overline{r}_1 = \cdots = \overline{r}_6=0$ hold, then there exists 
a local isometric embedding of $M^3$ into the $4$-dimensional space of constant 
curvature $c$.
\end{cor}

\noindent
Since the fundamental theorem of hypersurfaces also holds in the case where the ambient 
space is the space of constant curvature (see, for example \cite{D}), we can easily verify 
the above corollary by replacing $R_{ijkl}$ by $\overline{R}_{ijkl}$ in the 
proof of Theorem ~\ref{WTR} and Proposition~\ref{Prop_R}.
Note that the Gauss and Codazzi equations are expressed as 
\begin{align*}
& R_{ijkl}-c(\delta_{ik}\delta_{jl}-\delta_{il}\delta_{jk}) 
= \alpha_{ik}\alpha_{jl}-\alpha_{il}\alpha_{jk}, \\
& \beta_{ijk} = \beta_{ikj}, 
\end{align*}
respectively in this case, where $\beta_{ijk} = (\nabla_{X_k}\alpha)(X_i,X_j)$ as before.
Since the space of constant curvature is locally symmetric, the derived Gauss 
equation takes the completely same form 
$$
S_{ijklm} = \beta_{ikm}\alpha_{jl}+\alpha_{ik}\beta_{jlm}
-\beta_{ilm}\alpha_{jk}-\alpha_{il}\beta_{jkm}.
$$
Concerning this subject, see also \cite{R3}, \cite{AH2}.

\bigskip
\bigskip

\section{Applications}

\subsection{Warped product manifold}
As an application of our main theorem, we consider warped product manifolds 
$M = B \times_f F$, where $B$ is the base, $F$ is the fiber and $f$ is 
a warping function.
We assume $f>0$ throughout.
(For details on warped product manifolds, see for example, O'Neill \cite{ON}, Chen \cite{Chen2}, 
Dajczer-Tojeiro \cite{DT}.)
In \S 5.1 we drop the assumption ${\rm dim}\:M = 3$.
In terms of local coordinates, the Riemannian metric of $M$ is expressed as 
$$
ds^2= \sum_{i,j} g_{ij}dx_idx_j + f(x)^2 \sum_{p,q} h_{pq} dy_pdy_q,
$$
where $g_{ij}$ and $h_{pq}$ are the Riemannian metrics of $B$ and $F$, respectively.
In the following we always assume that the indices move in the range $i,j,k,l,m = 1,2, \cdots, 
{\rm dim} \: B$ and $p,q,r,s,t = 1,2, \cdots, {\rm dim} \: F$.
We first consider the general dimensional case, not restricting to the three dimension.
We put 
$$
X_i = \frac{\partial \;\;}{\partial x_i}, \qquad Y_p= \frac{\partial \;\;}{\partial y_p},
$$
and let $\nabla^B_{X_i}X_j$, $\nabla^F_{Y_p}Y_q$ be the Levi-Civita connections of $B$ and 
$F$, respectively.
We may consider $X_i$, $Y_p$ as vector fields on $M$.
Then, as is well known (cf. \cite{ON}, \cite{Chen2}, \cite{DT}), the Levi-Civita connection 
$\nabla$ of $M$ is given by 
\begin{align*}
& \nabla_{X_i} X_j = \nabla^B_{X_i} X_j, \\
& \nabla_{X_i}Y_p = \nabla_{Y_p} X_i = f^{-1}f_iY_p, \\
& \nabla_{Y_p} Y_q = \nabla^F_{Y_p} Y_q -f\,h_{pq} \,{\rm grad}\: f,
\end{align*}
\noindent
where $f_i = \frac{\partial f}{\partial x_i}$, and ${\rm grad} \: f$ is the gradient of $f$, i.e., 
the vector field on $B$ satisfying $g({\rm grad}\: f,X) = Xf$ for any vector field $X$ on $B$.
In local coordinates, we have ${\rm grad}\: f = \sum_{i,j} g^{ij}f_iX_j$, where $(g^{ij})$ is 
the inverse matrix of $(g_{ij})$.
The gradient ${\rm grad}\: f$ can be naturally considered as a vector field on $M$.
We remark that $h_{pq}$ is given by $h(Y_p,Y_q)$ as a tensor field on $F$, and if we consider 
$Y_p$, $Y_q$ as tangent vector fields on $M$, their inner product is given by 
$f(x)^2h_{pq}$, depending on $x \in B$.

In terms of the curvatures $R^B$, $R^F$ of $B$ and $F$, respectively, the curvature $R$ of 
$M$ is given by the following (see \cite{ON}, \cite{Chen2}, \cite{DT}):

\begin{align*}
& R(X_i,X_j)X_k = R^B(X_i,X_j)X_k, \\
& R(X_i,X_j)Y_p = 0, \\
& R(X_i,Y_p)X_j = f^{-1} H^f(X_i,X_j)Y_p, \\
& R(X_i,Y_p)Y_q = -f \,h_{pq} (\nabla^B_{X_i} \,{\rm grad} \: f), \\
& R(Y_p,Y_q)X_i = 0, \\
& R(Y_p,Y_q)Y_r = R^F(Y_p,Y_q)Y_r - || {\rm grad}\: f ||^2 (h_{qr}Y_p - h_{pr} Y_q),
\end{align*}

\noindent
where $H^f(X_i,X_j)$ is the Hessian of $f$ defined by 
$$
H^f(X_i,X_j) = X_iX_jf - (\nabla^B_{X_i} X_j) f,
$$
and $|| {\rm grad}\: f ||^2= \sum_{i,j} g^{ij} f_if_j$ is the square of the norm of 
${\rm grad}\: f$.
Note that $H^f$ is a symmetric $2$-tensor field.
From the above formula, we immediately obtain the following formulas 
\begin{align*}
& R_{ijkl} = R^B_{ijkl}, \\
& R_{ijkp} = R_{ijpq} = R_{ipqr} = 0, \\
& R_{ipjq} = -f H^f_{ij} h_{pq}, \\
& R_{pqrs} = f^2 \{ R^F_{pqrs} - || {\rm grad}\: f ||^2 
(h_{pr}h_{qs} - h_{ps}h_{qr}) \}, 
\end{align*}
\noindent
where we put $H^f_{ij} = H^f(X_i,X_j)$.
In these equalities, for example, $R_{ipjq}$ implies $R(X_i,Y_p,X_j,Y_q)$ etc.
We remark that the equality 
$$
H^f(X_i,X_j) = g(\nabla^B_{X_i} {\rm grad} \: f , X_j) 
= g(X_i,  \nabla^B_{X_j} {\rm grad} \: f)
$$
holds.

We must further calculate the covariant derivative of $R$ in order to apply our main 
theorem.
The formula of $S = \nabla R$ for warped products  may be also well known.
But we cannot find it in literatures, and so we here calculate it directly from the definition 
$$
S_{ijklm} = \frac{\partial R_{ijkl}}{\partial x_m} - \sum_a \Gamma_{mi}^aR_{ajkl}
 - \sum_a \Gamma_{mj}^aR_{iakl}  - \sum_a \Gamma_{mk}^aR_{ijal} 
 - \sum_a \Gamma_{ml}^aR_{ijka}
$$
etc.
We exhibit our result for a later reference.
We denote by $S^B_{ijklm}$, $S^F_{pqrst}$ the components of the covariant derivative of 
the curvature of $B$ and $F$, respectively.
Then we have 

\begin{align*}
& S_{ijklm} = S^B_{ijklm}, \\
& S_{ijklp} = S_{ijkpl} = S_{ijpqk} = 
S_{ijpqr} = S_{ipjqr} = S_{ipqrj} = 0, \\
& S_{ijkpq} = ( f_j H^f_{ik}-f_iH^f_{jk}+f \sum_{l,m} f_l\,g^{lm}R_{ijkm} )h_{pq}, \\
& S_{ipjqk} = \{ f_k H^f_{ij} - f (\nabla^B_{X_k} H^f)(X_i,X_j) \} h_{pq}, \\
& S_{ipqrs} = -f f_i R^F_{spqr} + f ( f_i || {\rm grad}\: f ||^2 
- f \sum_{j,k} H^f_{ij}\,g^{jk}f_k ) (h_{pr}h_{qs}-h_{pq}h_{rs}), \\
& S_{pqrsi} = 2S_{iqrsp}, \\
& S_{pqrst} = f^2 S^F_{pqrst}, 
\end{align*}

\noindent
where $(\nabla^B_{X_k} H^f)(X_i,X_j)$ is the covariant derivative of the Hessian given by 
$$
(\nabla^B_{X_k} H^f)(X_i,X_j) = X_k(H^f(X_i,X_j))-H^f(\nabla^B_{X_k} X_i,X_j) 
-H^f(X_i,\nabla^B_{X_k} X_j).
$$
We remark that the following equalities hold:
\begin{align*}
& \frac{1}{\,2\,} \frac{\partial \;\;}{\partial x_i} || {\rm grad}\: f ||^2 = H^f(X_i,{\rm grad}\: f) 
= \sum_{j,k} H^f_{ij} g^{jk}f_k, \\
& \sum_{l,m} f_l\,g^{lm}R^B_{ijkm} = -(R^B(X_i,X_j)X_k)f.
\end{align*}
\noindent
To prove these equalities, we frequently use the property $\nabla^B g = \nabla^F h =0$.
To confirm the second Bianchi identity $S_{ipjkq}+S_{ipkqj}+S_{ipqjk}=0$, we use 
the Ricci formula 
$$
(\nabla^B_{X_j} H^f)(X_k,X_i) - (\nabla^B_{X_k} H^f)(X_j,X_i)
= -(R^B(X_j,X_k)X_i)f.
$$

\medskip

\subsection{First application}
As a first application of Theorem~\ref{MainTh}, we consider the case ${\rm dim} 
\: B = 1$ and 
${\rm dim}\: F = 2$.
We slightly change the letters as follows.
$$
ds^2 = dx_1{}^2 + f(x_1)^2 (E dx_2{}^2+2Fdx_2dx_3+Gdx_3{}^2).
$$
Here, $E,F,G$ are functions of $x_2, x_3$, which define a Riemannian metric of the 
fiber $F$.
(Note that we can normalize the metric of $B$ as $dx_1{}^2$ by changing the variable suitably.
Also, do not confuse the fiber $F$ with the function $F(x_2,x_3)$.)
Now, we show the following theorem.

\begin{thm}\label{Appl1}
Under the above notations, if $M^3$ can be isometrically embedded into 
$\mathbb{R}^4$ and $f^{\prime \prime} \neq 0$, then the fiber $F$ is a 
space of positive constant curvature $K$, satisfying the inequality $K \geq f^{\prime}{}^2$, 
where $K$ is the Gaussian curvature of $F$.

Conversely, if $F$ is a space of positive constant curvature $K$, satisfying the inequality 
$K > f^{\prime}{}^2$, then $M^3$ can be locally isometrically embedded into $\mathbb{R}^4$, 
including the points where $f^{\prime \prime}=0$.
\end{thm} 

\begin{proof}
Applying the formula stated above, we have 
\begin{align}
& R_{1212} = -ff^{\prime \prime} E, \quad 
R_{1213} = -ff^{\prime \prime} F, \quad
R_{1313} = -ff^{\prime \prime} G, \notag \\
& R_{1223} = R_{1323} = 0, \notag \\ %\label{appleq2} \\
& R_{2323} = f^2(K-f^{\prime}{}^2) \Delta, \notag
\end{align}
where $\Delta = EG-F^2>0$.
The Gaussian curvature $K$ of $F$ is given by $K = R^F_{2323}/\Delta$.

The covariant derivative $S = \nabla R$ is given by 
\begin{align*}
& S_{12121} = (f^{\prime} f^{\prime \prime} - ff^{\prime \prime \prime})E, \:\:
S_{12131} = (f^{\prime} f^{\prime \prime} - ff^{\prime \prime \prime})F, \:\:
S_{13131} = (f^{\prime} f^{\prime \prime} - ff^{\prime \prime \prime})G, \\
& S_{12122} = S_{12123} = S_{12132} = S_{12133} = S_{12231} = 0, \\
& S_{12232} = S_{13132} = S_{13133} = S_{13231} = S_{13233} = 0, \\
& S_{13232} = ff^{\prime}(f^{\prime}{}^2-ff^{\prime \prime} -K)\Delta, \quad 
S_{12233} = - S_{13232}, \quad S_{23231} = 2 S_{13232}, \\
& S_{23232} = f^2S^F_{23232}, \quad S_{23233} = f^2 S^F_{23233}.
\end{align*}
Then we have 
$$
| \widetilde{R} | = 
\begin{vmatrix}
R_{1212} & R_{1213} & 0 \\
R_{1213} & R_{1313} & 0 \\
0 & 0 & R_{2323} 
\end{vmatrix} 
= f^4f^{\prime \prime}{}^2(K-f^{\prime}{}^2) \Delta^2.
$$
Further, concerning Rivertz's six polynomials, we have 
\begin{align*}
& r_3 = -f^4f^{\prime \prime}{}^2\,S^F_{23232} \,\Delta, \\
& r_5 = -f^4f^{\prime \prime}{}^2\,S^F_{23233} \,\Delta, 
\end{align*}
and the remaining four polynomials are identically zero.
We can easily examine these facts, since many components of $R$ and $S$ are zero.

Since $f>0$, $f^{\prime \prime} \neq 0$ and $\Delta>0$, we have 
$S^F_{23232} = S^F_{23233} =0$ from $r_3 = r_5 = 0$, which implies that $F$ is locally 
symmetric. 
Since $F$ is $2$-dimensional, it implies that $F$ is a space of constant curvature.
In addition, from the condition $|\widetilde{R}| \geq 0$, we have $K \geq f^{\prime}{}^2 \geq 0$.
If $K=0$, then we have $f^{\prime} \equiv 0$, which contradicts the assumption 
$f^{\prime \prime} \neq 0$.
Hence $K$ is positive.

Conversely, in case $F$ is a space of positive constant curvature $K$, satisfying 
the inequality $K>f^{\prime}{}^2$, we can explicitly construct a local isometric embedding 
as follows, including the case $f^{\prime \prime}=0$.

We express $K = 1/r^2$, where $r$ is a positive constant.
Then $F$ can be considered as an open subset of the $2$-dimensional sphere with radius 
$r$.
Let 
$$
(x_2,x_3) \mapsto (a(x_2,x_3),b(x_2,x_3),c(x_2,x_3))
$$
be a local isometric embedding of $F$ into $\mathbb{R}^3$ with $a^2+b^2+c^2=r^2$.
By using a function $p(x_1)$, we construct a map $\varphi : M^3 \rightarrow \mathbb{R}^4$ 
by
$$
\varphi(x_1,x_2,x_3) = (p(x_1),f(x_1)a(x_2,x_3),f(x_1)b(x_2,x_3),f(x_1)c(x_2,x_3)).
$$
Then, we can easily see that $\varphi$ gives an isometric embedding if and only if the 
condition 
\begin{align}
p^{\prime}(x_1)^2+r^2f^{\prime}(x_1)^2=1 \label{appleq1}
\end{align}
is satisfied.
Note that to check this result, we use the equalities 
$$
a\frac{\partial a\;}{\partial x_2}+b\frac{\partial b\;}{\partial x_2}+
c\frac{\partial c\;}{\partial x_2}=
a\frac{\partial a\;}{\partial x_3}+b\frac{\partial b\;}{\partial x_3}+
c\frac{\partial c\;}{\partial x_3}=0, 
$$
which are obtained from $a^2+b^2+c^2=r^2$.
Then, if the inequality $r^2f^{\prime}{}^2<1$ holds, we can easily find a function $p(x_1)$ 
satisfying the equality (\ref{appleq1}) in an open interval.
Clearly the inequality $r^2f^{\prime}{}^2<1$ is equivalent to $K>f^{\prime}{}^2$.
\end{proof}

\begin{rem}\label{Appl2}
(1) 
We remark that in case $f^{\prime \prime} = 0$, we cannot show the converse part 
of this theorem by only applying our main theorem, because $|\widetilde{R}|=0$ in this case.
In other words, there is a case where a local isometric embedding exists even if  
$|\widetilde{R}|=0$ and non-flat.
In case $|\widetilde{R}|=0$, it is in general hard to determine the existence or non-existence 
of local isometric embeddings of $M^3$ into $\mathbb{R}^4$.

(2) 
Consider the case where the fiber $F$ is flat.
Then $M^3$ can be locally isometrically embedded into $\mathbb{R}^4$ 
if and only if 
$f(x_1)=px_1+q$ with $q>0$.
In fact, if the local isometric embedding exists, then from the conditions $K=0$ and 
$|\widetilde{R}|\geq 0$ we have $f^{\prime}{}^2f^{\prime \prime}{}^2 \leq 0$, which 
is equivalent to $f^{\prime}$=constant.
Hence $f(x_1)$ takes the above form.
(Rivertz's six equalities are automatically satisfied since $f^{\prime \prime} = 0$.)

Conversely, assume $f(x_1)=px_1+q$.
Since the fiber $F$ is flat, we may assume that the metric of $M^3$ is expressed as 
$$
ds^2=dx_1{}^2+f(x_1)^2(dx_2{}^2+dx_3{}^2).
$$
If $p=0$, then the space $M^3$ is flat, and it can be locally isometrically embedded into 
$\mathbb{R}^3$.
In case $p \neq 0$, we choose non-zero constants $k$, $l$ satisfying $p^2(k^2+l^2)=1$.
Then the map $\varphi : M^3 \rightarrow \mathbb{R}^4$ defined by 
$$
\varphi(x_1,x_2,x_3) = \left( kf(x_1) \cos \frac{x_2}{k}, kf(x_1) \sin \frac{x_2}{k}, 
lf(x_1) \cos \frac{x_3}{l}, lf(x_1) \sin \frac{x_3}{l} \right)
$$
gives a local isometric immersion of $M^3$ into $\mathbb{R}^4$.
Remark that if the domain of $(x_2,x_3)$ is sufficiently large, the map 
$\varphi$ is not embedding since $\cos$ and $\sin$ are periodic.
\end{rem}

\bigskip

\subsection{Monge-Amp\`{e}re equation}
Next, as a second application of Theorem~\ref{MainTh}, we consider the case 
${\rm dim} \: B = 2$ and ${\rm dim}\: F = 1$.
In this case, the metric of $M^3$ is expressed as 
$$
ds^2 = E dx_1{}^2+2Fdx_1dx_2+Gdx_2{}^2+ f(x_1,x_2)^2 dx_3{}^2,
$$
where, $E,F,G$ are functions of $x_1, x_2$.
Note that in case $f$ is constant, $M^3$ can be considered as a product of 
a $2$-dimensional Riemannian manifold and a $1$-dimensional flat space, and it can be 
always locally realized in $\mathbb{R}^4$, 
provided that $E$, $F$, $G$ are real analytic functions.
It is a consequence of classical Janet-Cartan's local isometric embedding theorem for the 
$2$-dimensional case.

Now we denote by $K$ the Gaussian curvature of $E dx_1{}^2+2Fdx_1dx_2+Gdx_2{}^2$, and 
$H^f_{ij}$ is the component of the Hessian of $f$.
Then, we have the following theorem.

\begin{thm}\label{Appl3}
Under the above notations, assume that $K \neq 0$ everywhere.
If $M^3$ can be isometrically embedded into $\mathbb{R}^4$, then the following 
two conditions hold.

\smallskip

$({\rm i})$ $(H^f_{11}H^f_{22}-H^f_{12}{}^2)K \geq 0$.

\smallskip

$({\rm ii})$ $H^f_{11}H^f_{22}-H^f_{12}{}^2+K(f_1{}^2G-2f_1f_2F+f_2{}^2E) = cK \Delta$ 

\noindent
for some non-negative constant $c$, and 
$f_i = \frac{\partial f}{\partial x_i}$ $(i=1,2)$.

\smallskip

Conversely, if the inequality $(H^f_{11}H^f_{22}-H^f_{12}{}^2)K > 0$ holds and the 
equality $({\rm ii})$ holds for some constant $c$, then $c>0$ and $M^3$ can 
be locally isometrically embedded into $\mathbb{R}^4$.
\end{thm}

\begin{proof}
To apply Theorem~\ref{MainTh}, we must calculate the curvature and its covariant 
derivative.
The curvature is given by 
\begin{align*}
& R_{1212} = K \Delta, \quad
R_{1213} = R_{1223} = 0, \\
& R_{1313} = -fH^f_{11}, \quad
R_{1323} = -fH^f_{12}, \quad
R_{2323} = -fH^f_{22}.
\end{align*}
The components $H^f_{ij}$ of the Hessian are explicitly given by 
\begin{align*}
H^f_{11} & = f_{11} - \Gamma^1_{11}f_1-\Gamma_{11}^2f_2, \\
H^f_{12} & = f_{12}-\Gamma_{12}^1f_1-\Gamma_{12}^2f_2, \\
H^f_{22} & = f_{22}-\Gamma_{22}^1f_1-\Gamma_{22}^2f_2, 
\end{align*}
where $f_{11} = \frac{\partial^2 f}{\partial x_1{}^2}$ etc, and $\nabla^B_{X_i} X_j = 
\sum_k \Gamma_{ij}^k X_k$.
The components $\Gamma_{ij}^k$ are given by 
\begin{align*}
& \Gamma^1_{11} = \frac{1}{2\Delta} (GE_1-2FF_1+FE_2), \quad
\Gamma^2_{11} = -\frac{1}{2\Delta} (FE_1-2EF_1+EE_2), \\
& \Gamma^1_{12} = -\frac{1}{2\Delta} (FG_1-GE_2), \quad
\Gamma^2_{12} = \frac{1}{2\Delta} (EG_1-FE_2), \\
& \Gamma^1_{22} = -\frac{1}{2\Delta} (GG_1-2GF_2+FG_2), \quad
\Gamma^2_{22} = \frac{1}{2\Delta} (FG_1-2FF_2+EG_2), 
\end{align*}
where $E_1 = \frac{\partial E}{\partial x_1}$ etc. 

The covariant derivative of the curvature is given by 
\begin{align*}
& S_{12121} = S^B_{12121}, \qquad 
S_{12122} = S^B_{12122}, \\
& S_{12123} = S_{12131} = S_{12132} = S_{12231} = 0, \\
& S_{12232} = S_{13133} = S_{13233} = S_{23233} = 0, \\
& S_{12133} = f_2H^f_{11}-f_1H^f_{12}-f(f_1F-f_2E)K, \\
& S_{12233} = f_2H^f_{12}-f_1H^f_{22}-f(f_1G-f_2F)K, \\
& S_{13131} = f_1H^f_{11}-fH^f_{111}, \qquad 
S_{13132} = f_2H^f_{11}-fH^f_{112}, \\
& S_{13231} = f_1H^f_{12}-fH^f_{121}, \qquad 
S_{13232} = f_2H^f_{12}-fH^f_{122}, \\
& S_{23231} = f_1H^f_{22}-fH^f_{221}, \qquad 
S_{23232} = f_2H^f_{22}-fH^f_{222},
\end{align*}
where $S^B_{12121}$, $S^B_{12122}$ are the components of the covariant derivative 
of the curvature of $B$, and $H^f_{ijk}$ is the component of the covariant derivative 
of $H^f$ given by 
$$
H^f_{ijk} = (\nabla^B_{X_k} H^f)(X_i,X_j).
$$
Note that the equality $H^f_{ijk} = H^f_{jik}$ holds, and we have 
$$
S^B_{12121} = K_1 \Delta, \qquad S^B_{12122} = K_2 \Delta, 
$$
where $K_i = \frac{\partial K}{\partial x_i}$ ($i=1,2$). 
To check these equalities on $S^B$, we use the properties 
$$
\Gamma^1_{11} + \Gamma^2_{12} = \frac{\Delta_1}{2\Delta}, \qquad
\Gamma^1_{12} + \Gamma^2_{22} = \frac{\Delta_2}{2\Delta},
$$
where $\Delta_i = \frac{\partial \Delta}{\partial x_i}$.

Concerning the determinant $|\widetilde{R}|$, we have 
$$
|\widetilde{R}| = \begin{vmatrix}
R_{1212} & 0 & 0 \\
0 & R_{1313} & R_{1323} \\
0 & R_{1323} & R_{2323} 
\end{vmatrix}
= f^2(H_{11}^f H^f_{22}-H^f_{12}{}^2)K \Delta.
$$
Rivertz's six polynomials are given by $r_1=r_4=r_5=r_6=0$ and
\begin{align*}
r_2 = & f^2 \{ (H_{11}^f H^f_{22}-H^f_{12}{}^2)K_1
-(H_{22}^fH^f_{111}-2H^f_{12}H^f_{112}+H^f_{11}H^f_{122})K \\
&\hspace{0.5cm}  -(f_1G-f_2F)H^f_{11}K^2 \} \Delta,\\
r_3 = & f^2 \{ (H_{11}^f H^f_{22}-H^f_{12}{}^2)K_2
-(H_{22}^fH^f_{121}-2H^f_{12}H^f_{221}+H^f_{11}H^f_{222})K \\
&\hspace{0.5cm}  +(f_1F-f_2E)H^f_{22}K^2 \} \Delta.
\end{align*}
We can easily check these facts, since many components of $R$ and $S$ vanish.

Now, we show that two conditions $r_2=r_3=0$ are equivalent to 
\begin{align*}
H^f_{11}H^f_{22}-H^f_{12}{}^2+K(f_1{}^2G-2f_1f_2F+f_2{}^2E) = cK \Delta 
\end{align*}
for some constant $c$.
For this purpose, we have only to prove the following two equalities 
\begin{align}
& \frac{\partial \;\;}{\partial x_1} 
\left( \frac{H^f_{11}H^f_{22}-H^f_{12}{}^2+K(f_1{}^2G-2f_1f_2F+f_2{}^2E)}{K \Delta}
\right)
= -\frac{r_2}{f^2K^2 \Delta^2} , \label{appleq2_1} \\
& \frac{\partial \;\;}{\partial x_2} 
\left( \frac{H^f_{11}H^f_{22}-H^f_{12}{}^2+K(f_1{}^2G-2f_1f_2F+f_2{}^2E)}{K \Delta}
\right) 
= -\frac{r_3}{f^2K^2 \Delta^2}. \label{appleq2_2}
\end{align}

We first calculate $(H^f_{11}H^f_{22}-H^f_{12}{}^2)_1$, where the subscript $1$ 
means the partial derivative by $x_1$.
From the definition of $H^f_{ijk}$ we have 
\begin{align*}
& H^f_{111} = \frac{\partial H^f_{11}}{\partial x_1} -2\Gamma^1_{11}H^f_{11}
-2\Gamma^2_{11}H^f_{12}, \\
& H^f_{121} = \frac{\partial H^f_{12}}{\partial x_1} -\Gamma^1_{11}H^f_{12}
-\Gamma^2_{11}H^f_{22}-\Gamma^1_{12}H^f_{11}
-\Gamma^2_{12}H^f_{12}, \\
& H^f_{221} = \frac{\partial H^f_{22}}{\partial x_1} -2\Gamma^1_{12}H^f_{12}
-2\Gamma^2_{12}H^f_{22}.
\end{align*}
Next, by using the Ricci formula 
$$
H^f_{kij} - H^f_{jik} = -(R^B(X_j,X_k)X_i)f = \sum_{l,m} g^{lm}R^B_{jkil}f_m,
$$
we have 
\begin{align*}
H^f_{112} & = H^f_{211}+\sum_{l,m} g^{lm}R^B_{211l}f_m \\
& = H^f_{121}-\sum_{l,m} g^{lm}R^B_{121l}f_m \\
& = H^f_{121}-\sum_{m} g^{2m}R^B_{1212}f_m \\
& = H^f_{121}-(g^{21}f_1 +g^{22}f_2)R^B_{1212} \\
& = H^f_{121}+\frac{1}{\Delta}(f_1F-f_2E)K \Delta \\
& = \frac{\partial H^f_{12}}{\partial x_1} -\Gamma^1_{11}H^f_{12}
-\Gamma^2_{11}H^f_{22}-\Gamma^1_{12}H^f_{11}-\Gamma^2_{12}H^f_{12}+(f_1F-f_2E)K.
\end{align*}
In a similar way, we have 
$$
H^f_{122} = \frac{\partial H^f_{22}}{\partial x_1} -2\Gamma^1_{12}H^f_{12}
-2\Gamma^2_{12}H^f_{22}+(f_1G-f_2F)K.
$$
Thus we have 
\begin{align*}
& (H^f_{11}H^f_{22}-H^f_{12}{}^2)_1 = \frac{\partial H^f_{11}}{\partial x_1}H^f_{22}
+ H^f_{11}\frac{\partial H^f_{22}}{\partial x_1}
-2H^f_{12}\frac{\partial H^f_{12}}{\partial x_1} \\
= & (H^f_{111}+2\Gamma^1_{11}H^f_{11}+2\Gamma^2_{11}H^f_{12})H^f_{22} \\
& \hspace{0.5cm} +H^f_{11} \{ H^f_{122} +2\Gamma^1_{12}H^f_{12}
+2\Gamma^2_{12}H^f_{22}-(f_1G-f_2F)K \} \\
& \hspace{0.5cm} -2H^f_{12} \{ H^f_{112} +\Gamma^1_{11}H^f_{12}
+\Gamma^2_{11}H^f_{22}+\Gamma^1_{12}H^f_{11}+\Gamma^2_{12}H^f_{12}
-(f_1F-f_2E)K \} \\
= & (H^f_{22}H^f_{111}-2H^f_{12}H^f_{112}+H^f_{11}H^f_{122})
+2(\Gamma^1_{11}+\Gamma^2_{12})(H^f_{11}H^f_{22}-H^f_{12}{}^2) \\
& \hspace{0.5cm} -(f_1G-f_2F)H^f_{11}K+2(f_1F-f_2E)H^f_{12}K \\
= & (H^f_{22}H^f_{111}-2H^f_{12}H^f_{112}+H^f_{11}H^f_{122})
+\frac{\Delta_1}{\Delta}(H^f_{11}H^f_{22}-H^f_{12}{}^2) \\
& \hspace{0.5cm} -(f_1G-f_2F)H^f_{11}K+2(f_1F-f_2E)H^f_{12}K.
\end{align*}
Therefore, the numerator of 
$$
\frac{\partial \;\;}{\partial x_1} 
\left( \frac{H^f_{11}H^f_{22}-H^f_{12}{}^2}{K \Delta} \right)
$$
is equal to 
\begin{align}
& (H^f_{11}H^f_{22}-H^f_{12}{}^2)_1 K \Delta - (H^f_{11}H^f_{22}-H^f_{12}{}^2)K_1 \Delta 
-(H^f_{11}H^f_{22}-H^f_{12}{}^2) K \Delta_1 \label{appleq3} \\
= & (H^f_{22}H^f_{111}-2H^f_{12}H^f_{112}+H^f_{11}H^f_{122}) K \Delta 
- (H^f_{11}H^f_{22}-H^f_{12}{}^2)K_1 \Delta \notag \\
& \hspace{0.5cm}  -\{ (f_1G-f_2F)H^f_{11}-2(f_1F-f_2E)H^f_{12} \} K^2 \Delta. \notag
\end{align}
Next, we calculate the numerator of 
$$
\frac{\partial \;\;}{\partial x_1} \left(
\frac{f_1{}^2G-2f_1f_2F+f_2{}^2E}{\Delta} \right).
$$
It is equal to 
\begin{align*}
& (f_1{}^2G-2f_1f_2F+f_2{}^2E)_1 \Delta - (f_1{}^2G-2f_1f_2F+f_2{}^2E) \Delta_1 \\
= & \{ 2(f_1G-f_2F)f_{11}-2(f_1F-f_2E)f_{12}+f_1{}^2G_1-2f_1f_2F_1+f_2{}^2E_1 \} \Delta \\
& \hspace{0.5cm}  - (f_1{}^2G-2f_1f_2F+f_2{}^2E) \Delta_1.
\end{align*}
Here, we substitute 
\begin{align*}
& f_{11} = H^f_{11}+\Gamma^1_{11}f_1+\Gamma^2_{11}f_2, \\
& f_{12} = H^f_{12}+\Gamma^1_{12}f_1+\Gamma^2_{12}f_2
\end{align*}
into the above.
Then, it finally becomes 
\begin{align}
2 (f_1G-f_2F)H^f_{11} \Delta  - 2(f_1F-f_2E)H^f_{12} \Delta. \label{appleq4}
\end{align}
In fact, it can be easily seen that the terms concerning $f_1{}^2$, $f_1f_2$, $f_2{}^2$ all vanish.
For example, the coefficient of $f_1{}^2$ is equal to 
\begin{align*}
& 2(\Gamma^1_{11}G-\Gamma^1_{12}F) \Delta +G_1\Delta-G\Delta_1 \\
= & G(GE_1-2FF_1+FE_2)+F(FG_1-GE_2)
-(F^2G_1-2FGF_1+G^2E_1) \\
= & 0.
\end{align*}
Adding the above two equalities (\ref{appleq3}), (\ref{appleq4}), we 
obtain (\ref{appleq2_1}).
We can prove (\ref{appleq2_2}) in the same way.
Thus we have 
$$
H^f_{11}H^f_{22}-H^f_{12}{}^2+K(f_1{}^2G-2f_1f_2F+f_2{}^2E) = cK \Delta
$$ 
for some constant $c$.

In case the function $f(x_1,x_2)$ satisfies this equation, we have 
$$
|\widetilde{R}| = f^2K^2 \Delta \{c \Delta - (f_1{}^2G-2f_1f_2F+f_2{}^2E) \}.
$$
Then if the conditions $| \widetilde{R} | \geq 0$ and $K \neq 0$ hold, we have 
$$
c \Delta  \geq f_1{}^2G-2f_1f_2F+f_2{}^2E \geq 0,
$$
which implies that the constant $c$ is non-negative.
If $|\widetilde{R}| >0$ and $K \neq 0$, we have $c>0$.

Now Theorem~\ref{Appl3} follows immediately from Theorem~\ref{MainTh} and 
Theorem~\ref{WTR}.
\end{proof}

Note that the equality (ii) in Theorem~\ref{Appl3} is a second order partial differential 
equation on $f(x_1,x_2)$ that is called the Monge-Amp\`{e}re equation.
We remark that if we multiply the function $f$ by a positive constant $k$, then the constant $c$ 
is multiplied by $k^2$.
Hence, we can freely change the constant $c$, provided it is positive.
 
\medskip

\begin{exam}\label{Appl5}
We consider a metric defined in a neighborhood of the origin 
$$
ds^2 = Edx_1{}^2+dx_2{}^2+f(x_1,x_2)^2dx_3{}^2
$$
with $E = 1+2ax_1+2bx_2$ $(b \neq 0)$ and $f(x_1,x_2) > 0$.
The Gaussian curvature of the base metric $Edx_1{}^2+dx_2{}^2$ is given by $K = b^2/E^2
>0$.

Applying Theorem~\ref{Appl3}, we know that if this metric can be locally isometrically 
embedded into $\mathbb{R}^4$, then we have 
\begin{align}
& f_{11}f_{22}-f_{12}{}^2-\frac{a}{E}f_1f_{22}+bf_2f_{22}+\frac{2b}{E}f_1f_{12}
+\frac{b^2}{E}f_2{}^2 = \frac{b^2c}{E},  \label{appleq9_1} \\
& \frac{1}{E}f_1{}^2 + f_2{}^2 \leq c \label{appleq9_2}
\end{align}
for some non-negative constant $c$.
Conversely, if the function $f(x_1,x_2)$ satisfies the Monge-Amp\`{e}re equation 
(\ref{appleq9_1}) and the inequality $\frac{1}{E}f_1{}^2 + f_2{}^2 < c$ near the origin, then 
the above metric can be locally isometrically embedded into $\mathbb{R}^4$.
Note that the last inequality is equivalent to $f_1(0,0)^2+f_2(0,0)^2<c$.

In fact, we have 
\begin{align*}
& \Gamma_{11}^1= \frac{a}{E}, \;\;\; \Gamma_{11}^2 = -b, \;\;\; \Gamma_{12}^1 = \frac{b}{E}, 
\;\;\; \Gamma_{12}^2 = \Gamma_{22}^1 = \Gamma_{22}^2 = 0
\end{align*}
and 
\begin{align*}
H_{11}^f = f_{11} - \frac{a}{E}f_1+bf_2, \;\;\; H_{12}^f = f_{12}-\frac{b}{E}f_1, 
\;\;\; H_{22}^f = f_{22}.
\end{align*}
On the other hand, we have 
$$
K(f_1{}^2G-2f_1f_2F+f_2{}^2E) = \frac{b^2}{E^2}(f_1{}^2+f_2{}^2E).
$$
From these equalities it is easy to see that the condition (ii) in Theorem~\ref{Appl3} 
is equivalent to the above (\ref{appleq9_1}).

Concerning the inequality (\ref{appleq9_2}), from the Monge-Amp\`{e}re equation, we have 
\begin{align*}
& K(H_{11}^fH_{22}^f-H_{12}^f{}^2) \\
= & cK^2\Delta - K^2 (f_1{}^2G-2f_1f_2F+f_2{}^2E) \\
= & \frac{b^4c}{E^3}-\frac{b^4}{E^4}(f_1{}^2+f_2{}^2E) \\
= & \frac{b^4}{E^3} \left( c-\frac{1}{E}f_1{}^2-f_2{}^2 \right).
\end{align*}
Then the inequality (\ref{appleq9_2}) follows immediately from the condition (i) 
in Theorem~\ref{Appl3} since $b \neq 0$ and $E>0$.
The converse part can be examined in a similar way.

To solve the Monge-Amp\`{e}re equation (\ref{appleq9_1}) satisfying $f_1(0,0)^2
+f_2(0,0)^2<c$ is another hard problem.
But it actually possesses a solution such as 
$$
\displaystyle{f(x_1,x_2) = \frac{1}{\,2\,} \int_{-\frac{d}{c}}^{2ax_1+2bx_2} 
\sqrt{\frac{ct+d}{a^2+b^2(t+1)}}\, dt},
$$
where $d$ is a constant satisfying $0<d < c$.
It can be easily verified that the function $f(x_1,x_2)$ satisfies the above two conditions.

The function 
$$
f(x_1,x_2) = \varphi(x_1)+kx_2
$$
also gives a solution of (\ref{appleq9_1}) for any function $\varphi(x_1)$ if $k^2=c$.
But concerning the inequality, this function satisfies $f_1(0,0)^2+f_2(0,0)^2 
= \varphi^{\prime}(0)^2+c \geq c$, and does not satisfy the desired inequality 
$f_1(0,0)^2+f_2(0,0)^2<c$ at the origin.

As another example, in case $\psi^{\prime}(x_1) \neq 0$, the function 
$$
f(x_1,x_2) = \varphi(x_1) + \psi(x_1)x_2
$$
gives a solution of (\ref{appleq9_1}) if and only if 
$$
\varphi^{\prime}(x_1) = \displaystyle{\frac{1}{2b \,\psi^{\prime}(x_1)} 
\left\{ (1+2ax_1)\psi^{\prime}(x_1)^2-b^2(\psi(x_1)^2-c) \right\}}.
$$
In this case we see that the inequality $f_1(0,0)^2+f_2(0,0)^2 <c$ is equivalent to 
$$
\displaystyle{\left\{ \psi^{\prime}(0)^2+b^2(\psi(0)^2-c) \right\}^2 < 0},
$$
which is also not the case for any choice of the function $\psi(x_1)$.
\end{exam}

\medskip

\begin{exam}\label{Appl6}
As a next application of Theorem~\ref{Appl3}, we consider the following metric 
$$
ds^2 = E(x_1)dx_1{}^2+2F(x_1)dx_1dx_2+G(x_1)dx_2{}^2+f(x_1)^2dx_3{}^2, 
$$
all of whose components are functions of one variable $x_1$.
We assume that this metric can be isometrically embedded into $\mathbb{R}^4$, 
under the conditions $K \neq 0$ and $G^{\prime} \neq 0$.

We can easily verify the equalities 
\begin{align*}
& H^f_{11}H^f_{22}-H^f_{12}{}^2 = \frac{GG^{\prime}}{2\Delta}f^{\prime} f^{\prime \prime} 
- \frac{G^{\prime}}{4\Delta^2} (G^2E^{\prime}-2FGF^{\prime}+F^2G^{\prime})
f^{\prime}{}^2, \\ 
& K =\frac{G^{\prime}(GE^{\prime}-2FF^{\prime}+EG^{\prime})}{4\Delta^2}
- \frac{G^{\prime \prime}}{2\Delta}
=  -\frac{1}{4G^{\prime}} \left( \frac{G^{\prime}{}^2}{\Delta} \right)^{\prime}.
\end{align*}
Then the equation (ii) in Theorem~\ref{Appl3} becomes the form 
$$
2GG^{\prime}f^{\prime}f^{\prime \prime} + (G^{\prime}{}^2-2GG^{\prime \prime})
f^{\prime}{}^2 =
c \{ G^{\prime}(GE^{\prime}-2FF^{\prime}+EG^{\prime})-2G^{\prime \prime}\Delta \}.
$$
It is easy to see that this is equivalent to 
$$
\left( \frac{Gf^{\prime}{}^2}{G^{\prime}{}^2} \right)^{\prime} 
= c \left( \frac{EG-F^2}{G^{\prime}{}^2} \right)^{\prime}.
$$
Hence we have 
$$
\frac{Gf^{\prime}{}^2}{G^{\prime}{}^2} 
= \frac{c\Delta}{G^{\prime}{}^2}-A 
$$
for some constant $A$, which implies 
$$
f^{\prime}{}^2 =\frac{c\,\Delta-A\,G^{\prime}{}^2}{G}.
$$
If $f$ satisfies this condition, we have 
$$
H^f_{11}H^f_{22}-H^f_{12}{}^2 = AG^{\prime}{}^2K
$$
after some calculations, and hence we have $K(H^f_{11}H^f_{22}-H^f_{12}{}^2) 
= AG^{\prime}{}^2K^2$, which implies that $A$ must be non-negative from the condition 
(i) in Theorem~\ref{Appl3}.

Conversely, if $A$ is positive and $f$ satisfies the equality $f^{\prime}{}^2 
=(c \Delta - AG^{\prime}{}^2)/G$, 
then we have $c >0$, and the above metric can be locally isometrically embedded into 
$\mathbb{R}^4$.
For example, in the special case where $c=1$ and $F=0$, the embedding is explicitly given by 
\begin{align*}
& \varphi(x_1,x_2,x_3) \\
& \;\;\; = \left(2\sqrt{AG(x_1)}\,\cos \frac{x_2}{2\sqrt{A}},
\,2\sqrt{AG(x_1)}\,\sin \frac{x_2}{2\sqrt{A}},\,f(x_1) \cos x_3, \,f(x_1) \sin x_3 \right).
\end{align*}
Precisely, we must restrict the domain of $(x_2,x_3)$ to make $\varphi$ 
an embedding.

In addition, if $E(x_1) \equiv 1$ and $G(x_1)=\cos^2 ax_1$ for a positive constant 
$a$, the base manifold $B$ is a space of constant positive curvature with $K=a^2$.
The warping function $f(x_1)$ must be of the form 
$$
f(x_1) = \int_b^{x_1} \sqrt{1-4Aa^2 \sin^2 ax_1}\,dx_1
$$
for some constant $b$, and only in this case, $M^3$ can be locally isometrically  
embedded into $\mathbb{R}^4$ provided $A>0$.
\end{exam}

\bigskip

\subsection{$3$-dimensional Lie groups}
In our previous paper \cite{AH}, based on the method by Kaneda \cite{K1}, we classified 
left-invariant Riemannian metrics on $3$-dimensional Lie groups that can be locally 
isometrically embedded into $\mathbb{R}^4$.
By applying the main theorem (Theorem~\ref{MainTh}), we can obtain the same results 
through a little different approach.
In the following we use the same symbols as in \cite{AH}.

Let $G$ be a $3$-dimensional Lie group with a left-invariant Riemannian metric $g$.
We denote by $\mathfrak{g}$ the Lie algebra of $G$.
Naturally, there is the inner product of $\mathfrak{g}$ induced by the metric $g$.
First of all, we consider the case where $\mathfrak{g}$ is solvable.
In this case it is already known that there exists an orthonormal basis $\{ e_1,e_2,e_3 \}$ 
of $\mathfrak{g}$ such that the bracket operations are given by 
$$
[e_1,e_2] = ae_2+2be_3, \quad [e_1,e_3] = 2ce_2+de_3, \quad [e_2,e_3] = 0  
$$
for some real numbers $a,b,c,d$ (cf. \cite{AH}, \cite{HT}, \cite{HTT}, \cite{KTT}).
In this case we have the following proposition.

\begin{prop}\label{Appl4}
Under the above notations, Rivertz's six equalities $r_1= \cdots = r_6=0$ are equivalent to 
the condition 
$$
ad=4b^2, \;\; b=c \quad {\rm or} \quad a=d, \;\; b+c=0.
$$
\end{prop}

\begin{proof}
In \cite{AH} we already calculated the Levi-Civita connection of $G$ and its curvature.
The Levi-Civita connection is given by 
\begin{align*}
& \nabla_{e_1} e_1 = 0, \quad \nabla_{e_1} e_2 = (b-c)e_3, \quad 
\nabla_{e_1} e_3= -(b-c)e_2, \\
& \nabla_{e_2}e_1 = -ae_2-(b+c)e_3, \quad \nabla_{e_2} e_2 = ae_1, \quad 
\nabla_{e_2} e_3 = (b+c)e_1, \\
& \nabla_{e_3} e_1 = -(b+c)e_2-de_3, \quad \nabla_{e_3} e_2 = (b+c)e_1, \quad
\nabla_{e_3} e_3 = de_1, 
\end{align*}
where we consider $e_i$ ($i=1,2,3$) as left-invariant vector fields on $G$.
The curvature tensor is given by 
\begin{align*}
& R_{1212} = -a^2-(b+c)(3b-c), \\
& R_{1213} = -2(ac+bd), \\
& R_{1223} = R_{1323} = 0, \\
& R_{1313} =(b+c)(b-3c)-d^2, \\
& R_{2323} = (b+c)^2-ad.
\end{align*}
The covariant derivative $S = \nabla R$ can be calculated by the formula 
\begin{align*}
S_{ijklm} & = (\nabla_{e_m}R)(e_i,e_j,e_k,e_l) \\
& =e_m(R(e_i,e_j,e_k,e_l)) -R(\nabla_{e_m}e_i,e_j,e_k,e_l)-R(e_i,\nabla_{e_m}e_j,e_k,e_l) \\
& \hspace{0.5cm} - R(e_i,e_j,\nabla_{e_m}e_k,e_l)-R(e_i,e_j,e_k,\nabla_{e_m}e_l).
\end{align*}
Note that $e_i$ etc. are left-invariant vector fields on $G$, and hence $R(e_i,e_j,e_k,e_l)$ 
is constant on $G$, and the first term $e_m(R(e_i,e_j,e_k,e_l))$ vanishes.
By direct calculations, we have 
\begin{align*}
& S_{12121} = -2(b-c)R_{1213}, \\
& S_{12122}=S_{12123} = 0, \\
& S_{12131} = (b-c)(R_{1212}-R_{1313}), \\
& S_{12132} = S_{12133} = S_{12231} = 0, \\
& S_{12232} = (b+c)R_{1212}-aR_{1213}-(b+c)R_{2323}, \\
& S_{12233} = dR_{1212}-(b+c)R_{1213}-dR_{2323}, \\
& S_{13131} = 2(b-c)R_{1213}, \\
& S_{13132} = S_{13133} = S_{13231} = 0, \\
& S_{13232} = (b+c)R_{1213}-aR_{1313}+aR_{2323}, \\
& S_{13233} = dR_{1213}-(b+c)R_{1313}+(b+c)R_{2323}, \\
& S_{23231} = S_{23232} = S_{23233} = 0.
\end{align*}

Then, concerning Rivertz's six equalities,  we have $r_3 = r_5 =0$ at once since many 
components are zero.
Now, we consider the remaining four conditions.
We assume $r_1=r_2=r_4=r_6=0$.
At first, we have 
\begin{align*}
r_6 & = R_{1313}R_{2323}S_{12232}-R_{1213}R_{2323}S_{12233}
-R_{1213}R_{2323}S_{13232}+R_{1212}R_{2323}S_{13233} \\
& = \{ (b+c)(R_{1212}-R_{1313})-(a-d)R_{1213} \}R_{2323}{}^2 \\
& = -(b-c) \{(a-d)^2+4(b+c)^2 \} \{(b+c)^2-ad \}^2.
\end{align*}
Next, we can directly examine the equality 
\begin{align}
& R_{1212}R_{1313}-R_{1213}{}^2 + (b-c)^2(a+d)^2 \label{appleq5} \\
= & \{ ad-(b+c)^2 \} \{ad+(3b-c)(b-3c) \}. \notag
\end{align}
From the equality $r_6=0$, we have the following three possibilities.

(i) 
In case $b=c$, we have $S_{12121} = S_{12131}  = S_{13131} = 0$.
Hence we have 
\begin{align}
& r_1 = -(R_{1212}R_{1313}-R_{1213}{}^2)S_{12232}, \notag \\
& r_2 = -(R_{1212}R_{1313}-R_{1213}{}^2)(S_{12233}+S_{13232}),  \label{appleq6} \\
& r_4 = -(R_{1212}R_{1313}-R_{1213}{}^2)S_{13233}. \notag
\end{align}
In this case, from the equality (\ref{appleq5}), we have 
$$
R_{1212}R_{1313}-R_{1213}{}^2 = (ad-4b^2)^2.
$$
Hence, by calculating $S_{12232}$ etc., we have 
\begin{align*}
& r_1 = -4b(ad-4b^2)^3, \\
& r_2 = 2(a-d)(ad-4b^2)^3, \\
& r_4 = 4b(ad-4b^2)^3.
\end{align*}
Thus, we have $ad=4b^2$ or $a=d$, $b=c=0$ from the conditions $r_1=r_2=r_4=0$.

(ii) 
In case $a=d$ and $b+c=0$, we have 
\begin{align*}
& R_{1212} = R_{1313} = R_{2323} = -a^2, \\
& R_{1213} = R_{1223} = R_{1323} = 0, 
\end{align*}
and $S_{ijklm} = 0$ for all indices.
Hence the equalities $r_1=r_2=r_4=0$ automatically hold.
Note that the case $a=d$, $b=c=0$ in (i) is contained in this case.

(iii) 
In case $(b+c)^2=ad$, we have $R_{2323} = 0$.
Hence we have formally the same expressions as in (\ref{appleq6}), concerning $r_1, r_2, r_4$.
In these expressions, by replacing the product $ad$ by $(b+c)^2$, we have 
\begin{align*}
& r_1 = (b-c) \{ a^2+(b+c)^2 \}(R_{1212}R_{1313}-R_{1213}{}^2), \\
& r_2 = 2(a+d)(b^2-c^2) (R_{1212}R_{1313}-R_{1213}{}^2), \\
& r_4 = (b-c) \{(b+c)^2+d^2 \} (R_{1212}R_{1313}-R_{1213}{}^2).
\end{align*}
If $R_{1212}R_{1313}-R_{1213}{}^2 \neq 0$, we have $b=c$, or 
$a=d=b+c=0$.
But these cases are already considered in the above.
In case $R_{1212}R_{1313}-R_{1213}{}^2 = 0$, we have from (\ref{appleq5}) 
$b=c$ or $a+d=0$.
The former case is already treated, and in the latter case, we have $(b+c)^2=ad=-a^2$, 
i.e., $a=d=b+c=0$.
This is also treated in the above.
Thus, in any case, we have 
$$
ad=4b^2, \;\; b=c \quad {\rm or} \quad a=d, \;\; b+c=0.
$$

The converse part can be easily verified by a similar argument as above.
\end{proof}

We remark that we here use left-invariant vector fields $e_i$ ($i=1,2,3$) to calculate $R_{ijkl}$ 
and $S_{ijklm}$, instead of using vector fields $X_i= \frac{\partial \;\;}{\partial x_i}$ defined by 
local coordinates, which we did in \S 5.1$\sim$\S 5.3.
But there is no problem since $R_{ijkl}$ and $S_{ijklm}$ are tensors, and their values are 
determined pointwisely for each fixed basis of the tangent space.

\bigskip

Now, we apply Proposition~\ref{Appl4} to left-invariant metrics on 
$3$-dimensional solvable Lie groups.
In the following, we use the same symbols as in \cite{AH} in expressing Lie algebras.

At first, concerning two Lie algebras $\mathfrak{h_3}$ and $\mathfrak{r}_{3,1}$, 
left-invariant metrics are essentially uniquely determined with bracket operations 
\begin{align*}
\mathfrak{h}_3 & : \; [e_1,e_2] = e_3, \\
\mathfrak{r}_{3,1} & : \; [e_1,e_2] = e_2, \;\; [e_1,e_3] = e_3.
\end{align*}
Here, $\{ e_1, e_2, e_3 \}$ is an orthonormal basis of the Lie algebra as before.
It is easy to see that the inequality $|\widetilde{R}| < 0$ holds for both 
cases (for details, 
see \cite{AH}, p.203$\sim$204).
Hence they cannot be locally isometrically embedded into $\mathbb{R}^4$, and we may 
omit these two Lie algebras in the following discussion.

Remaining $3$-dimensional solvable Lie algebras and their non-trivial bracket operations 
are given by 
\begin{align*}
& \mathbb{R}^3 \;\;\; : \\
& \mathfrak{r}_3 \;\;\;\;\, : \;\; [e_1,e_2] = e_2+2\lambda e_3, \;\; [e_1,e_3] = e_3 
\;\; (\lambda > 0) \\
& \mathfrak{r}_{3,\alpha} \;\; : \;\; [e_1,e_2] = e_2+2\lambda(\alpha-1)e_3, \;\; [e_1,e_3] 
= \alpha e_3 \;\; (-1 \leq \alpha < 1, \; \lambda \in \mathbb{R}) \\
& \mathfrak{r}_{3,\alpha}^{\prime} \;\; : \;\; [e_1,e_2] = \alpha e_2 -\lambda e_3, \;\; 
[e_1,e_3] = \frac{1}{\lambda}e_2+\alpha e_3 \;\; (\alpha \geq 0, \; \lambda \geq 1).
\end{align*}
The parameter $\alpha$ indicates the isomorphism class of Lie algebras and $\lambda$ is 
the parameter of the moduli space of left-invariant metrics on each Lie algebra.

Then, it is easy to see that left-invariant metrics satisfying the condition in 
Proposition~\ref{Appl4} are exhausted by 
$$
\mathbb{R}^3, \quad \mathfrak{r}_{3,0} \; (\lambda=0), \quad 
\mathfrak{r}_{3,\alpha}^{\prime} \; (\lambda=1).
$$
It is clear that $\mathbb{R}^3$ with any left-invariant metric is flat and can be 
locally isometrically embedded into $\mathbb{R}^3$.
The space corresponding to $\mathfrak{r}_{3,0}$ with $\lambda=0$ is the direct product of 
the hyperbolic plane $\mathbb{R}{\rm H}^2$ and $\mathbb{R}^1$.
Hence it can be locally isometrically embedded into $\mathbb{R}^3 \times \mathbb{R}^1 
= \mathbb{R}^4$.
As for the Lie algebra $\mathfrak{r}_{3,\alpha}^{\prime}$ with $\lambda=1$, we have 
$$
|\widetilde{R}| = \begin{vmatrix}
-\alpha^2 & 0 & 0 \\
0 & -\alpha^2 & 0 \\
0 & 0 & -\alpha^2
\end{vmatrix} = -\alpha^6.
$$
Hence, from the condition $|\widetilde{R}| \geq 0$, we have $\alpha=0$, and in this 
case the space is flat.
Therefore, it can be locally isometrically embedded into $\mathbb{R}^3$.
In this way, we can classify the left-invariant Riemannian metrics on 
$3$-dimensional solvable Lie groups again (see Theorem 1.1 of \cite{AH}).

In the above alternative proof, before considering the Gauss equation, we first examine Rivertz's 
equalities.
By this method we can speedily sweep out many left-invariant metrics that cannot be realized 
in $\mathbb{R}^4$.
It is the merit of this approach.

\bigskip

As for $3$-dimensional simple Lie groups, we have already used Rivertz's equalities in 
\cite{AH} in some different way.
We here explain it briefly, by using the same notations as in \cite{AH}.

For simple Lie algebras it is known that there exists an orthonormal basis 
$\{ e_1, e_2, e_3 \}$ such that the bracket operations are given by 
$$
[e_1,e_2] = \lambda_3e_3, \;\; [e_2,e_3] = e_1, \;\; [e_3,e_1] = \lambda_2e_2
$$
for some constants $\lambda_2 >0$, $\lambda_3 \neq 0$.
Then after some calculations, we know that 
$$
R_{1213} = R_{1223} = R_{1323} = 0
$$
and the components of $S_{ijklm}$ vanish except $S_{12131}, S_{12232}, S_{13233}$.
From these facts, we have immediately $r_2=r_3=r_5=0$, and the remaining three equalities 
are expressed as 
\begin{align*}
& r_1 = -R_{1212}(R_{2323}S_{12131}+R_{1313}S_{12232})=0, \\
& r_4 = R_{1313}(R_{2323}S_{12131}-R_{1212}S_{13233})=0, \\
& r_6 = R_{2323}(R_{1313}S_{12232}+R_{1212}S_{13233})=0.
\end{align*}
We can prove $R_{1212}, R_{1313}$, $R_{2323} \neq 0$ by using the assumption that the Gauss 
equation admits a solution and $\lambda_2>0$, $\lambda_3 \neq 0$ (cf. 210, \cite{AH}).
Thus we have the equalities 
\begin{align}
& R_{2323}S_{12131}= -R_{1313}S_{12232}=R_{1212}S_{13233}, \label{appleq7}
\end{align}
which are nothing but the equalities in Lemma 5.5 of \cite{AH}.

The equalities (\ref{appleq7}) are algebraic equations on the parameters $\lambda_2$, 
$\lambda_3$ with degree five.
In \cite{AH}, by solving these algebraic equations, we showed that the left-invariant metric that 
can be locally isometrically embedded into $\mathbb{R}^4$ is exhausted by the standard 
metric on $\mathfrak{so}(3)$ up to isometry and scaling.
In this way, we complete the classification.
We may say that we have essentially used Rivertz's six equalities already in \cite{AH}.

\medskip

Finally, we remark that in the case $|\widetilde{R}| = 0$, we cannot in general determine 
whether the Gauss equation admits a solution or not.
Actually, there occur both cases.
If $M$ is flat, then we have clearly $|\widetilde{R}| = 0$, and it is locally isometrically 
embedded into $\mathbb{R}^3 \subset \mathbb{R}^4$.
The product space $\mathbb{R}{\rm H}^2 \times \mathbb{R}^1$ appeared above is also 
included in this case.

On the contrary, the Lie algebra whose bracket operations are given by 
$$
[e_1,e_2] = 3e_3, \;\; [e_2,e_3] = 4e_1, \;\; [e_3,e_1] = 3e_2
$$
in terms of an orthonormal basis $\{ e_1,e_2,e_3 \}$ gives an opposite example.
In fact, its curvature is given by 
$$
R_{1212} = R_{1313} = 4, \quad R_{ijkl}=0 \:\: ({\rm otherwise}).
$$
Hence we have $|\widetilde{R}| = 0$, and it can be easily verified that the Gauss equation 
does not admit a solution in this case.
(Note that this Lie algebra is isomorphic to $\mathfrak{so}(3)$ as a Lie algebra.)
 
\bigskip

Since any left-invariant Riemannian metric is real analytic, it follows that 
any $3$-dimensional Lie group with left-invariant metric can be always locally isometrically 
embedded into $\mathbb{R}^6$ from classical Janet-Cartan's theorem.
Thus, our next theme should be to classify left-invariant metrics that can be locally 
isometrically embedded into $\mathbb{R}^5$.
But to attack this problem, we must find a new obstruction to show the non-existence of 
local isometric embeddings into $\mathbb{R}^5$ on the one hand, and on the other hand, 
we must find new isometric embeddings of left-invariant metrics.

Here we give one such example.
We consider the following $3$-dimensional Lie group, depending on the parameter 
$\alpha \in \mathbb{R}$:
$$
G_{\alpha} = \left\{ \left. \begin{pmatrix}
1 & y & z \\ 0 & x & 0 \\ 0 & 0 & x^{\alpha} \end{pmatrix} \right| x>0 \right\}.
$$
The parameters $x,y,z$ give a coordinate of $G_{\alpha}$.
Its Lie algebra is generated by the basis 
$$
e_1= \begin{pmatrix} 0 & 0 & 0 \\ 0 & -1 & 0 \\ 0 & 0 & -\alpha \end{pmatrix}, \;\;
e_2= \begin{pmatrix} 0 & 1 & 0 \\ 0 & 0 & 0 \\ 0 & 0 & 0 \end{pmatrix}, \;\;
e_3= \begin{pmatrix} 0 & 0 & 1 \\ 0 & 0 & 0 \\ 0 & 0 & 0 \end{pmatrix},
$$
whose bracket operations are given by 
$$
[e_1,e_2] = e_2, \;\; [e_1,e_3] = \alpha e_3.
$$
Hence it is isomorphic to the Lie algebra $\mathfrak{r}_{3,\alpha}$, when $-1 \leq \alpha \leq 1$.
We introduce the left-invariant metric on $G_{\alpha}$ such that the above $\{ e_1,e_2,e_3 \}$ 
becomes orthonormal.

We denote the left-invariant vector fields corresponding to $e_i$ by the same letter $e_i$.
Then they are explicitly given by 
$$
e_1=\displaystyle{ -x \frac{\partial \;\;}{\partial x}-y \frac{\partial \;\;}{\partial y}
-\alpha z \frac{\partial \;\;}{\partial z}}, \;\;
e_2 = \displaystyle{\frac{\partial \;\;}{\partial y}}, \;\;
e_3 =  \displaystyle{\frac{\partial \;\;}{\partial z}}.
$$
Thus the left-invariant metric is expressed as 
$$
ds^2=\frac{1}{x^2} \left\{dx^2+(xdy-ydx)^2+(xdz-\alpha zdx)^2 \right\}.
$$

Now we consider the mapping $\varphi$ from $G_{\alpha}$ to $\mathbb{R}^5$ expressed 
in the form
$$
\varphi(x,y,z) = \left(f(x), \,kx \cos \frac{y}{kx},\,kx \sin \frac{y}{kx}, \,x^{\alpha}\cos 
\frac{z}{x^{\alpha}}, \,x^{\alpha} \sin \frac{z}{x^{\alpha}} \right) \label{appleq8}
$$
for some positive constant $k$.
Then it is easy to see that $\varphi$ gives an isometric immersion if and only if the equality 
\begin{align}
f^{\prime}{}^2 = \frac{1-k^2x^2-\alpha^2 x^{2\alpha}}{x^2} \label{appleq8}
\end{align}
holds.
In case $\alpha \geq 0$, we can easily find a function $f(x)$ defined in some open interval, 
satisfying (\ref{appleq8}).
In case $\alpha < 0$, by choosing a positive constant $k$ suitably, we can also show that 
the numerator $1-k^2x^2-\alpha^2x^{2\alpha}$ takes a positive value at some point $x > 0$.
(The constant $k$ necessarily depends on the parameter $\alpha$.)
Thus we can find a function $f(x)$ which satisfies the equality (\ref{appleq8}) in some open 
interval, and hence the Lie group $G_{\alpha}$ can be locally isometrically embedded into 
$\mathbb{R}^5$ in any case.
Combining with our previous result, we know that this map gives the least dimensional local 
isometric embedding of $G_{\alpha}$ into the Euclidean spaces.

Concerning the local isometric embedding problem for the $3$-dimensional 
Riemannian manifolds into $\mathbb{R}^5$, there is Kaneda's deep investigation now 
in progress (cf. \cite{K2}).

%\bmhead{Acknowledgments}

%\section*{Declarations}
%{\bf Conflict of interest} On behalf of all authors, the corresponding author states that there is no %conflict of interest.

\end{document}